\documentclass[reqno]{amsart}
\usepackage{diagrams}
\usepackage{amssymb}
\usepackage{mathrsfs}
\usepackage{cite}
\usepackage{MnSymbol}
\usepackage{a4wide}
 
\usepackage{xcolor}

\DeclareMathOperator\C{\mathbb C}

\DeclareMathOperator\Z{\mathbb Z}
\newcommand{\Om}{\Omega}
\newcommand{\ber}{\operatorname{ber}}
\newcommand{\G}{{\mathbb G}}
\newcommand{\sSch}{\operatorname{sSch}}

\newtheorem{theorem}{Theorem}[section]
\newtheorem{lemma}[theorem]{Lemma}
\newtheorem{cor}[theorem]{Corollary}
\newtheorem{prop}[theorem]{Proposition}
\theoremstyle{definition}
\newtheorem{definition}[theorem]{Definition}
\newtheorem{example}[theorem]{Example}

\theoremstyle{remark}
\newtheorem{remark}[theorem]{Remark}

\newcommand{\dontprint}[1]\relax

\newcommand{\osp}{\operatorname{osp}}

\newcommand{\bos}{{\operatorname{bos}}}
\newcommand{\Proj}{\operatorname{Proj}}
\newcommand{\Conf}{\operatorname{Conf}}

\renewcommand{\P}{{\mathbb P}}
\newcommand{\A}{{\mathbb A}}
\newcommand{\wt}{\widetilde}
\newcommand{\ot}{\otimes}

\newcommand{\Hom}{\operatorname{Hom}}
\newcommand{\Isom}{\operatorname{Isom}}
\newcommand{\Aut}{\operatorname{Aut}}
\newcommand{\Ext}{\operatorname{Ext}}
\renewcommand{\AA}{{\mathcal A}}
\newcommand{\BB}{{\mathcal B}}
\newcommand{\TT}{{\mathcal T}}
\newcommand{\DD}{{\mathcal D}}
\newcommand{\CC}{{\mathcal C}}
\newcommand{\EE}{{\mathcal E}}
\newcommand{\FF}{{\mathcal F}}
\newcommand{\GG}{{\mathcal G}}
\newcommand{\II}{{\mathcal I}}
\newcommand{\KK}{{\mathcal K}}
\newcommand{\LL}{{\mathcal L}}
\newcommand{\MM}{{\mathcal M}}
\newcommand{\NN}{{\mathcal N}}
\newcommand{\HH}{{\mathcal H}}
\newcommand{\OO}{{\mathcal O}}
\newcommand{\UU}{{\mathcal U}}
\newcommand{\PP}{{\mathcal P}}
\renewcommand{\SS}{{\mathcal S}}

\newcommand{\VV}{{\mathcal V}}

\newcommand{\XX}{{\mathcal X}}

\newcommand{\si}{\sigma}
\newcommand{\de}{\delta}
\newcommand{\sub}{\subset}

\newcommand{\Spec}{\operatorname{Spec}}
\newcommand{\Spf}{\operatorname{Spf}}
\newcommand{\Res}{\operatorname{Res}}
\newcommand{\Ann}{\operatorname{Ann}}
\newcommand{\Det}{\operatorname{Det}}
\newcommand{\Ber}{\operatorname{Ber}}

\newcommand{\lan}{\langle}
\newcommand{\ran}{\rangle}
\newcommand{\ov}{\overline}
\newcommand{\im}{\operatorname{im}}
\newcommand{\om}{\omega}

\renewcommand{\a}{\alpha}
\renewcommand{\b}{\beta}
\newcommand{\ga}{\gamma}

\newcommand{\id}{\operatorname{id}}
\newcommand{\und}{\underline}
\renewcommand{\th}{\theta}
\newcommand{\coker}{\operatorname{coker}}
\newcommand{\red}{{\operatorname{red}}}
\newcommand{\hra}{\hookrightarrow}

\newcommand{\De}{\Delta}

\newcommand{\fm}{{\frak m}}
\newcommand{\fg}{{\frak g}}

\newcommand{\eps}{\epsilon}
\newcommand{\bb}{{\mathbf b}}
\newcommand{\pa}{\partial}
\newcommand{\Art}{\operatorname{Art}}
\newcommand{\Def}{\operatorname{Def}}
\newcommand{\Sh}{\operatorname{Sh}}
\renewcommand{\div}{\operatorname{div}}
\newcommand{\tot}{\operatorname{tot}}
\newcommand{\reg}{\operatorname{reg}}
\newcommand{\Th}{\Theta}

\numberwithin{equation}{section}

\title{The moduli space of stable supercurves and its canonical line bundle}

\author{Giovanni Felder}
\address{Department of mathematics,
ETH Zurich, 8092 Zurich, Switzerland}
\email{felder@math.ethz.ch}
\author{David Kazhdan}
\address{Einstein Institute of Mathematics,
The Hebrew University of Jerusalem,
Jerusalem 91904, Israel}
\email{kazhdan@math.huji.ac.il}
\author{Alexander Polishchuk}
\address{
    Department of Mathematics, 
    University of Oregon, 
    Eugene, OR 97403, USA; National Research University Higher School of Economics, Moscow, Russia
  }
  \email{apolish@uoregon.edu}

\begin{document}

\begin{abstract}
We prove that the moduli of stable supercurves with punctures is a smooth proper DM stack and study an analog of the Mumford's isomorphism for its canonical line bundle. 
\end{abstract}
%The analog of Artin's criterion states that ???

%{\color{red}
%I would replace many of your Note by \begin{definition}\end{definition}. Also would make more clear the destintion between places when you consider families and/or individual supercurves. Would add a short description of a content in the beginning of each section}

\maketitle

\section{Introduction}

The moduli space $\SS_g$ of supercurves of genus $g$ (aka super Riemann surfaces, aka SUSY curves) has been around in mathematics and physics
since the 80's. It plays an important role in superstring theory and has been studied using the language of algebraic geometry in \cite{LR}, \cite{CV}, \cite{BHR}, \cite{DW} and
other works.
One long standing gap in the mathematical side of the story has been the study of the analog of Deligne-Mumford compactification by stable supercurves 
(which is a proper Deligne-Mumford superstack).
One of the goals of the present paper is to contribute to filling this gap. This compactification seems to be necessary for continuing the study of superstring supermeasure (see 
\cite{Witten}, \cite{FKP}).

The definition of a stable supercurve and a sketch of a proof that this gives a smooth and proper DM-stack is contained Deligne's letter
to Manin \cite{Deligne} (supernodes were also introduced independently in physics literature, see \cite{Cohn}). The main part of Deligne's letter is
devoted to the infinitesimal part of the theory. In particular, he describes miniversal deformations of two types fo supernode singularities (see 
Sec.\ \ref{nodes-def-sec}). The idea of the rest of the proof is to use the superanalog of the Artin's criterion for proving algebraicity of a stack.

In the present paper we revisit Deligne's letter and generalize its results to the case of stable supercurves with punctures. 
One of the special features of the super-case is that instead of ``marked points" we consider supercurves with {\it Neveu-Schwarz (NS) and Ramond punctures}.
Of these, NS punctures are like marked points: for a family of (stable) supercurves $X\to S$ they are given by sections $S\to X$. However, Ramond punctures have a different
nature: they are given by relative Cartier divisors $R\sub X$ such that the projection $R\to S$ is smooth of dimension $0|1$, and also the definition of being a supercurve
is modified near $R$ (see Sec.\ \ref{def-stable-supercurve-sec} for details). Note that the nodes of stable supercurves can also be of two types, NS and Ramond.

\bigskip

\noindent
{\bf Theorem A}. {\it There exists a smooth and proper DM-stack $\ov{\SS}_{g,n_{NS},n_R}$ over $\C$ representing the functor of families of stable supercurves of genus $g$ with
$n_{NS}$ NS punctures and $n_R$ Ramond punctures.}

\medskip

After we finished this work we learned about the work of Moosavian and Zhou \cite{MZ}, where Theorem A is also proved.
Their work also contains a lot of useful foundational results in algebraic supergeometry, e.g., proves the existence of Hilbert superschemes.
Another work with some foundational results, including the existence of Hilbert and Picard superschemes, is the forthcoming paper \cite{Hilbert-schemes}.

%We fill some details in the part of Deligne's letter  and prove that the stack $\ov{\SS}_g$ of stable supercurves defined by Deligne is a smooth and proper DM-stack. 

We mostly follow \cite{Deligne} in the part concerning deformation theory (extending it to the case of stable supercurves with punctures). 
In the rest of the proof we avoid using the Artin's criterion and give a more direct proof based on the existing solution of the corresponding purely even moduli problem 
due to Cornalba and Jarvis. Namely, the functor of families of stable supercurves
restricted to even bases is precisely the functor of generalized spin structures considered in \cite{Cornalba}, \cite{Jarvis}. 
%Knowing that this functor is represented by a proper
%DM-stack simplifies the study of stable supercurves considerably.

In the course of proof of Theorem A we find a natural relatively ample line bundle on any family of stable supercurves with punctures (see Theorem \ref{ample-thm}).
Although this is not strictly necessary for Theorem A (where one could use an algebraic space to replace the Hilbert scheme), this can be viewed as a super-analog 
of the well known fact that the relative log canonical bundle is relatively ample for a family of stable curves.
The important difference is that in the case of a family of stable supercurves $X\to S$ the relative canonical bundle on the smooth part does not extend to a line bundle on the entire
family (the extension is not locally free at NS nodes). However, we prove that its square does extend to a line bundle which we still denote as $\om_{X/S}^2$.
% (one needs a technical assumption that the corresponding map from $S$ to the moduli stack is smooth).
This line bundle plays an important technical role in the rest of the paper. In the hindsight, this is not too surprizing since the restriction of $\om_{X/S}^2$ to the reduced part
of moduli of supercurves recovers the relative dualizing sheaf of the corresponding usual family of stable curves.

One point which is not highlighted in \cite{Deligne} is that there is a canonical effective Cartier divisor $\De$ supported on the boundary of the compactification $\ov{\SS}_g$
(in the presence of odd variables, which are nilpotent, such a structure is not unique).
Essentially this structure is already seen from the study of deformations of the nodes of supercurves which can be of two types, NS and Ramond.
We also define a decomposition $\De=\De_{NS}+\De_R$ corresponding to these two types of nodes
and find two global expressions for the line bundle $\OO(\De)$ using Berezinians of certain natural morphisms of sheaves associated with the universal curve
(see Sec.\ \ref{boundary-str-sec}).

The second goal of this work is to study the canonical line bundle $K_{\ov{\SS}_g}$
over $\ov{\SS}_g$, i.e., the Berezinian of the cotangent vector bundle. 
The super analog of Mumford isomorphism for smooth supercurves,
expressing the canonical line bundle of $\SS_g$ in terms of natural Berezinian line bundles was considered in
\cite{Voronov}, \cite{RSV} (see also \cite{Diroff} where some work is done in the punctured case).
Namely, it states that
$$K_{\SS_g}\simeq\Ber_1^5,$$
where
$$\Ber_1:= \Ber(R\pi_*\OO_X)\simeq \Ber(R\pi_*\om_{X/\ov{\SS}_g}),$$
where $\pi:X\to \SS_g$ is the
universal stable supercurve.

%Over $\SS_g$, combining the Kodaira-Spencer map
%and the super-analog of the Mumford isomorphism, one gets an isomorphism of this line bundle with $\Ber_1^5$. 
In the case of the moduli of stable supercurves (and in the presence of punctures)
we still have the line bundle $\Ber_1$ defined as above, however
the expression for the canonical bundle has to be corrected. 

\bigskip

\noindent
{\bf Theorem B}. {\it Let $\ov{\SS}= \ov{\SS}_{g,n_{NS},n_R}$.
There exists a canonical isomorphism
$$K_{\ov{\SS}}\rTo{SM_{\ov{\SS}}}
\Ber_1^5\otimes \bigotimes_{i=1}^{n_{NS}}\Psi_i(-2\De_{NS}-\De_R),$$
where $\De_{NS}$ and $\De_R$ are the components of the boundary divisor corresponding to nodes of NS and Ramond type,
and $\Psi_i$ is the line bundle associated with the $i$th NS puncture $P_i$:
$$\Psi_i:=P_i^*\om_{X/\ov{\SS}}.$$
%where $X/\ov{\SS}$ is the universal stable supercurve.
}

\medskip

In the case when there are no punctures the isomorphism becomes
$$K_{\ov{\SS}_g}\simeq \Ber_1^5(-2\De_{NS}-\De_R).$$

Finally, we study the restriction of the super-Mumford isomorphism $SM_{\ov{\SS}_g}$ to the NS boundary divisor, and show that it is related to the similar isomorphisms for lower genus. As in the classical case, each component of $\De_{NS}$ is an image of a natural gluing map $B\to \ov{\SS}_g$,
where $B$ is some lower genus moduli space with more punctures or a product of two such moduli spaces (see Sec.\ \ref{NS-boundary-gluing-sec}).

\bigskip

\noindent
{\bf Theorem C}. {\it Let $B\to \ov{\SS}_g$ be the gluing map to one of the components of $\De_{NS}$ from a lower genus (uncompactified) moduli space $B$.
There exists a natural isomorphism (see below) of the normal bundle
$$N_B:=\OO(\De)|_B\simeq \Psi_1^{-1}\ot \Psi_2^{-1},$$
such that the following diagram is commutative up to a sign
\begin{diagram}
K_{\ov{\SS}_g}(2\De)|_B&\rTo{SM_{\ov{\SS}_g}|_B}& \Ber_1^5|_B\\
\dTo{}&&\dTo{}\\
K_B\ot N_B&\rTo{SM_B}&(\Ber_1^B)^5
\end{diagram}
}

The identification of $N_B$ in Theorem C comes from the identification of $\De$ as the vanishing
locus of a morphism induced on Berezinians by the morphism
$$R\pi_*(\Om_{X/\ov{\SS}_g})\to R\pi_*j_*(\Om_{U/\ov{\SS}_g}),$$
where $j:U\to X$ is the embedding of the smooth locus (see Sec.\ \eqref{boundary-equation-sec}).

The components of the Ramond boundary divisor $\De_R$ have a more complicated relation to lower genus moduli spaces (see Sec.\
\ref{R-boundary-gluing-sec}), and in this case we formulate a conjectural analog of Theorem C in Sec.\ \ref{R-splitting-sec}.

The paper is organized as follows.
In Section \ref{basics-sec} we review the basics on stable supercurves starting from the definitions.
In particular, in \ref{local-descr-sec} we discuss standard local coordinates of smooth supercurves (near an ordinary point and near a Ramond
puncture). 
%Namely, we prove in Lemma \ref{local-descr-smooth-lem} that standard coordinates exist locally with respect to \'etale topology (all the previous
%constructions of such coordinates were proved for classical topology). 
Then in \ref{NS-correspondence-sec} we discuss a well-known correspondence
between the NS-punctures and divisors.

%In Section \ref{infinit-sec} we discuss some infinitesimal results on supercurves. We recall the construction from \cite{Deligne} of miniversal deformation families 
%of two types of node singularities of stable supercurves and consider sheaves of infinitesimal automorphisms of stable supercurves (possibly with punctures).

In Section \ref{def-sec} we study the deformation functor of a stable supercurve $X_0$ with punctures. 
The main result is that this functor is smooth (more precisely,
it is smooth over the product of deformation functors of the singularities of $X_0$). We start with infinitesimal results, reproducing Deligne's calculation 
of infinitesimal deformations of two types of super node singularities (see Theorem \ref{miniversal-Deligne-thm}) and identifying the sheaves of infinitesimal automorphisms of stable supercurves with punctures.
Then we study local deformations: of affine neighborhoods of a smooth point,
of a Ramond puncture, and of a singular point (see \ref{no-R-punct-def-sec}, \ref{R-punct-def-sec} and Lemma \ref{aff-formal-node-smooth-lem}). 
Finally we prove the smoothness result for global deformations, Proposition \ref{global-def-prop}.

In Section \ref{square-sec} we prove that an extension of the square of the relative canonical bundle $\om_{U/S}^2$ from the smooth locus $U\sub X$
of a sufficiently nice family $X\to S$ of stable supercurves is a line bundle, and that after some corrections at the punctures, it becomes relatively ample
(see Theorems \ref{can-square-thm} and \ref{ample-thm}).
We then show that the inverse of this line bundle shows up in the computation of the sheaf of infinitesimal automorphisms (see Theorem \ref{inf-aut-thm}).

In Section \ref{proof-thm} we prove Theorem A. Mostly we use the previous results on deformation theory and the known results on the moduli of
generalized spin curves.

In Section \ref{KS-map-sec} we study the behavior of the Kodaira-Spencer map for a family of smooth supercurves degenerating to a stable supercurve.
We consider the classical case of (even) curves in \ref{KS-class-sec}, then the case of an NS node in \ref{super-KS-NS-sec} and the case of a Ramond node
in \ref{super-KS-R-sec}. The main observation is that the Kodaira-Spencer map has a natural extension over the entire base involving
the subsheaf of the tangent space to the base consisting of vector fields preserving the degeneration divisor.

In Section \ref{boundary-str-sec} we study the boundary divisor $\De$ of the moduli of stable supercurves. We give a definition of the boundary divisor as a Cartier divisor
(which is not automatic since we work with nonreduced spaces). We compute the corresponding line bundle in terms of the complex $[\Om_{X/S}\to j_*\Om_{U/S}]$, where $U\sub X$ is the smooth locus of a family $X\to S$ (see \ref{boundary-equation-sec}). We prove
in \ref{normal-cross-sec} that $\De$ is a normal crossing divisor and define in \ref{De-NS-R-sec}
the subdivisors $\De_{NS}$ and $\De_R$ corresponding to the NS and Ramond type nodes such that $\De=\De_{NS}+\De_R$.
Then we discuss in \ref{NS-boundary-gluing-sec} and
\ref{R-boundary-gluing-sec} the gluing maps from lower genus moduli spaces to the boundary components.
Note that in the case of a Ramond node there is an extra odd parameter involved in the gluing.

In Section \ref{canonical-bundle-sec} we study the canonical line bundle over the moduli of stable supercurves with the goal of proving Theorem B.
The identification of the canonical line bundle for the moduli of smooth supercurves is a combination of the isomorphism coming from the Kodaira-Spencer map and of an analog of
the Mumford's isomorphism between different Berezinian line bundles. We have to investigate what happens with both these ingredients near the boundary divisor.
For the Kodaira-Spencer map this was done in Section \ref{KS-map-sec}.
The main new nontrivial computation is that of the behavior of the super Mumford's isomorphism as the supercurve degenerates (see \ref{super-Mum-NS-sec}).
Extending this picture to supercurves with punctures is relatively easy and is done in \ref{can-line-bun-punctures-sec}.

In Section \ref{Splitting-sec} we prove Theorem C. Again the bulk of the argument is the study of the restriction of super Mumford's isomorphism to the boundary divisor
in \ref{NS-splitting-sec}. Then in \ref{KS-splitting-sec} we do the same for the isomorphism coming from the Kodaira-Spencer map. In \ref{R-splitting-sec}
we describe a conjectural picture for the case of a Ramond boundary component.

In Appendix \ref{ample-sec} we prove a relative ampleness criterion for flat morphisms of superschemes (see Proposition \ref{ample-crit-prop}), which we
use to prove Theorem \ref{ample-thm}.

\subsection*{Conventions} 
All the rings are assumed to be $\Z_2$-graded supercommutative with $1$.
We say that a ring (resp., a superscheme) is {\it even} if its odd component (resp., the odd component of the structure sheaf) is zero.
We often say ``subscheme" for brevity where we should say ``sub-superscheme". 
For a sheaf $F$ of $\OO_X$-modules on a superscheme $X$ we denote by $F^+$ and $F^-$ its even and odd parts.
By a ``bundle" on a superscheme we mean a ($\Z_2$-graded) locally free $\OO_X$-module of finite rank. By a subbundle $F\sub E$ in a bundle
$E$ we mean a ($\Z_2$-graded) locally free $\OO_X$-submodule, which is locally a direct summand.
On any superscheme $X$ we denote by $\NN_X\sub \OO_X$ the ideal generated by odd functions.
We denote by $X_{\bos}$ the usual scheme with the same underlying topological space as $X$ and with the structure sheaf 
$\OO_X/\NN_X$. We work over $\C$. 
%(although many of our algebraic results hold over fields of characteristic $\neq 2$).

\subsection*{Acknowledgements}
We thank Ugo Bruzzo and Daniel Hern\'andez Ruip\'erez for useful discussions and
%sharing with us the manuscript of \cite{Hilbert-schemes} and
for telling us about the work \cite{MZ}. We thank the anonymous referee for many useful comments.
G.F. is partially supported by the National Centre of
Competence in Research ``SwissMAP --- The Mathematics of Physics'' of the
Swiss National Science Foundation. 
%He thanks the Hebrew University of Jerusalem, where part of this work was done, for hospitality.
D.K. is partially supported by the ERC under grant agreement 669655.
A.P. is partially supported by the NSF grant DMS-2001224, 
%by the National Center of Competence in Research ``SwissMAP --- The Mathematics of Physics'' of the Swiss National Science Foundation, 
and within the framework of the HSE University Basic Research Program and by the Russian Academic Excellence Project `5-100'.
%While working on this project, A.P. was visiting ETH Zurich and Hebrew University of Jerusalem. 
%He would like to thank these institutions for hospitality and excellent working conditions.

\section{Stable supercurves}\label{basics-sec}

In this section we discuss some basic facts about stable supercurves and their families, starting with definitions. In particular,
we discuss local descriptions of nodes and punctures, the correspondence between the NS punctures and divisors, and the
connection with generalized spin structures.

\subsection{Superstacks}

We refer to \cite{CCF} and \cite[Sec.\ 2]{BHR} for some basics on superschemes.
The notion of an algebraic (super)stack is developed similarly to the classical case (see \cite[Def.\ (5.1)]{Artin}):
one considers a category fibered in groupoids $\XX$ over the category of Noetherian superalgebras over $\C$,
which is a limit preserving stack (see \cite[(1.1)]{Artin}), such that the diagonal $\XX\to \XX\times \XX$ is
representable and there exists a smooth surjective morphism from a scheme to $\XX$ (see \cite[Sec.\ 3]{CV}).
If the latter morphism can be chosen to be \'etale, then one gets the notion of Deligne-Mumford (DM) stack.

For example, it is proved in \cite{CV} that the moduli stack of (smooth) supercurves is a smooth separated DM stack over $\C$.

\subsection{Definition of families of supercurves}\label{def-stable-supercurve-sec}

\begin{definition}
Let $S$ be a superscheme. A family of {\it smooth supercurves with Ramond punctures} is a smooth
superscheme $X$ over $S$ of relative dimension $1|1$, equipped with a subbundle $\DD\sub \TT_{X/S}$ of rank $0|1$,
such that the map given by the commutator of vector fields 
\begin{equation}\label{commutator-map}
\DD^{\ot 2}\to \TT_{X/S}/\DD
\end{equation}
(which is a map of line bundles of rank $1|0$) is injective and its divisor of vanishing is the disjoint union 
$$R=\sqcup_{i=1}^{n_R} R_i,$$ 
where each $R_i$ is smooth of dimension $0|1$ over $S$. The components $R_i$ are called {\it Ramond punctures}.
\end{definition}

A distribution $\DD$ as above is often referred to as a {\it superconformal structure} on $X$.
An isomorphism of superschemes $X\simeq X'$ as above is called {\it superconformal} if it preserves the superconformal structures,
i.e., sends the distribution on $X$ to the distribution on $X'$.

In the above situation the map $D_1\ot D_2\mapsto \frac{1}{2}[D_1,D_2]$ induces an isomorphism
\begin{equation}\label{Ramond-D2-isom}
\DD^{\ot 2}\rTo{\sim} \TT_{X/S}/\DD(-R).
\end{equation}
Thus, we have an exact sequence
$$0\to \DD\to \TT_{X/S}\to \DD^{\ot 2}(R)\to 0,$$
or dually,
$$0\to \DD^{-2}(-R)\to \Om_{X/S}\to \DD^{-1}\to 0.$$
From this we get an isomorphism $\om_{X/S}:=\Ber(\Om_{X/S})\simeq \DD^{-1}(-R)$, or equivalently,
\begin{equation}\label{Ramond-D-om-isom}
\DD\simeq \om_{X/S}^{-1}(-R).
\end{equation}
Hence, taking the dual of the embedding $\DD\to \TT_{X/S}$, we get a surjective morphism
$$\de:\Om_{X/S}\to \DD^{-1}\simeq \om_{X/S}(R),$$
whose kernel is exactly the orthogonal to $\DD\sub \TT_{X/S}$.
Equivalently, we can view $\de$ as a derivation $\OO_X\to \om_{X/S}(R)$, trivial on $\OO_S$.
Note that $\DD$ is recovered from $\de$ as the orthogonal to $\ker(\de)$.

\begin{example}\label{de-comp-ex}
Suppose $X$ is an open subscheme in $\A^{1|1}_S$ with relative coordinates $(z,\th)$ and
$\DD$ is generated by $D=\pa_\th+f\th\pa_z$, for some even function $f$ on $X$.
Then 
$$\frac{1}{2}[D,D]=f\pa_z,$$
so the condition that $\DD$ defines a structure of a supercurve on $X$ (with no punctures) 
is that $f$ is invertible. 
Thus, the canonical isomorphism 
$$\DD^{\ot 2}\to \TT_{X/S}/\DD$$ 
sends $D\ot D$ to $f\pa_z \mod\DD$.
Hence, the isomorphism
$$\DD\rTo{\sim}\Ber(\TT_{X/S})\simeq \om_{X/S}^{-1}$$
sends $D$ to $f[dz|d\th]^{-1}$, where $[dz|d\th]$ is a generator of $\om_{X/S}$ associated with the basis $(dz,d\th)$ of $\Om_{X/S}$.  
Thus, we can compute the dual of the embedding $\DD\to \TT_{X/S},$
$$\de:\Om_{X/S}\to \DD^{-1}\simeq \om_{X/S}.$$
Namely, $\de$ sends $dz$ to $\th[dz|d\th]$ and $d\th$ to $f^{-1}[dz|d\th]$. Thus, if we view $\de$ as a derivation,
we have
$$\de(\phi)=D(\phi)\cdot f^{-1}[dz|d\th].$$
\end{example}

\begin{remark}\label{deltaplus-rem}
Assume that $S$ is even. Then for a smooth supercurve $(X,\DD)$ over $S$, the bosonic truncation $C=X_{\bos}=(|X|,\OO_X^+)$ is a family of smooth curves over $S$,
and we have a natural projection $\pi:X\to C$. Furthermore, we can view $L=\OO_X^-$ as a line bundle on $C$. Since $\pi$ is a finite morphism, we have a natural identification of $\Z_2$-graded coherent sheaves on $C$,
$$\om_X=\pi^!\om_C=\und{\Hom}(\OO_X,\om_C)=\om_C\oplus \und{\Hom}(L,\om_C),$$
so the odd part of the derivation $\de$ induces an isomorphism $L\rTo{\sim} \und{\Hom}(L,\om_C)$, or equivalently $L^2\rTo{\sim} \om_C$. 
It is well known that under the above identification of $\om_X^+$ with $\om_C$ the even part of $\de$ is given by the de Rham differential $d:\OO_C\to \om_C$
(e.g., this can be deduced from the local standard structure of smooth supercurves, see \cite{Deligne} or Sec.\ \ref{local-descr-sec} below).
\end{remark}

\begin{definition}
(i) A {\it supercurve with punctures} over $S$, is a superscheme $X$ of finite type over $S$, flat and relatively Cohen-Macaulay, together with 
\begin{itemize}
\item
a collection of disjoint closed subschemes called {\it NS-punctures} and {\it R-punctures}
$$P_i\sub X, \ i=1,\ldots,n_{NS}; \ \ R_i\sub X, i=1,\ldots,n_R,$$
such that the projection $P_i\to S$ is an isomorphism, and each $R_i$ is a Cartier divisor (flat over $S$);
\item
a derivation
$$\de:\OO_X\to \om_{X/S}(R),$$
trivial on pull-backs of functions on $S$, where $R=\sum R_i$. Here $\om_{X/S}$ is the relative dualizing sheaf
so that $\pi^!\OO_S=\om_{X/S}[1]$ (the fact that it is a sheaf follows from the Cohen-Macaulay property).
\end{itemize}
We impose the following additional properties.
\begin{itemize}
\item There is an open fiberwise dense subset $U\sub X$ such that $U/S$ is smooth of dimension $1|1$,
and $P_i\sub U$, $R_i\sub U$;
\item the derivation $\de$ corresponds to a structure of a smooth supercurve with Ramond punctures on $(U,(R_i))$ over $S$;
\item on every geometric fiber $X_s$, $\de^-$ induces an isomorphism 
\begin{equation}\label{de-odd-parts-isom}
\OO_{X_s}^-\rTo{\sim}\om_{X_s}(\sum R_{is})^-;
\end{equation}
\end{itemize} 
By an {\it isomorphism} of supercurves with punctures over $S$ we mean an isomorphism of superschemes over $S$ compatible with all the structures.
We say that a relative vector field on $X$ (i.e., a derivation of $\OO_X$, trivial on $\pi^{-1}\OO_S$) is {\it superconformal} if it preserves $\de$. 

\noindent
(ii) A supercurve with punctures $(X/S,P_\bullet,R_\bullet)$ is called {\it stable (resp., prestable)} if $X$ is proper over $S$ and
for every geometric fiber $X_s$, passing to the bosonic truncation $(X_s)_{\bos}$ with the marked points induced by $(P_i)$ and $(R_i)$
one gets a usual stable (resp., prestable) pointed curve.
\end{definition}

The role of the isomorphism \eqref{de-odd-parts-isom} can be understood by looking at the case of even $S$ (see Lemma \ref{even-case-rem} below).
As observed in \cite{Deligne}, if $U\sub X$ is the maximal open subset which is smooth over $S$, then
$\de$ induces a structure of a smooth supercurve on $U\setminus \cup_i R_i$ over $S$. 

Since Cohen-Macaulay property appears in the definition of supercurves, we recall some results involving it that will be useful for us.

\begin{lemma}\label{CM-str-sh-lem}
(i) Let $X$ be a Cohen-Macaulay locally Noetherian superscheme, $j:U\to X$ an open embedding such that $Z=X\setminus U$ has codimension $\ge 2$ (resp., $\ge 1$) in $X$.
Then the natural map $\OO_X\to j_*\OO_U$ is an isomorphism (resp., injective). Similarly, if $\FF$ is a Cohen-Macaulay coherent sheaf on $X$ with full support, then the natural map
$\FF\to j_*j^*\FF$ is an isomorphism (resp., injective).

\noindent
(ii) Let $X\to S$ be flat, relatively Cohen-Macaulay, with fibers of pure dimension $1$, and let $j:U\hra X$ be a fiberwise dense open subscheme. Then the canonical map
$$\om_{X/S}\to j_*j^*\om_{X/S}$$
is injective.

\noindent
(iii) Let $X$ be a superscheme over a field $k$. Then $X$ is Cohen-Macaulay with a dense open which is smooth of dimension $1|1$ if and only $X_{\bos}$ is a reduced curve,
$\OO_X^-$ is a Cohen-Macaulay sheaf with full support on $X_{\bos}$, locally free of rank $1$ on a dense open, and $\OO_X^-\cdot \OO_X^-=0$.
\end{lemma}

\begin{proof}
(i) It is enough to check the vanishing of the corresponding local cohomology $\und{H}^i_Z(\OO_X)=\und{H}^i_Z(\FF)$ for $i=0,1$ (resp., for $i=0$). But this follows from \cite[Exp.\ VII,Cor.\ 1.4]{SGA2}.

\noindent
(ii) Otherwise, in some neighborhood of a point in $X$ the kernel of this morphism
would be a nonzero subsheaf of $\om_{X/S}$ with finite support over $S$. So we would get a nonzero morphism 
$$\OO_Z\to \om_{X/S}=\pi^!\OO_S[-1],$$
with $Z$ finite over $S$. By adjunction of the pair $(\pi_*,\pi^!)$, it would correspond to a nonzero morphism $\pi_*\OO_Z\to \OO_S[-1]$
which is impossible.

\noindent
(iii) This follows easily from the interpretation of the Cohen-Macaulay condition for a superring in terms of the even and odd components given in \cite[Lem.\ 7.5]{MZ}.
\end{proof}

Supercurves with punctures over even bases can be described in purely even terms as follows (cf.\ \cite[1.5]{Deligne}).

\begin{lemma}\label{even-case-rem}
Assume that $S$ is even. 
Then the data of a supercurve $(X,P_\bullet,R_\bullet)$ over $S$ is equivalent to the following data:
\begin{itemize}
\item a flat family of pointed curves $(C,p_\bullet,r_\bullet)$ over $S$ with reduced geometric fibers, smooth near the marked points;
\item a coherent sheaf $L$ on $C$, flat and relatively Cohen-Macaulay over $S$, locally free of rank $1$ over a fiberwise
dense open in $C$;
\item an isomorphism
\begin{equation}\label{gen-spin-str-eq}
L\rTo{\sim}\und{\Hom}(L,\om_C(\sum r_i)).
\end{equation}
\end{itemize}    
Under this correspondence one has $|C|=|X|$, 
$$\OO_X^+=\OO_C, \ \OO_X^-=L, \ \om_X^+=\om_C, \ \om_X^-=\und{\Hom}(L,\om_C)\simeq L(-\sum r_i),$$
where $\OO_X^-\cdot \OO_X^-=0$, $\OO_X^-\cdot \om_X^+=0$, and the action map $\OO_X^-\ot \om_X^-\to \om_X^+$ is identified with
the natural evaluation map.
The puncture $P_i$ is the image of the composition $S\rTo{p_i}C\to X$; and $R_i$ is the schematic preimage of $r_i$ under the projection
$X\to C$. The isomorphism \eqref{gen-spin-str-eq} corresponds to $\de^-$.
\end{lemma}

\begin{proof}
Let us consider the decomposition $\OO_X=\OO_X^+\oplus \OO_X^-$ into the even and odd components.
Then the bosonic truncation $C=X_{\bos}=(|X|,\OO_X^+)$ is a family of curves over $S$. Every geometric fiber $C_s$ is a Cohen-Macaulay
curve, with a dense smooth open. Equivalently, $C_s$ is a reduced curve. 
Furthermore, $L=\OO_X^-$ is a coherent sheaf on $C$ with the claimed properties, by Lemma \ref{CM-str-sh-lem}(iii). Conversely, using the same lemma,
starting with $(C,L)$ we get a superscheme over $S$ with required properties by setting $\OO_X=\OO_C\oplus L$ (where the product $L\cdot L$ is zero in $\OO_X$).

The marked points $P_i:S\to X$ factor through marked points $p_i:S\to C$. On the other hand, 
the Cartier divisors $R_i\hra X$ induce Cartier divisors $r_i\hra C$. By assumption, the latter are also marked points on $C$.
We claim that $R_i$ i exactly the preimage of $r_i$ under the canonical projection $\pi:X\to C$. Indeed, this immediately follows from the fact that
locally $R_i$ is given by an even equation $f\in \OO_X^+=\OO_C$, and $r_i$ is given by the same equation $f$ in $C$.

We have a natural identification of $\OO_X$-modules, compatible with the $\Z_2$-grading
$$\om_X=\pi^!\om_C=\und{\Hom}(\OO_X,\om_C).$$
Taking the odd parts, we get an isomorphism
$$\om^-_X\simeq \und{\Hom}(L,\om_C).$$
Thus, $\de^-$ induces an isomorphism \eqref{gen-spin-str-eq}. On the other hand, $\de^+:\OO_C\to \om_C(\sum r_i)$ is induced by the usual de Rham differential
(since this is so on a dense open subset, see Remark \eqref{deltaplus-rem}). Hence, we can recover $\de$ from the isomorphism \eqref{gen-spin-str-eq}.
\end{proof}

A coherent sheaf $L$ on a prestable curve $C/S$ equipped with an isomorphism \eqref{gen-spin-str-eq} is called a {\it generalized spin structure}.

\begin{lemma}\label{isom-stable-lem} 
Let $(X,P_\bullet,R_\bullet)$ and $(X',P'_\bullet,R'_\bullet$) be a pair of supercurves over $S$,
and let $\varphi:X\rTo{\sim} X'$ be an isomorphism of superschemes over $S$, sending $P_i$ to $P'_i$ and 
$R_j$ to $R'_j$. Assume that the restriction of $\varphi$ to a fiberwise dense open subscheme $U\sub X$, which is smooth over $S$, is an isomorphism of smooth
supercurves. Then $\varphi$ is an isomorphism of supercurves.
\end{lemma}

\begin{proof}
Let $j:U\to X$ be the inclusion.
By assumption, we have $\varphi^*\om_{X'/S}(R')\simeq \om_{X/S}(R)$.
Thus, we can view $\varphi^*\de_{X'}$ as a derivation $\OO_X\to \om_{X/S}(R)$, and we have
to prove the equality $\varphi^*\de_{X'}=\de_X$. Since we know that this equality holds over $U$,
the assertion follows from Lemma \ref{CM-str-sh-lem}(ii).
%This is clear on the smooth locus and follows from local descriptions near the nodes.
\end{proof}

We will later need the following result about infinitesimal deformations of supercurves. 

\begin{lemma}\label{inf-def-emb-lem} (\cite[Prop.\ 1.7]{Deligne})
Let $X_0$ be a supercurve with punctures over $\C$. Suppose an infinitesimal deformation $X$ of $X_0$ as a supercurve
over $\C[t]/(t^2)$ or over $\C[\tau]$, where $t$ is even and $\tau$ is odd, is trivial as a deformation of a superscheme. Then it is trivial as a deformation
of a supercurve.
\end{lemma}

\begin{proof} Using the correspondence of Lemma \ref{even-case-rem}, we get the data $(C,L,\varphi_0)$, where $\varphi_0$ is an isomorphism
$$\varphi_0:L\rTo{\sim}\und{\Hom}(L,\om_C),$$ 
so that $\OO_{X_0}=\OO_C\oplus L$. It is convenient to think of $\varphi_0$ as a map
$$\varphi_0:L^{\ot 2}\to \om_C.$$
  
 \medskip
 
\noindent
{\it The case of even base $\C[t]/(t^2)$}.
In this case we can still use Lemma \ref{even-case-rem} for $X$. Since the underlying superscheme is not deformed, this means that only the map $\varphi_0$
is deformed to $\varphi=\varphi_0+t\psi$ for some map $\psi:L\to \und{\Hom}(L,\om_C)$. Since $\varphi_0$ is an isomorphism, we can write $\psi=\varphi_0\circ a$,
where $a$ is an endomorphism of $L$. Let us set 
$$\a=\id+\frac{t}{2}a:L\rTo{\sim} L.$$
Then on the locus where $L$ is locally free we can think of $a$ as a function on $C$, so
$\a^{\ot 2}=\id+ta$, and
$$\varphi_0\circ \a^{\ot 2}=\varphi_0+t\varphi_0\circ a=\varphi,$$
which implies that this equality holds everywhere, so the data
$(C,L,\varphi_0)$ and $(C,L,\varphi)$ are isomorphic.
%=\varphi_0\circ \a$, where $\a$ is the infinitesimal automorphism of $L^{\ot 2$. Since $(\a-\id)^2=0$, $\a$ admits a square root $\a^{1/2}$. Now
%one can easily check that $\varphi$???

\medskip

\noindent
{\it The case of odd base $\C[\tau]$}.
We have to consider possible supercurve structures $\de:\OO_X\to \om_{X/S}$ on $X=X_0\times S$ over $S= \Spec(\C[\tau])$ reducing to $\de_0$ on $X_0$ over $\C$.
Extending $\de_0$ to $X$ by extension of scalars, we can write
$$\de=\de_0+\tau D,$$
for some odd derivation $D:\OO_{X_0}\to \om_{X_0}$. Let
$$D_0:\OO_C\to L, \ \ D_1:L\to \om_C$$ be the 
components of $D$. A simple local calculation shows that the condition for $\de$ to define a supercurve structure on the smooth part of $X$ implies that $D_1$ is determined by $D_0$.

We want to find an automorphism $\a$ of $X/S$, trivial on $X_0$, such that $\de$ is obtained from $\de_0$ by conjugation by $\a$. Note that $\a$ necessarily has form
$$\a(f)=f+\tau E(f),$$
where $E:\OO_{X_0}\to \OO_{X_0}$ is an odd derivation. We take $E$ with the components 
$$E_0=D_0:\OO_C\to L, \ \ E_1=0:L\to \OO_C.$$
An easy calculation (on a smooth locus) shows that $\a_{\om} \de_0 \a^{-1}=\de_0+\tau \wt{D}$ produces $\wt{D}$ with $\wt{D}_0=D_0$. Here $\a_{\om}:\om_{X/S}\to \om_{X/S}$
is the map induced by $\a$. Hence $\wt{D}=D$, as required. 
\end{proof}

\subsection{Dualizing sheaf on the formal completion}\label{formal-supercurve-sec}

We will sometimes want to argue ``formally locally", so it will be useful to look at the obtained structure on the formal completion of $\OO_{X,x}$ for a supercurve $X/S$.
For this we use the duality theory for formal schemes developed in \cite{AJL} (or rather, its superanalog). Namely, by the results of \cite[Sec.\ 2]{AJL} (extended to the supercase), 
we can identify the completion
$\hat{\om}_{X,x}$ with the dualizing sheaf for the morphism of formal superschemes $\hat{X}=\Spf(\hat{\OO}_{X,x})\to \hat{S}=\Spf(\hat{\OO}_{S,s})$, where $s\in S$ is the image of $x\in X$.
The structure derivation $\de:\OO_X\to \om_{X/S}$ induces a continuous derivation 
$$\hat{\de}:\OO_{\hat{X}}\to \om_{\hat{X}/\hat{S}},$$
which is compatible with the derivation $\hat{\de}_{X_s}:\OO_{\hat{X}_s,x}\to \om_{\hat{X}_s/k(s)}$, where $\hat{X}_s$ is the formal neighborhood of $x$ in $X_s$. 
In particular, this will allow us to make sense of deformations of supercurve singularities in Sec.\ \ref{nodes-def-sec}.

%Thus, if $X/S$ is a supercurve, then for any point $s\in S$ and $x\in X_s$, considering the completions   
%it makes sense to have ??? a supercurve structure on a complete Noetherian $R$-superalgebra $A$, such as $\hat{\OO}_{X,x}$ (where $R$ is a base superring). Assume that the %corresponding dualizing complex $\pi^!\OO_S$ is concentrated in degree $1$, and we denote the corresponding module as $\om_{A/R}$. A supercurve structure is
%defined by an odd derivation $\de:\A\to \om_{A/R}$. Furthermore, we assume that there exist finitely many even elements $f_1,\ldots,f_k$ such that the map
%$$A\to A[f_1^{-1}]\oplus\ldots\oplus A[f_k^{-1}]$$
%is injective, and each localization $A[f_i^{-1}]$ is isomorphic to ???

The following formal local description of $\om_{X/S}$ will be useful (this is a particular case of \cite[1.6]{Deligne}).
%we don't know how to prove the assertion \cite[1.6]{Deligne} in general).

\begin{lemma}\label{om-stable-lem}
Let $S=\Spec(R)$, where $R$ is a local Artinian $\C$-superalgebra; $C=\Spf(B)$, with $B=R[\![w_1,w_2]\!]/(w_1w_2)$; 
%where $t$ is an even element in the maximal ideal $\fm\sub R$;
$X=\Spf(A)$, where $B\sub A$ is a finite extension of superalgebras. Assume that for $i=1,2$, there exists an odd element $\th_i\in A[w_i^{-1}]$, such that
$$A[w_i^{-1}]\simeq %B[z_i^{-1}][\th_i]\simeq 
R(\!(z_i)\!)[\th_i],$$
where $w_i=z_i^d$ for some $d>0$.
Then for every $\eta_i\in \om_{X/S}[w_i^{-1}]$, for $i=1,2$,
% where $U_i\sub X$ is the open where $z_i$ is invertible, 
we can write
$$\eta_i=(a_i(z_i)+b_i(z_i)\th_i)[dz_i|d\th_i],$$ 
where $a_i(z_i)$ and $b_i(z_i)$ are Laurent series.
We set $\Res_{z_i=0}(\eta_i):=\Res_{z_i=0}b_i(z_i)dz_i$. Then
$$\om_{X/S}:=\{(\eta_1,\eta_2)\in \om_{X/S}[w_1^{-1}]\oplus \om_{X/S}[w_2^{-1}] \ |\ \eta_1, \eta_2 \text{ compatible and }
\forall a\in A, \Res_{z_1=0}(a\eta_1)+\Res_{z_2=0}(a\eta_2)=0\},$$
where compatibility of $\eta_1$ and $\eta_2$ means that they define the same element of $\om_{X/S}[(w_1w_2)^{-1}]$.
\end{lemma}

\begin{proof}
%Let $\pi:X\to C$ be the projection, and let $j:\ov{U}=\ov{U}_1\cup\ov{U}_2\to C$ be the open embedding, where $\ov{U}_i\sub C$ is the open where $z_i$ is invertible .
%We will combine a natural isomorphism 
Since $\pi:X\to C$ is a finite morphism, we have
\begin{equation}\label{om-X-S-Hom-isom}
\om_{X/S}\simeq\und{\Hom}(\pi_*\OO_X,\om_{C/S}).
\end{equation}
On the other hand, we have a well known description of $\om_{C/S}(C)$
as compatible pairs $(\eta_1,\eta_2)$, with $\eta_i\in \om_{C/S}[z_i]$ such that
\begin{equation}\label{C-res-condition-eq}
\Res_{z_1=0}(b\eta_1)+\Res_{z_2=0}(b\eta_2)=0,
\end{equation}
for every $b\in B=\OO(C)$.
Note that the isomorphism \eqref{om-X-S-Hom-isom} is induced by the trace map $\tau:\pi_*\om_{X/S}\to \om_{C/S}$, whose localizations on $\om_{X/S}[w_i^{-1}]$ are given by
$$\tau([a(z_i)+b(z_i)\th_i][dz_i|d\th_i])=b(z_i)dz_i.$$
The identification of $\om_{X/S}(X)$ with $\Hom_B(A,\om_{C/S}(C))$, sends $\eta\in \om_{X/S}(X)$
%which in turn is identified with a subset in
%$\Hom_B(A,\om_{U/S}(U))$.
to the map $A\to \om_{C/S}(C):a\mapsto \tau(a\eta)$. 

Now we use the fact that $\tau(a\eta)$ is determined by the localizations $\tau(a\eta_i)$, $i=1,2,$
where $\eta_i\in \om_{X/S}[w_i^{-1}]$,
which should satisfy the sum of residue condition \eqref{C-res-condition-eq}.
It follows that $\eta$ is determined by $\eta_i$, $i=1,2$, which should satisfy the condition
$$\Res_{z_1=0}(b\tau(a\cdot \eta_1))+\Res_{z_2=0}(b\tau(a\cdot \eta_2))=0$$
for any $a\in A$, $b\in B$. Since $\tau$ is $B$-linear we can absorb the multiplication by $b$, so we obtain the claimed characterization.
\end{proof}

\subsection{Local descriptions}\label{local-descr-sec}

%For a moment, let us consider stable supercurves over $\C$ (i.e., the case where the base is the reduced point).

\subsubsection{Smooth supercurves with punctures}

It is well known (see e.g., \cite[Lem.\ 1.2]{LR} or \cite[Lem.\ 3.1]{DW} in the absolute case) 
that locally in classical topology near a point of a smooth supercurve, there exist relative coordinates $(z,\th)$ such that 
$\DD$ is generated by $D=\partial_\th+\th\partial_z$ and $\de$ is given by $\de(f)=D(f)\cdot [dz|d\th]$.
We show that the same assertion holds with respect to the \'etale topology and also consider an analogous statement for the case of Ramond punctures
(cf.\ \cite[Prop.\ 3.6]{BHR}).

\begin{lemma}\label{local-descr-smooth-lem}
(i) Let $X/S$ be a smooth supercurve and let $(z,\th)$ be a pair of even and odd local functions such that $dz$ and $d\th$ generate the relative cotangent bundle near a point $p\in X$.
Then there exists another pair $(w,\eta)$ like this defined in an \'etale neighborhood of $p$, with $w\equiv z\mod\NN_X$, such that $\DD$ is generated
by $\partial_\eta+\eta\partial_w$.

\noindent
(ii) Now let $X/S$ be a smooth supercurve with a Ramond puncture $R\sub X$, and let $(z,\th)$ be a pair of even and odd local functions such that $dz$ and $d\th$ generate the relative cotangent bundle and such that the ideal of $R$ is generated by $z$.
Then locally in \'etale topology there exist a change of coordinates to $(w,\eta)$, with $w\equiv z\mod\NN_X$, 
such that $\DD$ is generated by $\partial_\eta+\eta w\partial_w$.
\end{lemma}

\begin{proof} 
(i) Let $\DD$ be generated by an odd vector field $D$ of the form $D=f\partial_\th+g\partial_z$. One has 
$$\frac{1}{2}[D,D]\equiv f\partial_\th(g)\cdot \partial_z \mod\NN_X\TT_{X/S}.$$
Note also that modulo nilpotents $D$ reduces to $f\partial_\th$.
Hence, in order for $D$ and $[D,D]$ to generate $\TT_{X/S}$ both $f$ and $\partial_\th(g)$ should be invertible.

Thus, we can assume that $\DD$ is generated by some vector field of the form $D=\partial_\th+g\partial_z$, where $g=g_1(z)+g_0(z)\th$,
with $g_0$ even and $g_1$ odd. 
Furthermore, $g_0$ is invertible.
Let us look for $w$ and $\eta$ in the form 
\begin{align*}
&\eta=a_0(z)\cdot \th,\\
&w=z+a_1(z)\th,
\end{align*}
where $a_0$ is invertible even and $a_1$ is odd.
Changing to the new coordinates we get
$$D=[a_0+ga'_0\th]\partial_\eta+[-a_1+g(1+a'_1\th)]\partial_w.$$
In order for $\DD$ to be generated by $\partial_\eta+\eta\partial_w$ we need
the equation
$$-a_1+g(1+a'_1\th)=(a_0+ga'_0\th)\cdot a_0\th$$
to be satisfied. This is equivalent to the system
\begin{align*}
&-a_1+g_1=0,\\
&g_0+g_1a'_1=a_0^2.
\end{align*}
Since $g_0$ is invertible, in an \'etale neighborhood we can choose $a_0$ such that $a_0^2=g_0+g_1g'_1$.
This $a_0$ together with $a_1=g_1$ is a solution.

\medskip

\noindent
(ii) Let $D=f\partial_\th+g\partial_z$ be a generator of $\DD$. From the calculation in the beginning of part (i), we get an isomorphism
$$(\TT_{X/S}/\DD+[\DD,\DD])\ot \OO_{X_\bos}\simeq \OO_{X_\bos}/(f\partial_\th(g))\cdot \partial_z\oplus \OO_{X_\bos}/(f)\cdot \partial_\th$$
where $\OO_{X_\bos}=\OO_X/\NN_X$.
Thus, this sheaf surjects onto $\OO_{X_\bos}/(f)\oplus \OO_{X_\bos}/(f)$.
Since this quotient has to be isomorphic to $\OO_X/(\NN_X+(z))$, we deduce that $f$ is invertible modulo $\NN$, so $f$ is invertible.

Thus, we can assume that $\DD$ is generated by $D=\partial_\th+g\partial_z$, where
$g=g_1(z)+g_0(z)\th$. Then we have
$$\frac{1}{2}[D,D]=[g_0+g_1g'_1+(g_1g'_0-g_0g'_1)\th]\cdot \partial_z,$$ 
where $f'$ denotes the derivative with respect to $z$. 
In order for the commutator map $\DD^{\ot 2}\to \TT_{X/S}/\DD$ to vanish exactly on the divisor $(z)$ we should have
$$g_0+g_1g'_1+(g_1g'_0-g_0g'_1)\th=z(u_0(z)+u_1(z)\th),$$
with $u_0$ invertible. In other words,
\begin{equation}\label{g0-g1-eq}
\begin{array}{l}
g_0+g_1g'_1=zu_0, \\ 
g_1g'_0-g_0g'_1=zu_1.
\end{array}
\end{equation}
We claim that this implies that $g_1$ is divisible by $z$. Indeed, let us write
$$g_0\equiv p_0+q_0z \mod (z^2),$$ 
$$g_1\equiv p_1+q_1z\mod (z^2),$$
where $p_i$ and $q_i$ are functions on the base. Note that since $g_1$ is odd, from
the first of the equations \eqref{g0-g1-eq} we get that $q_0$ is invertible. Also, looking at the constant terms of these equations we get
\begin{align*}
&p_0+p_1q_1=0,\\
&p_1q_0-p_0q_1=0.
\end{align*}
This implies that
$$p_1(q_0+q_1^2)=p_1q_0=0.$$
Since $q_0$ is invertible, we deduce that $p_1=0$ which proves our claim that $g_1$ is divisible by $z$.

Now we are going to make the same change of coordinates as in (i) with an additional constraint that $a_1$ is divisible by $z$. 
Since $\eta w=a_0z\th$, we have to solve the equation
$$-a_1+g(1+a'_1\th)=(a_0+ga'_0\th)\cdot a_0z\th,$$
or equivalently, the system
\begin{align*}
&-a_1+g_1=0,\\
&g_0+g_1a'_1=a_0^2z.
\end{align*}
Thus, we get the solution by taking $a_0$ to be the square root of $u_0$
and $a_1=g_1$ (which is divisible by $z$).
\end{proof}

%Locally near a Ramond puncture, there exist relative coordinates $(z,\th)$ such that $\DD$ is generated by $\partial_\th+z\th\partial_z$ and the ideal of $R$ is generated
\begin{definition} We refer to $(w,\eta)$ as in Lemma \ref{local-descr-smooth-lem} as {\it standard coordinates}.
\end{definition}

By Example \ref{de-comp-ex}, if $(z,\th)$ are standard coordinates on a smooth supercurve (away from punctures)
then the canonical derivation is given by 
$$\de(f)=(\partial_\th+\th \partial_z)(f)\cdot [dz|d\th].$$
On the other hand, if $(z,\th)$ are standard coordinates near a Ramond puncture then
$$\de(f)=(\partial_\th+\th z\partial_z)(f)\cdot [\frac{dz}{z}|d\th].$$

\subsubsection{Nodes}\label{nodes-subsubsec}

In the rest of this subsection we consider only the absolute case, i.e., $S=\Spec(\C)$.
Let $X$ be a stable supercurve over $\C$. We denote by $(C,L)$ be the underlying generalized spin curve, and fix a node
$q\in C$. Recall that by definition $X$ is required to be Cohen-Macaulay (CM).
Using the classification of CM sheaves of rank $1$ on the nodal singularity
(see e.g. \cite[Sec.\ 1]{Jarvis-95}), we obtain that there are two possibilities for $L$:

\begin{itemize}
\item $L$ is locally free near $q$. Then we say that this is a {\it Ramond node} (or {\it R-node}).

\item $L$ is the push forward of a line bundle on the normalization of $C$ near $q$. Then we say that this is a {\it NS node}.
\end{itemize}

Furthermore, we have the following local descriptions near the nodes (see \cite{Deligne}), where the description of the generators $\om_{X/S}$ follows from Lemma
\ref{om-stable-lem}.

\begin{itemize}
\item Near a Ramond node $X$ has coordinates $(z_1,z_2,\th)$ (where $\th$ is odd) subject to $z_1z_2=0$. The complement of the node is the union of two branches 
$U_1$ and
$U_2$, where $z_i$ is invertible on $U_i$, and $\om_X$ is free with the basis $\bb$ given by
$$\bb=\begin{cases} [\frac{dz_1}{z_1}|d\th] & \text{ on } U_1,\\ -[\frac{dz_2}{z_2}|d\th] & \text{ on } U_2.\end{cases}$$
The derivation $\de$ is given by
$$\de(f)=\begin{cases} (\pa_\th+\th z_1\pa_{z_1})(f)\cdot \bb & \text{ on } U_1,\\ (\pa_\th-\th z_2\pa_{z_2})(f)\cdot \bb & \text{ on } U_2.\end{cases}$$

\item Near an NS node $X$ has coordinates $z_1,z_2,\th_1,\th_2$ subject to the equations
$$z_1z_2=z_1\th_2=z_2\th_1=\th_1\th_2=0.$$
The complement of the node is given again as the union of two branches $U_1$ and $U_2$, where $z_i$ is invertible on $U_i$ 
and $(z_i,\th_i)$ form coordinates on $U_i$. The generators $[dz_1|d\th_1]$ and $[dz_2|d\th_2]$ of $\om_{U_1}$ and of $\om_{U_2}$ extend to sections $s_1$ and $s_2$
of $\om_X$ (zero on another branch), however, they do not generate it: there is an extra section 
$$s_0=\begin{cases} \frac{\th_1}{z_1}[dz_1|d\th_1] & \text{ on } U_1,\\ -\frac{\th_2}{z_2}[dz_2|d\th_2] & \text{ on } U_2.\end{cases}.$$
The derivation $\de$ is given by
$$\de(f)=\begin{cases} (\pa_{\th_1}+\th_1 \pa_{z_1})(f)\cdot \bb & \text{ on } U_1,\\ (\pa_{\th_2}+\th_2\pa_{z_2})(f)\cdot \bb & \text{ on } U_2.\end{cases}.$$
\end{itemize}

\begin{remark}
The notions of NS node and NS puncture (resp., Ramond node and Ramond puncture) are related via the gluing construction that will be discussed in Sec.\ \ref{boundary-str-sec}.
Namely, two prestable supercurves with NS punctures $(X_1,P_1)$, $(X_2,P_2)$ can be glued along their NS punctures forming a new prestable supercurve with an NS node.
Similarly, one can glue two Ramond punctures into a Ramond node.
\end{remark}

\subsection{Correspondence between NS-punctures and divisors}\label{NS-correspondence-sec}

Here we recall a natural correspondence between NS-punctures $P\sub X$ and effective Cartier divisors $D$ supported at $P$, for a smooth supercurve $X/S$
(with arbitrary base $S$), see \cite{RSV88}, \cite{BR}. Note this correspondence only uses a neighborhood of $P$ in $X$, so it can also be applied for stable supercurves with punctures
of both types, i.e, for each NS-puncture $P_i$ we have a natural Cartier divisor $D_i$ supported at $P_i$.

It is based on the fundamental fact that for a smooth supercurve $X/S$, the composed map
\begin{equation}\label{smooth-superconf-map}
\TT^{sc}_{X/S}\to \TT_{X/S}\to \TT_{X/S}/\DD\simeq \DD^{\ot 2}\simeq \om_{X/S}^{-2}
\end{equation}
is an isomorphism of sheaves, where $\TT^{sc}_{X/S}\sub \TT_{X/S}$ is the subsheaf of superconformal vector fields (those preserving $\DD$). 
Indeed, the proof is easily obtained from the existence of standard coordinates as in Lemma \ref{local-descr-smooth-lem} 
(see \cite[Lem.\ 2.1]{LR}). We denote by
$$\a:\om_{X/S}^{-2}\rTo{\sim} \TT^{sc}_{X/S}$$
the inverse isomorphism. 

\begin{definition}\label{NS-div-corr-def} 
We say that an effective relative Cartier divisor $D\sub X$ {\it corresponds} to an NS-puncture $P\sub X$ if $D$ is supported at $P$
and the following property holds. If the ideal of $D$ is generated
locally by a function $f$ then the ideal of $P$ is generated by $(f,A(f))$, where $A$ is any generator of the distribution
giving the superconformal structure. Note that in this case we necessarily have an inclusion of subschemes $P\sub D$.
\end{definition}

The existence and uniqueness of a divisor $D$ corresponding to an NS-puncture $P$ is easy to check using standard coordinates $(z,\th)$.
Namely, with respect to such coordinates the ideal of $P$ has form $(z,\th-a)$, for some odd function $a$ on the base.
Then corresponding divisor $D$ is given by the ideal $(z+a\th)$. In the next lemma we give a coordinate-free characterization of this correspondence.

\begin{lemma}\label{NS-div-corr-lem}
For a smooth supercurve $X/S$ and an NS-puncture $P\sub X$ there exists a unique effective relative Cartier divisor $D$,
such that the following square is commutative 
\begin{equation}\label{NS-div-corr-square}
\begin{diagram}
\om_{X/S}^{-2}&\rTo{\a}& \TT^{sc}_{X/S}\\
\dTo{}&&\dTo{}\\
\om_{X/S}^{-2}|_{D}&\rTo{\b}&\TT_{X/S}|_{P}
\end{diagram}
\end{equation}
%where $\a$ is the natural (non-$\OO$-linear) isomorphism in a neighborhood of $P_i$, 
where $\b$ is an isomorphism of sheaves of $\OO_S$-modules.
Furthermore, the divisor $D$ corresponds to the NS-puncture $P$ in the sense of Definition \ref{NS-div-corr-def}.
%the ideal of $P$ is recovered from a local equation $f$ of $D$ as $(f,A(f))$.
\end{lemma}

\begin{proof} Let $\a_{P}:\om_{X/S}^{-2}\to \TT_{X/S}|_P$ denote the composition of $\a$ with the projection to $\TT_{X/S}|_{P}$.
The commutativity of the square \eqref{NS-div-corr-square} together with injectivity of $\b$ imply that the kernel of $\a_P$ is exactly
the subsheaf $\om_{X/S}^{-2}(-D)\sub \om_{X/S}^{-2}$. This shows the uniqueness of $D$.

Now let us construct $\b$ and show the commutativity of \eqref{NS-div-corr-square} for $P$ given by the ideal $(z,\th-a)$ and $D$ given by the ideal
$(z+a\th)$, where $(z,\th)$ are standard coordinates and $a$ is an odd function on the base.
Note that the isomorphism $\a$ locally has form
$$f\cdot \bb^{-2}\mapsto f\cdot \partial_z+\frac{1}{2}(-1)^{|f|}A(f)\cdot A,$$
where $A=\pa_{\th}+\th\pa_z$ 
and $\bb=[dz|d\th]$ is a generator of $\om_X$ (see the proof of \cite[Lem.\ 2.1]{LR}, where however the factor $\frac{1}{2}$ is missing).
It is easy to see that
$$\a_{P}((z+a\th)f\cdot\bb^{-2})=0,$$
while $\a_{P}(\bb^{-2})$ and $\a_{P}(\th\cdot \bb^{-2})$ form an $\OO_S$-basis of $\TT_{X/S}|_{P}$.
Since $\bb^{-2}$ and $\th\cdot\bb^{-2}$ project to an $\OO_S$-basis of $\om_{X/S}^{-2}|_{D}$, this shows that the assertion holds with
$$\b(\bb^{-2})=\a_P(\bb^{-2}), \ \ \b(\th\cdot \bb^{-2})=\a_P(\th\cdot\bb^{-2}).$$
\end{proof}

Note that since the composition $P\hra D\to S$ is an isomorphism, the composition $\OO_S\hra \OO_D\to \OO_P=\OO_D/(I_P/I_D)$ is an isomorphism,
which means that we have a canonical splitting
\begin{equation}\label{NS-div-splitting-eq}
\OO_D\simeq \OO_S\oplus I_P/I_D. 
\end{equation}
We will use this splitting later (see Remark \ref{Res-rem}).

\begin{cor}\label{NS-puncture-Tsc-isom-cor} 
The map \eqref{smooth-superconf-map}
induces an isomorphism of sheaves
$$\TT^{sc}_{(X,P)/S}\rTo{\sim} (\TT_{X/S}/\DD)(-D)\simeq \om_{X/S}^{-2}(-D),$$
where $\TT^{sc}_{(X,P)/S}\sub \TT^{sc}_{X/S}$ is the sheaf of superconformal vector fields preserving the ideal of $P$.
\end{cor}

\begin{proof}
This follows from the fact that the map $\a$ induces an isomorphism of the kernels of the vertical arrows in the commutative square \eqref{NS-div-corr-square}.
The kernel of the left vertical arrow in this square is $\om_{X/S}^{-2}(-D)$, while the kernel of the right vertical arrow is $\TT^{sc}_{(X,P)/S}$.
\end{proof}

\begin{remark}
For a general smooth family $X\to S$ of relative dimension $(1|1)$ there is a canonical dual family $\hat{X}\to S$, parametrizing irreducible Cartier divisors in the fibers of $X\to S$
(see e.g., \cite{BR}). The double dual $\hat{\hat{X}}$ is canonically identified with $X$. The superconformal structure on $X/S$ can be viewed as an isomorphism of $X$ with $\hat{X}$.
\end{remark}

\subsection{Supercurves over even bases and spin curves}\label{even-moduli-sec}

If we only consider even bases, the functor of (stable) supercurves with punctures coincides with the functor
of stable curves with punctures equipped with generalized spin structures 
i.e., coherent sheaves $L$ as in Lemma \ref{even-case-rem}.
Note that in the setting of Lemma \ref{even-case-rem} we still refer to the two kinds of marked points $p_\bullet$ and $r_\bullet$ and NS and Ramond punctures.

The moduli stack $\ov{\SS}_{g,n_{NS},n_R}$ of stable spin curves of genus $g$
with $n_{NS}$ NS-punctures and $n_R$ Ramond punctures (where $n_R$ is necessarily even)
was studied in \cite{Jarvis}, \cite{JKV}, as a particular case of the stack of stable $r$-spin curves.
In particular, it shown to be a smooth proper Deligne-Mumford stack. 
Note that although for $r>2$ there are different versions of the functor of stable $r$-spin curves, however, in the case $r=2$, they all coincide (see
\cite{Jarvis-95}).

%Also, the following boundary gluing maps were constructed???

%The following construction will be useful. 

Let us recall how the generalized spin structures over stable curves look like.
A generalized spin structure $L$ over a smooth curve $C$ with punctures is a line bundle $L$ equipped with an isomorphism
$$L^2\simeq \om_C(r_1+\ldots+r_n),$$
where $(r_\bullet)$ are the Ramond punctures.

Now let $(C,p_\bullet,r_\bullet)$ be a stable curve, and
let $\rho:\wt{C}\to C$ be the partial normalization map, resolving a single node $q\in C$.
Let us equip $\wt{C}$ with punctures in the following way:
first, it inherits all the punctures (NS and Ramond) of $C$. Secondly,
we mark the two points in $\rho^{-1}(q)$ as two NS (resp., Ramond) punctures on $\wt{C}$ if $q$ is a NS (resp., Ramond) 
node on $C$.

For a generalized spin structure $L$ over $C$ 
let us define the sheaf $\wt{L}$ on $\wt{C}$ as the quotient of $\rho^*L$ by the torsion subsheaf.
Then $\wt{L}$ is a line bundle on $\wt{C}$ and 
$\wt{L}$ is a generalized spin structure on $\wt{C}$ with NS and Ramond punctures defined as above.

More precisely, if $q$ is a Ramond node then $L$ is locally free near $q$, so $\rho^*L$ is still locally free on $\wt{C}$ and $\wt{L}=\rho^*L$.
If $q$ is an NS node then $L$ is locally isomorphic to the ideal of the node, so $\rho^*L$ acquires a nontrivial torsion. A local computation shows that the quotient $\wt{L}$ of $\rho^*L$
by the torsion subsheaf is locally free, and the composition
of the natural maps $L\to \rho_*\rho^*L\to \rho_*\wt{L}$ is an isomorphism.
%whereas $\wt{L}$ is locally free near the preimage of $q$.
%In the latter case we have a natural isomorphism $L\simeq \rho_*\wt{L}$.

%\section{Infinitesimal study}\label{infinit-sec}

\begin{remark} The coarse moduli of the stack $\ov{\SS}_{g,n_{NS},n_R}$ considered in \cite{Jarvis}, \cite{JKV} gives the same compactification of the classical moduli space
of spin curves as the one constructed by Cornalba \cite{Cornalba}.
On the other hand, there is a different moduli stack constructed in \cite{AJ} (which gives the same coarse moduli as the other constructions), 
where stable curves are replaced by certain stacky curves. We will not use the latter moduli stack in this paper.
\end{remark}

\section{Deformations}\label{def-sec}

In this section we will study the deformation theory of stable supercurves (with punctures). In particular, we will prove that the deformation functor of a stable supercurve is smooth
and will compute its tangent space.

\subsection{Deformation functors on Artin superalgebras}

The development of deformation theory in the super context goes back to \cite{Vaintrob}, see also \cite[Sec.\ 7.1.9]{MZ}.
We consider the category $\Art_{\C}$ of local Artinian $\C$-superalgebras with the residue field $\C$.
Recall that a surjection $A\to B$ is called a {\it small extension} if $\fm\ker(A\to B)=0$,
where $\fm$ is the maximal ideal in $A$. Every surjection in $\Art_{\C}$ is a composition of a finite number of small extensions.

We apply some standard results of deformation theory (say, those in \cite{FM}), or rather their superanalogs (which are straightforward to prove).
Recall that for a set-valued functor $F$ on $\Art_{\C}$ such that $F(\C)=\{*\}$, and for morphisms $A'\to A$, $A''\to A$ in $\Art_{\C}$ one has a natural map
$$F(A'\times_A A'')\to F(A')\times_{F(A)} F(A'') \ \ \ \ \ \ \ \ \ \ \ \ \  \ \ \ \ \ \ \ \ \ \ \ \ \ (\star)$$
One considers the following Schlessinger conditions on $F$:

(H1) The map $(\star)$ is surjective if $A''\to A$ is a small extension;

(H2) The map $(\star)$ is an isomorphism if $A=\C$ and $A''\to A$ is a small extension.

We refer to functors satisfying (H1) and (H2) as {\it deformation functors}. Such a functor $F$ is called {\it smooth} if $F(A')\to F(A)$ is surjective for any small extension $A'\to A$.
More generally, a morphism of deformation functors $F\to G$ is called {\it smooth} if for every small extension $A'\to A$ in $\Art_{\C}$, the natural map
$$F(A')\to G(A')\times_{G(A)}F(A)$$
is surjective.

Note that the {\it tangent space} $t_F$ to a deformation functor $F$ is defined as the sum of even and odd components given by
$$t_F^+=F(\C[\eps]/\eps^2), \ \ t_F^-=F(\C[\tau]/\tau^2),$$
where $\eps$ is even and $\tau$ is odd. Elements of $t_F^+$ (resp., $t_F^-$) are referred to as even (resp., odd) {\it infinitesimal} deformations.

A morphism $h_R\to F$ from the functor $h_R$ pro-representable by a Noetherian complete local $\C$-algebra $\C$ to a deformation functor $F$, is called the {\it hull} of $F$ if
it is smooth, and the induced map on tangent spaces is an isomorphism. In this case we also say that the corresponding deformation over $R$ is {\it miniversal}.

\subsection{Miniversal deformations of the nodes}\label{nodes-def-sec}

In this subsection, following Deligne \cite{Deligne}, we describe miniversal deformations of two types of nodal singularities of prestable supercurves.
First, we give a precise definition of what do we mean by deformations of supercurve singularities.
%and show how this leads to the proof of absence of infinitesimal automorphisms. We also study the sheaf of infinitesimal automorphisms
%in the presence of punctures.

\begin{definition}\label{formal-supercurve-def-def}
Let $A/\C$ be the completion of the local ring of a singular point on a supercurve over $\C$, so it is equipped with a derivation $\de_A:A\to \om_{\Spf(A)/\C}$
(see Sec.\ \ref{formal-supercurve-sec}). By a {\it deformation of $A$} over a local Artinian superalgebra $R$ we mean a flat $R$-superalgebra $B$ equipped with an isomorphism
$B\ot_R \C\simeq A$, and a derivation $\de_B:B\to \om_{\Spf(B)/R}$, inducing $\de_A$ under the reduction with respect to the homomorphism $R\to \C$.
\end{definition}

We begin by describing two families of supercurves over $S=\Spec(\C[t])$, where $t$ is even.

\subsubsection{Ramond node}\label{univ-def-node-I}

Define $X/S$ as a subscheme of the affine space $S\times \A^{2|1}$ over $S$ with coordinates $z_1,z_2,\th$, given 
by the equation 
$$z_1z_2=t.$$
Note that we have a natural trivialization $\bb$ of $\om_{X/S}$.

Over the open subset where $z_1\neq 0$ (resp., $z_2\neq 0$), we have $\bb=[\frac{dz_1}{z_1}|d\th]$
(resp., $\bb=[-\frac{dz_2}{z_2}|d\th]$), and $\de$ is given by $(\pa_\th+\th z_1\pa_{z_1})\cdot\bb$ 
(resp., $(\pa_\th-\th z_2\pa_{z_2})\cdot\bb$).

\subsubsection{NS node}\label{univ-def-node-II}

Define $X/S$ as a subscheme of the affine space $S\times \A^{2|2}$ over $S$ with coordinates $z_1,z_2,\th_1,\th_2$, given 
by the equations 
\begin{equation}\label{NSnode-miniversal-eq}
z_1z_2=-t^2, \ \ z_1\th_2=t\th_1, \ \ z_2\th_1=-t\th_2, \ \ \th_1\th_2=0.
\end{equation}

Over the chart where $z_i\neq 0$ ($i=1,2$),
\begin{equation}\label{NS-miniversal-de-formula}
\de(f)=(\pa_{\th_i}+\th_i\pa_{z_i})(f)[dz_i|d\th_i].
\end{equation}

\begin{remark} We follow the choice of signs in \cite{Deligne} in the equations \eqref{NSnode-miniversal-eq}. Note that one can get rid of signs replacing $z_1$ by $-z_1$ and $\th_1$ and $-\th_1$. As was pointed out by the anonymous referee, there is a natural way to see this deformation of the NS node as a quotient of the Ramond node deformation $z_1z_2=t$ by the $\Z_2$-action 
$$z_1\mapsto -z_1, \ \ z_2\mapsto -z_2, \ \ \th\mapsto -\th.$$
Namely, the algebra of functions on this $\Z_2$-quotient is generated by $Z_1=z_1^2$, $Z_2=z_2^2$, $\Th_1=z_1\th$ and $\Th_2=z_2\th$ subject to relations obtained from
\eqref{NSnode-miniversal-eq} by changing all signs to $+$.
\end{remark}

The following result is proved in \cite{Deligne}. For the convenience of the reader we provide the proof below, with some details elaborated.

\begin{theorem}\label{miniversal-Deligne-thm}
The families over $\C[\![t]\!]$ induced by the above two families are miniversal deformations of the completed rings of the two types of nodes.
\end{theorem}

\begin{proof} Let us denote by $X_0$ the formal spectrum of the completed algebra $R$ of one of the two types of nodes over $\C$ (see \ref{nodes-subsubsec}).
We denote by $F$ the functor on $\Art_{\C}$ of deformations of $X_0$ as a formal supercurve (see Definition \ref{formal-supercurve-def-def}. 
It is easy to see that this functor satisfies (H1) and (H2).
Namely, in the context of (H1), if $X/A$ (resp., $X'/A'$, $X''/A''$) is a deformation of $X_0$ over $A$ (resp., $A'$, $A''$), then $\OO_{\wt{X}}:=\OO_{X'}\times_{\OO_X}\times \OO_{X''}$
can be viewed as a superscheme deformation $\wt{X}$ of $X_0$ over $\wt{A}:=A'\times_A A''$ inducing $X'/A'$ and $X''/A''$. Furthermore, the map
$$\OO_{X'}\times_{\OO_X}\times \OO_{X''}\to \om_{X'/A'}\times_{\om_{X/A}}\om_{X''/A''}\simeq \om_{\wt{X}/\wt{A}}$$
induced by the supercurve structures on $X'/A'$ and $X''/A''$, gives a supercurve structure on $\wt{X}/\wt{A}$. This proves the surjectivity needed for (H1).
In the case $A=\C$, if $\wt{X}$ is any supercurve over $A'\times A''$, let $X'/A'$, $X''/A''$ be the induced supercurves over $A'$ and $A''$. Then we have the induced map
$$\a:\OO_{\wt{X}}\to \OO_{X'}\times_{\OO_{X_0}} \OO_{X''}$$
compatible with the supercurve structures. Since both algebras are flat over $A'\times A''$ and the map of algebras over $\C$, induced by $\a$, is the identity map on $\OO_{X_0}$,
$\a$ is an isomorphism. This proves (H2).
%The main calculation done in \cite{Deligne} is that of the tangent space to the deformation functor for the two types of singularities.

It is enough to prove in both cases triviality of odd infinitesimal deformations and the fact that 
the space of even infinitesimal deformations is a $1$-dimensional space with generators coming from the above families over $\C[t]$. 
Indeed, assume we know this.   
We have a natural morphism $G\to F$ from the functor $G$ prorepresentable by $\C[\![t]\!]$ corresponding to one of our two families,
and by assumption, the map of tangent spaces $t_G\to t_F$ is an isomorphism. Since the functor $G$ is smooth, this implies that the morphism $G\to F$ is smooth (and so $G\to F$
is a hull of $F$, i.e., our deformation over $\C[\![t]\!]$ is miniversal).

Let $A\to B$ be a small extension with the kernel $I$. Then as is well known,
there exists a natural transitive action of $t_F\ot_{\C} I$ on every nonempty fiber of the map $F(A)\to F(B)$ (and similarly for $G$). 
Given $x_B\in G(B)$ mapping to $y_B\in F(B)$, together with a lift $y_A\in F(A)$ of $y_B$, we need to find $x_A\in G(A)$ mapping to both $x_B$ and $y_A$.
For this, first, we can find some $x'_A\in G(A)$ lifting $x_B\in G(B)$. Let $y'_A\in F(A)$ be the image of $x'_A$. Then $y'_A$ is in the same fiber of $F(A)\to F(B)$ as $y_A$,
so it is obtained from $y_A$ by the action of an element $z\in t_G\ot_{\C} I\simeq t_F\ot_{\C} I$. Now we define $x_A$ to be the element in the fiber of $G(A)\to G(B)$ over
$x_B$, obtained from $x'_A$ by the action of $-z$.

Thus, we are reduced to calculating the infinitesimal deformations. The calculation below is from \cite[Sec.\ 3]{Deligne}.
First, we calculate the space $T^1:=T^1_{R/\C}$ of first-order deformations of our superalgebra $R$ over $\C$ as $\Ext^1(L_{R/\C},R)$, where $L_{R/C}$ is the cotangent complex.
The latter is computed using a presentation of $R=S/I$ as the quotient of a super polynomial ring $S$, and a presentation $I=\coker(F_1\to F_0)$, where $F_0$ and $F_1$ are free 
$S$-modules. So the generators of $F_0$ correspond to generators $(f_i)$ of the ideal $I$, while generators of $F_1$ correspond to syzygies between $f_i$. Then $T^1$ is given by the cohomology in the middle term of the complex
$$\Hom_S(\Om_{S/\C},R)\rTo{d_0} \Hom_S(F_0,R)\rTo{d_1} \Hom_S(F_1,R),$$
where the first arrow sends a derivation $D:S\to R$ to the map $e_i\mapsto D(f_i)$.

In the case of the Ramond node, we have a single (even) relation $z_1z_2=0$ and no syzygies, so $T^1$ is identified with the quotient of $R$ by the partial derivatives of 
$z_1z_2$, so the space $T^1$ has dimension $(1,1)$. The corresponding universal deformation of the first order is given by the relation
\begin{equation}\label{Ramond-gen-family}
z_1z_2=t+\tau\th
\end{equation}
over $\C[t,\tau]/(t^2,t\tau)$ (where $t$ is even and $\tau$ is odd).

In the case of the NS node, we have four relations 
$$f=z_1z_2, \ \ \phi_1=z_1\th_2,  \ \ \phi_2=z_2\th_1, \ \ g=\th_1\th_2$$ 
and $4$ generating syzygies
$$\si_1=\th_1f-z_1\phi_2, \ \ \si_2=\th_2f-z_2\phi_1, \ \ s_1=z_1g-\th_1\phi_1, \ s_2=z_2g+\th_2\phi_2.$$
Thus, with respect to the dual bases of $\Hom_S(F_i,R)$ and the basis $\pa_{z_1},\pa_{z_2},\pa_{\th_1},\pa_{\th_2}$ of $\Hom_S(\Om_{S/\C},R)$, the maps $d_0$ and $d_1$
are given by
$$d_0=\left(\begin{matrix} z_2 & z_1 & 0 & 0 \\ \th_2 & 0 & 0 & z_1 \\ 0 & \th_1 & z_2 & 0 \\ 0 & 0 & \th_2 & -\th_1\end{matrix}\right), \ \ 
d_1=\left(\begin{matrix} \th_1 & 0 & -z_1 & 0 \\ \th_2 & -z_2 & 0 & 0 \\ 0 & -\th_1 & 0 & z_1 \\ 0 & 0 & \th_2 & z_2\end{matrix}\right).$$
We need to calculate $\ker(d_1)/\im(d_0)$. 

For $v=af^*+b\phi_1^*+c\phi_2^*+dg^*\in \ker(d_1)$, where $a,b,c,d\in R$, one has $z_1d=\th_1b$, and $z_2d=-\th_2c$. This implies that $d=\th_1 d_1+\th_2 d_2$.
Hence, modifying $v$ by $d_0(\pa_{\th_1})$ and $d_0(\pa_{\th_2})$ (the last two columns of $d_0$), we can make $d=0$. 
Let $H\sub \Hom_S(F_0,R)$ denote the span of $f^*$, $\phi_1^*$ and $\phi_2^*$. We proved that $\ker(d_1)=\ker(d_1)\cap H+\im(d_0)$.
Hence, 
$$\ker(d_1)/\im(d_0)=\ker(d_1)\cap H/\im(d_0)\cap H.$$
Note that $d_0(p\pa_{z_1}+r\pa_{z_2}+s\pa_{\th_1}+t\pa_{\th_2})\in H$ if and only if $\th_2 s=\th_1 t$, which is equivalent to $s,t\in (\th_1,\th_2)$.
Since $d_0(\th_1\pa_{\th_1})=d_0(\th_2\pa_{\th_2})=0$, we obtain that 
$$\im(d_0)\cap H=d_0(R\pa_{z_1}+r\pa_{z_2}+\th_2\pa_{\th_1}+\th_1\pa_{\th_2}).$$

Next, for $af^*+b\phi_1^*+c\phi_2^*\in \ker(d_1)$ we have $z_1c\in (\th_1)$, $z_2b\in (\th_2)$, which implies $b,c\in (\th_1,\th_2)$. Thus, adding multiples of $d_0(\pa_{z_1})$
and $d_1(\pa_{z_2})$, we can assume $b\in (\th_1)$, $c\in (\th_2)$. Then the condition to be in $\ker(d_1)$ reduces to $a\in (\th_1,\th_2)$. This identifies $T^1$ with the
quotient of the space of $af^*+b\phi_1^*+c\phi_2^*$ with $a\in (\th_1,\th_2)$, $b\in (\th_1)$, $c\in (\th_2)$, modulo the submodule generated
 by $\th_2d_1(\pa_{z_1})$, $\th_1d_1(\pa_{z_2})$, $d_1(\pa_{\th_1})$ and $d_1(\pa_{\th_2})$. Thus, the quotient is $4$-dimensional with the basis
$$\th_1f^*, \ \ \th_2 f^*, \ \ \th_1\phi_1^*, \ \ \th_2\phi_2^*.$$

In other words, the universal deformation of the first order is given by the relations
\begin{equation}\label{NS-gen-family}
z_1z_2+\tau_1\th_1+\tau_2\th_2=0, \ \ z_1\th_2+t_1\th_1=0, \ \ z_2\th_1+t_2\th_2=0, \ \ \th_1\th_2=0
\end{equation}
over $\C[t_1,t_2,\tau_1,\tau_2]/(t_1^2,t_1t_2,t_2^2)$, where $t_1,t_2$ are even and $\tau_1,\tau_2$ are odd.

By Lemma \ref{inf-def-emb-lem}, the tangent space $t_F$ to our deformation problem is a subspace in $T^1$, and over the restriction of the above family to $t_F$,
%Next, we need to find the loci in the bases of these universal deformations of the 1st order, where 
we have a derivation $\de:\OO_X\to \om_{X/S}$
%corresponding to the supercurve structure on the complement to the singular point $U\sub X$, and 
deforming the derivation $\de_0$ on the special fiber $X_0$. 
From now on we denote by $X/S$ the formal superscheme given by either \eqref{Ramond-gen-family} or \eqref{NS-gen-family} over the superscheme $S=\Spec R$, corresponding
to an (even or odd) $1$-dimensional subspace in $T^1$. 
%Thus, $S$ is Cohen-Macaulay (being isomorphic to either $\Spec(\C[t]/(t^2))$ or $\Spec(\C[\tau])$), so by 
%\cite[Lem.\ 37.21.4]{Stacks}, the total space $X$ is also Cohen-Macaulay.  

%Note that the complement $U\sub X$ to the node is the union of two open subsets: $U_1$ where $z_1$ is invertible and $U_2$ where $z_2$ is invertible. 
%In what follows we will use the fact that global sections of $\OO_X$ and $\om_{X/S}$ are determined by their restrictions to $U$, in particular,
%to give a derivation $\de:\OO_X\to \om_{X/S}$ is equivalent to giving a derivation $\OO_U\to \om_{U/S}$ sending $\OO_X$ to $\om_{X/S}$.

Let $U_i/S$ denote the localization of one of the families \eqref{Ramond-gen-family} or \eqref{NS-gen-family} obtained by inverting $z_i$, for $i=1,2$.
Then we have $\OO(U_i)\simeq R[z_i,\th_i]$, for $i=1,2$, so $\om_{U_i/S}$ is a free $\OO(U_i)$-module with the basis $[dz_i|d\th_i]$.
On the other hand, it is easy to see that in both cases we have $z_1^2z_2^2=0$, so the corresponding algebra is a finite extension of $R[\![w_1,w_2]\!]/(w_1w_2)$, where $w_i=z_i^2$.
Thus, by Lemma \ref{om-stable-lem}, we can characterize the pairs $\om_1\in \om_{U_1/S}$, $\om_2\in \om_{U_2/S}$ corresponding to the global sections of $\om_{X/S}$ by the
condition
\begin{equation}\label{residue-condition-om-def}
\Res_{z_1=0}(f_1\om_1)+\Res_{z_2=0}(f_2\om_2)=0
\end{equation}
for any global function $f\in \OO_X$ with restrictions $f_1=f|_{U_1}$, $f_2=f|_{U_2}$.

Our strategy will be 
%to work with the induced family $X/S$ over a closed super-subscheme $S$ of the parameter superschemes of the 1st order deformations given by \eqref{Ramond-gen-family} or \eqref{NS-gen-family}; 
to extend $\de_0$ to a pair of compatible derivations $\de_{0,i}:\OO_{U_i}\to \om_{U_i/S}$, $i=1,2$, and then to see whether there exists a pair of compatible derivations 
$(\de_1,\de_2)$, with $\de_i$ vanishing over $U_i\cap X_0$, such that $(\de_{0,1}+\de_1,\de_{0,2}+\de_2)$ takes $\OO_X$ to $\om_{X/S}$ 
(we know that such $(\de_1,\de_2)$ exists if and and only if $S$ corresponds to the linear (super)-subspace of $t_F\sub T^1$.)
We will construct the extension $\de_{0,1}$ on $U_1$ using the original formulas for $X_0$, and then will find a compatible $\de_{0,2}$
using equations \eqref{Ramond-gen-family} and \eqref{NS-gen-family}. 

\medskip

\noindent
{\bf Case of Ramond node}.

In this case we just have to check that $t_F^-$ is zero. In other words, we have to consider the family given by $z_1z_2=\tau\th$.
%and check that $\de_0$ cannot be deform. 
Note that $\om_{X/S}$ is free is one generator $\bb$, such that
$$\bb|_{U_1}=[\frac{dz_1}{z_1}|d\th], \ \ \bb|_{U_2}=[-\frac{dz_2}{z_2}|d\th].$$
The derivations $\de_{0,i}:\OO_{U_i}\to \om_{U_i/S}$ are given by $f\mapsto D_{0,i}(f)\cdot \bb$, where
$$D_{0,1}|_{U_1}(f)=\pa_{\th}+\th z_1\pa_{z_1}, \ \ D_{0,2}|_{U_2}(f)=\pa_{\th}-\th z_2\pa_{z_2}.$$

We need to check that there are no compatible derivations $D_i:\OO_{U_i\cap X_0}\to \OO_{U_i\cap X_0}$ 
%(where $U_0$ is the complement to the singular point in $X_0$),
such that $(D_{0,1}+\tau D_1,D_{0,2}+\tau D_2)$ sends $\OO_X$ to $\OO_X$.
For this we observe that
$$z_1|_{U_1}=z_1, \ \ z_1|_{U_2}=\frac{\tau}{z_2}\cdot \th.$$
Hence,
$$D_{0,1}(z_1)|_{U_1}=\th z_1, \ \ D_{0,2}(z_1)|_{U_2}=\frac{\tau}{z_2}.$$
%Thus, $D_0(z_1)-\th z_1$ is zero on $U_1$ and is equal to $\tau/z_2$ on $U_2$. 
Note that $D_2(z_1)$ is zero on $U_2\cap X_0$ (since $z_1$ vanishes on $U_2\cap X_0$), hence $(D_{0,1}+\tau D_1,D_{0,2}+\tau D_2)(z_1)$ would be a global function on $X$, whose restriction to $U_2$ is equal to $\frac{\tau}{z_2}$,
which is impossible. 

\medskip

\noindent
{\bf Case of NS node, odd deformations}.

Again we want to check the vanishing of $t_F^-$. We know that $t_F^-$ is a subspace of the 2 dimensional odd vector space $(T^1)^-$ corresponding to the odd part of 
the family \eqref{NS-gen-family}. Note that we have an action of $\G_m^2$ on the NS node $X_0$, so that the weights of the generating functions are
$$wt(\th_1)=(1,0), \ \ wt(\th_2)=(0,1), \ \ wt(z_1)=(2,0), \ \ wt(z_2)=(0,2).$$
The subspace $t_F^-$ is invariant under the induced $\G_m^2$-action on $(T^1)^-$. Since the weights of $\tau_1$ and $\tau_2$ are
$$wt(\tau_1)=(1,2), \ \ wt(\tau_2)=(2,1),$$
we see that if $t_F^-$ is nonzero, it coincides with one of the coordinate lines in $(T^1)^-$.

Thus, it is enough to study the restriction of the family \eqref{NS-gen-family} to the $\tau_1$ direction:
$$z_1z_2=\tau_1 \th_1, \ \ z_1\th_2=z_2\th_1=\th_1\th_2=0.$$

Using the description of the global sections of $\om_{X/S}$ as pairs $(\om_1,\om_2)\in \om_{U_1/S}\oplus \om_{U_2/S}$ satisfying the residue condition \eqref{residue-condition-om-def}, 
we can lift the generators of $\om_{X_0}$ to global sections of $\om_X$.
Namely, the generator $s_1=([dz_1|d\th_1],0)$ can be lifted to 
$$\wt{s}_1:=([dz_1|d\th_1], -\tau_1 \frac{\th_2}{z_2^2}[dz_2|d\th_2])$$
(one has to apply the residue condition to the global function $z_2$).
On the other hand, the generator $s_2=(0,[dz_2|d\th_2])$ on $X_0$ lifts to the same pair on $X$ that we denote $\wt{s}_2$, and the generator 
$$s=(\frac{\th_1}{z_1}[dz_1|d\th_1],-\frac{\th_2}{z_2}[dz_2|d\th_2])$$
on $X_0$ lifts to the same pair on $X$ that we denote $\wt{s}$.

The derivations $\de_{0,i}:\OO_{U_i}\to \om_{U_i/S}$, for $i=1,2$, are still given by the formula \eqref{NS-miniversal-de-formula}. We need to show that for any choice of compatible 
derivations
$\de_i:\OO_{U_i}\to \om_{U_i/S}$, $i=1,2$, the pair $(\de_{0,1}+\tau_1\de_1,\de_{0,2}+\tau_1\de_2)$ does not take $\OO_X$ to $\om_{X/S}$.
Note that 
$$(\de_{0,1},\de_{0,2})(\th_1,0)=([dz_1|d\th_1],0)=\wt{s}_1+\tau_1(0,\frac{\th_2}{z_2^2}[dz_2|d\th_2]),$$ 
Hence, if $(\de_{0,1}+\tau_1\de_1,\de_{0,2}+\tau_1\de_2)(\th_1,0)\in \om_{X/S}$ and $\de_1(\th_1)=\eta_1$, then we would get
$$(\eta_1,\frac{\th_2}{z_2^2}[dz_2|d\th_2])\in \om_{X/S},$$
with $\eta_1$ regular on $U_1$, which is impossible by the above description of $\om_{X/S}$.

\medskip

\noindent
{\bf Case of NS node, even deformations}.

In this case we work with the first-order deformation
$$z_1z_2=0, \ \ z_1\th_2=t_1\th_1, \ \ z_2\th_1=t_2\th_2, \ \ \th_1\th_2=0.$$
We have to check that the equation $t_1+t_2=0$ on $S$ is necessary in order for this to be in $t_F$.
 
The generator $s_1$ (resp., $s_2$) of $\om_{X_0/S}$ lifts to 
$$\wt{s}_1=([dz_1|d\th_1],-t_1\frac{1}{z_2}[dz_2|d\th_2]) \ (\text{resp., } \wt{s}_2=(-t_2\frac{1}{z_1}[dz_1|d\th_1],[dz_2|d\th_2])),$$
while the generator $s$ lifts to the same pair on $X$ that we denote $\wt{s}$.

The derivations $\de_{0,i}:\OO_{U_i}\to \om_{U_i/S}$, for $i=1,2$, is still given by the formula \eqref{NS-miniversal-de-formula}. 
We have 
$$(\de_{0,1},\de_{0,2})(\th_1)=(\de_{0,1},\de_{0,2})(\th_1,t_2\frac{\th_2}{z_2})=([dz_1|d\th_1],t_2\frac{1}{z_2}[dz_2|d\th_2])=\wt{s}_1+(t_1+t_2)(0,\frac{1}{z_2}[dz_2|d\th_2]),$$
%and similarly,
%$$\de_0(\th_2)=\wt{s}_2+(t_1+t_2)(\frac{1}{z_1}[dz_1|t\th_2],0).$$
Suppose $(\de_{0,1}+t_1\de_1+t_2\de'_1,\de_{0,2}+t_1\de_2+t_2\de'_2)(\th_1)\in \om_{X/S}$. We have 
$$(t_1\de_1,t_1\de_2)(\th_1)=(t_1\de_1(\th_1),t_1\de_2(t_2\frac{\th_2}{z_2}))=(t_1\eta,0), \ \ (t_2\de'_1,t_2\de'_2)(\th_1)=(t_2\eta',0),$$
where we used the relations $t_1t_2=t_2^2=0$.
Hence, we get 
$$((t_1\eta+t_2\eta',(t_1+t_2)\frac{1}{z_2}[dz_2|d\th_2])\in \om_{X/S},$$
with $\eta,\eta'$ regular on $U_1$, which is impossible unless
$t_1+t_2=0$ on $S$.
\end{proof}

\begin{cor}\label{smooth-def-nodes-cor}
The functors of deformations of both types of nodes are smooth.
\end{cor}

\begin{proof} Indeed, the functor $G$ prorepresentable by $\C[\![t]\!]$ is smooth. Since the morphism of deformation functors $G\to F$ is smooth (where $F$ is our deformation functor), this implies that $F$ is smooth.
\end{proof}

The result on formal deformations implies in a standard way (essentially via Artin approximation techniques) the following \'etale local description of neighborhoods of singular points in families of prestable supercurves.

\begin{lemma}\label{etale-local-nodes-lem}
Let $X\to S$ be a prestable supercurve, where $S$ is of finite type over $\C$. Then for any node $q$ in a fiber $X_s$ over $s\in S$, there exists an \'etale neighborhood $V$ of $s$ in $S$,
such that for each preimage $q'$ of $q$ in $X_V$, there exists a morphism $t:V\to \A^1$, such that \'etale locally near $q'$ the family $X_V\to V$ is identified with the pull-back under $f$ of the standard NS-node or Ramond-node deformation over $\A^1$.
\end{lemma}

\begin{proof} This is proved in \cite[Prop.\ 7.10]{MZ}. The key idea (also used in Artin approximation) is to use a presentation of the formal completion as an inductive
limit of smooth algebras given by Popescu's theorem.
\end{proof}

\subsection{Sheaf of infinitesimal automorphisms}\label{inf-aut-sh-sec} 

For a moment let us consider only supercurves over the point base $\Spec(\C)$.

\begin{definition} We define the $\Z_2$-graded sheaf of infinitesimal automorphisms $\AA$ of any geometric structure of the form (superscheme over $\C$
plus extra structure) as follows. The even part $\AA^+$ is given by automorphisms of the trivial family of these structures over $\C[\eps]/(\eps^2)$,
where $\eps$ is even, trivial modulo $\eps$. Similarly, the odd part $\AA^-$ is given by automorphisms of the trivial family over $\C[\tau]/(\tau^2)$, where
$\tau$ is odd, trivial modulo $\tau$. For example, the sheaf of infinitesimal automorphisms of a superscheme is exactly the tangent sheaf.
\end{definition}

For a stable supercurve $(X,P_\bullet,R_\bullet)$ over $\C$, we denote by 
$$\AA_{X,P_\bullet,R_\bullet}\sub \TT_X$$ 
the sheaf of its infinitesimal automorphisms (where $\TT_X$ is the sheaf of super-derivations from $\OO_X$ to $\OO_X$). 
Note that the space of locally trivial
infinitesimal deformations of $(X,P_\bullet,R_\bullet)$ is $H^1(X,\AA_{X,P_\bullet,R_\bullet})$.

\begin{lemma}\label{A-X-U-lem}
Let $(X,P_\bullet,R_\bullet)$ be a stable supercurve, $j:U\hra X$ a smooth locus. Then  
$\AA_{(X,P_\bullet,R_\bullet)}$ is identified with the subsheaf of $j_*\AA_{U,P_\bullet,R_\bullet}$ consisting of
those derivations of $\OO_U$ that send $\OO_X$ to $\OO_X$. 
\end{lemma}

\begin{proof}
Note that $\OO_X$ is a subsheaf in $j_*\OO_U$, hence the sheaf $\TT_X$ of derivations from $\OO_X$ to $\OO_X$,
can be identified with the subsheaf in $j_*\TT_U$ consisting of derivations of $\OO_U$ sending $\OO_X$ to $\OO_X$.
Now the assertion follows from Lemma \ref{isom-stable-lem} (applied to automorphisms of $X\times \Spec(\C[\eps]/(\eps^2))$ and $X\times \Spec(\C[\tau]/(\tau^2))$).
\end{proof}

\begin{definition}
Now let us consider a stable supercurve with punctures
$(X,P_\bullet,R_\bullet)$ over any base $S$.
We define 
$$\AA_{X/S}=\AA_{(X,P_\bullet,R_\bullet)/S}\sub \TT_{X/S}$$ 
as the subsheaf consisting of derivations $v$ in $\TT_{X/S}$ preserving the punctures and preserving the distribution $\DD$ over the smooth locus.
We still call $\AA_{X/S}$ the sheaf of infinitesimal automorphisms.
\end{definition}

In the case $S=\Spec(\C)$ this agrees with our previous definition. Indeed this follows easily from Lemma \ref{A-X-U-lem} which states
%is clear on the smooth locus $U\sub X$.
that $\AA_X$ is the intersection of $\TT_X$ with $j_*\AA_U$ in $j_*\TT_U$.
Note however that the formation of $\AA_{X/S}$ is not compatible with the base change in general.

In the smooth case we have the following useful identification of the sheaf of infinitesimal automorphisms over an arbtrirary (not necessarily even) base.

\begin{prop}\label{inf-aut-smooth-punct-prop} 
Let $(X,P_\bullet,R_\bullet)\to S$ be a smooth supercurve with punctures, where $(P_i)_{i\in I}$ are NS punctures
and $(R_j)_{j\in J}$ are Ramond punctures.
Then one has a natural isomorphism
$$\AA_{(X,P_\bullet,R_\bullet)/S}\simeq \TT^{sc}_{(X,P_\bullet)/S}\simeq \om_{X/S}^{-2}(-\sum_{i\in I} D_i-2\sum_{j\in J} R_j),$$
where $D_i\sub X$ is a divisor associated with the NS puncture $P_i$ (see Sec.\ \ref{NS-correspondence-sec}).
\end{prop}

\begin{proof}
The first isomorphism corresponds to the fact that the Ramond punctures are recovered from the distribution $\DD\sub \TT_{X/S}$ as the vanishing divisor
of the $\OO_X$-linear map 
$$\DD^{\ot 2}\to \TT_{X/S}/\DD$$
induced by the Lie bracket. Thus, if a relative derivation on $X/S$ preserves $\DD$ and the NS punctures then it also preserves the Ramond punctures.

Next, we observe that there is a natural isomorphism
$$\TT_{X/S}/\DD\simeq \DD^2(\sum_j R_j)\simeq \om_{X/S}^{-2}(-\sum_j R_j).$$
(see \eqref{Ramond-D2-isom} and \eqref{Ramond-D-om-isom}). It remains to check that the natural map
%so we can consider the composed map
$\TT^{sc}_{(X,P_\bullet)/S}\to \TT_{X/S}/\DD$ induces an isomorphism
%\simeq  \om_{X/S}^{-2}(-\sum_j R_j).$$
$$\TT^{sc}_{(X,P_\bullet)/S}\rTo{\sim} (\TT_{X/S}/\DD)(-\sum_{i\in I} D_i-\sum_{j\in J} R_j).$$
%It remains to check the latter map induces an isomorphism of the source with the subsheaf $\om_{X/S}^{-2}(-\sum_{i\in I} D_i-2\sum_{j\in J} R_j)$.
This is a local question which is well known away from the punctures. Near the NS puncture this is Corollary \ref{NS-puncture-Tsc-isom-cor}, while near the Ramond puncture
this is proved in \cite[Prop.\ 3.12]{BHR} using standard coordinates (see Lemma \ref{local-descr-smooth-lem}). 
\end{proof}

In the case when the base is a point it is useful to rewrite the result in terms of the corresponding spin curve.

\begin{cor} Let $(X,P_\bullet,R_\bullet)$ be a smooth supercurve over $\C$ with the underlying spin curve $(C,L,p_\bullet,r_\bullet)$, where
$L^2\simeq \om_C(\sum r_j)$. Then
$$\AA_{X,P_\bullet,R_\bullet}^+\simeq \TT_C(-\sum p_i-\sum r_j), \ \ \AA_{X,P_\bullet,R_\bullet}^-\simeq \TT_C\ot L(-\sum p_i-\sum r_j).$$
\end{cor}

\begin{proof}
Let $\pi:X\to C$ be the projection. We have $\om_X\simeq \pi^*L(-\sum r_j)$ (see Lemma \ref{even-case-rem}), so
$$\om_X^{-2}\simeq \pi^*L^{-2}(2\sum r_j)\simeq \pi^*\TT_C(\sum r_j).$$
Hence, 
$$\om_X^{-2}(-\sum p_i-2\sum r_j)\simeq \pi^*\TT_C(-\sum p_i-\sum r_j)=\OO_X\ot_{\OO_C}\TT_C(-\sum p_i-\sum r_j).$$
Considering even and odd parts we get the result.
\end{proof}

\subsection{Sheaf of infinitesimal automorphisms for stable supercurves}

The following local analysis of the sheaf  $\AA_{X,P_\bullet,R_\bullet}$ is from \cite{Deligne} (we corrected a misprint in the case of
the Ramond node).

\medskip

\noindent
{\it Ramond node}.

Locally near such a node, we have $\OO_X=\OO_C\oplus L$, where $L$ is locally free, $L^{\ot 2}\simeq \om_C$,
and
$$\AA_X=\TT_C\oplus L\ot \om_C^{-1}.$$
%\sub j_*\AA_U.
%Note that if $C=B_1\cup B_2$ are two branches glued at points $q_i\in B_i$,
%then the restriction of $\om_C^{-1}$ to each branch $B_i$ is $\TT_{B_i}(-q_i)$.

\medskip

\noindent
{\it NS node}.

Let $C=B_1\cup B_2$ be the two branches. Each $B_i$ is equipped with a square root $L_i$
of $\om_{B_i}$.
Then $\AA_X=\AA_1\oplus \AA_2$, with
$$\AA_i=\TT_{B_i}(-q_i)\oplus L_i\ot \TT_{B_i}(-q_i)$$

%\medskip

%\noindent
%{\it NS-punctures}

%$\AA_{X,P}$ is a subsheaf of $\AA_X$ consisting of automorphisms preserving $P$ (i.e., the ideal $\OO_C(-p)\oplus L$). 
%Thus,
%$$\AA_{X,P}=\TT_C(-p)\oplus L\ot \TT_C(-p).$$

%\medskip

%\noindent
%{\it R-punctures}

%We have $\OO_X=\OO_C\oplus L$, where $L^{\ot 2}\simeq \om_C(r)$.
%We claim that
%$$\AA_{X,R}=\TT_C(-r)\oplus L\ot \TT_C(-r).$$
%Note that $\AA_{X,R}$ consists of infinitesimal automorphisms of $X$ as a superscheme that preserve the superconformal
%structure on the complement $U$ to the puncture 
%(since we can recover the ideal of the Ramond divisor as generated by even functions
%$f$ such that $f\cdot \de_U(\OO^-_X)$ is regular on $X$).

%Indeed, let $U\sub X$ be the complement to the R-puncture.
%Then $\AA_{X,R}^+$ is a subsheaf of $\TT_U$, and it should send $\OO_C$ to $\OO_C$, preserving the
%ideal $\OO_C(-r)$. It follows that $\AA_{X,R}^+\sub \TT_C(-r)$. We claim that the action of $\TT_C(-r)$ preserves
%$\OO_X$. Indeed, for $\th\in L$ and $v\in \TT_C(-r)$, we have
%$L_v(\th^2)\in \om_C$, so $L_v(\th^2)/\th\in L(-r)$.

%On the other hand, $\AA_{X,R}^-$ consists of sections of $L\ot\TT_C|_U$ that send $L$ to $\OO_C$ and $\om_C$ to $L$.
%This gives $\AA_{X,R}^-=L\ot \TT_C(-r)$.

We can also determine the sheaf of infinitesimal automorphisms globally and
check the absence of infinitesimal automorphisms for stable supercurves by the standard count
(this is done in \cite{Deligne} for the case of supercurves without punctures).

\begin{prop}\label{no-inf-aut-prop} 
Let $(X,P_\bullet,R_\bullet)$ be a stable supercurve over $\C$ with the underlying stable spin curve $(C,L,p_\bullet,r_\bullet)$. 
Let $\rho:\wt{C}\to C$ be the normalization with the induced spin structure $\wt{L}$
(see Sec.\ \ref{even-moduli-sec}), and let $(\wt{X},\wt{P}_\bullet,\wt{R}_\bullet)$ be the corresponding smooth supercurve with punctures.

\noindent
(i) One has natural a isomorphism
$$\AA^+_{X,P_\bullet,R_\bullet}\simeq \rho_*\AA^+_{\wt{X},\wt{P}_\bullet,\wt{R}_\bullet}\simeq \TT_C(-\sum p_i-\sum r_j)$$
and an exact sequence
$$0\to \AA^-_{X,P_\bullet,R_\bullet}\to \rho_*\AA^-_{\wt{X},\wt{P}_\bullet,\wt{R}_\bullet}\to \bigoplus_{q\text{ R-node}}\OO_q\to 0$$
where the summation is over all Ramond nodes.

\noindent
(ii) One has $H^0(X,\AA_{X,P_\bullet,R_\bullet})=0$.
\end{prop}

\begin{proof} 
%The vanishing of $H^0(X,\AA_X^+)$ is a well known (purely even) statement that a stable spin curve has no infinitesimal automorphisms.
%???
%By Proposition \ref{inf-aut-odd-prop}, to check the vanishing of $H^0(X,\AA_X^-)$, it is enough to check that
%$L\ot \om_{C/S}^{-1}(-\sum p_i-\sum r_j)$
%Let $\wt{C}\to C$ be the normalization and let $\wt{L}$ be the induced spin structure on $\wt{C}$ (with punctures),
%see Sec.\ \ref{even-moduli-sec}.
(i)  We have a natural morphism  
$$\kappa:\AA_{X,P_\bullet,R_\bullet}\to \rho_*\AA_{\wt{X},\wt{P}_\bullet,\wt{R}_\bullet}$$ 
which is an isomorphism
away from the nodes. Thus, the assertion can be checked by a local computation near the nodes (so we can forget about the punctures). 
Near an NS node we have
$$\AA^+_X\simeq \rho_*\TT_{\wt{C}}(-q_1-q_2), \ \ \AA^-_X\simeq \rho_*\wt{L}\ot\TT_{\wt{C}}(-q_1-q_2),$$
where $\{q_1,q_2\}\sub\wt{C}$ is the preimage of the node, so $\kappa$ is an isomorphism.

Near a Ramond node we have
$$\AA^+_X\simeq \TT_C\simeq \rho_*\TT_{\wt{C}}(-q_1-q_2), \ \ \AA^-_X\simeq L\ot \om_C^{-1},$$
whereas
$$\AA^-_{\wt{X}}\simeq \wt{L}\ot\om^{-1}_{\wt{C}}(-q_1-q_2)\simeq \rho^*(L\ot \om_C^{-1}).$$
Thus, $\kappa^+$ is an isomorphism, while $\kappa^-$ is an embedding with the cokernel of length 1 supported at the node.

\noindent
(ii) 
%$(\wt{C},\wt{L})$ is a smooth (not necessarily connected) stable pointed curve with a spin structure.
Since global infinitesimal automorphisms of $X$ embed into those for $\wt{X}$,
it is enough to prove the assertion in the case when $C$ is smooth. We can also assume it is connected.

When $C$ is smooth we have
$$\AA_X^+=\TT_C(-\sum p_i-\sum r_j), \ \ \AA_X^-=L\ot \TT_C(-\sum p_i-\sum r_j).$$
We have $L^2\simeq \om_C(\sum r_j)$, so
$$\deg L=g-1+n_R/2,$$
$$\deg L\ot \TT_C(-\sum p_i-\sum r_j)=-(g-1)-n_{NS}-n_R/2, \ \ 
\deg \TT_C(-\sum p_i-\sum r_j)=-2(g-1)-n_{NS}-n_R.$$
Note that $\deg \TT_C(-\sum p_i-\sum r_j)<0$ by stability of the underlying pointed curve.
Hence,
$$\deg L\ot \TT_C(-\sum p_i-\sum r_j)=\frac{1}{2}\deg \TT_C(-\sum p_i-\sum r_j) - n_{NS}/2<0,$$
so $\AA_X^\pm$ do not have global sections.
\end{proof}

\subsection{Smooth affine supercurves without Ramond punctures}\label{no-R-punct-def-sec}

Let $X_0$ be an affine (smooth) supercurve over $\C$.

\begin{lemma}\label{supeconf-field-pt}
For every $\C$-point $p\in X_0$, and every tangent vector $v_0\in T_pX_0$, there exists a superconformal
vector field $v$ on $X_0$ with $v(p)=v_0$.
\end{lemma}

\begin{proof} First, we claim that there is an exact sequence of sheaves with respect to \'etale topology
$$0\to \AA_{X_0,p}\to \AA_{X_0}\to i_*T_pX_0\to 0$$
where $\AA_{X_0}=\TT_{X_0}^{sc}$. Indeed, for this we need to check the surjectivity of the last arrow.
Locally in \'etale topology the distribution  $\DD\sub \TT_{X_0}$ has a generator $D=\pa_\th+\th\pa_z$ (see Lemma \ref{local-descr-smooth-lem}). 
Then the vector fields $D-2\th D^2$ and $D^2$ are superconformal and restrict to a basis of $T_pX_0$.

Next, we observe that $\AA_{X_0,p}$ is isomorphic to a coherent sheaf on $X_0$ (see Proposition \ref{inf-aut-smooth-punct-prop}).
Hence, $H^1(X_0,\AA_{X_0,p})=0$ and we obtained the required surjectivity of the map $H^0(X_0,\AA_{X_0})\to T_pX_0$.
\end{proof}

\begin{lemma}\label{subbun-def-lem}
Let $A\to A/I=B$ be a small extension in $\Art_{\C}$, $X_0$ a superscheme over $\C$, $E_0$ a vector bundle on $X_0$, $D_0\sub E_0$ a subbundle. 
We denote by
$$(X_B=X_0\times \Spec(B), E_B=B\ot_{\C} E_0, D_B=B\ot_{\C} D_0)$$
the trivial deformation of these data over $B$. Consider $X_A=X_0\times \Spec(A)$ with the vector bundle $E_A=A\ot_{\C} E_0$.
%Assume $X_A$ is a smooth superscheme of dimension $1|1$ over $A$ such that $X_A\times_{\Spec(A)}\Spec(B)\simeq X_B$.
Then subbundles $D_A\sub E_A$ reducing to $D_B\sub E_B$ over $B$
are in natural bijection with homomorphisms of bundles $D_0\to I\ot_{\C} (E_0/D_0)$ over $X_0$.
\end{lemma}

\begin{proof}
Since $\fm I=0$, we have a natural identification 
$$I\cdot D_A\simeq I\ot_A D_A\simeq I\ot_{\C}(\C\ot_A D_A)\simeq I\ot_{\C} D_0.$$
Thus, $D_A$ is an $\OO_{X_0}\ot A$-submodule in $A\ot_{\C} E_0$ containing $I\ot_{\C} D_0$ and reducing to $D_B$ modulo $I$.
Furthermore, since $D/I\cdot D_A=D_B$ embeds into $E_B=(A/I)\ot_{\C} E_0$, we have 
$$D_A\cap I\ot_{\C} E_0=I\cdot D_A=I\ot_{\C} D_0.$$
Let $\pi:A\ot_{\C} E_0\to B\ot_{\C} E_0$ be the projection. Let us consider the $\OO_{X_0}$-submodule
$$D'_A:=\pi^{-1}(1\ot D_0)\cap D_A\sub A\ot_{\C} E_0.$$
Since $D'_A$ contains $I\cdot D_A$ and $\pi$ induces a surjection from $A\cdot D'_A$ to $D_A/I\cdot D_A=A\ot_{\C} D_0$, we see that
$$D_A=A\cdot D'_A,$$
so $D_A$ is determined by $D'_A$. 

But $D'_A$ is in turn determined by the $\OO_{X_0}$-submodule
$$\ov{D}_A:=D'_A/(I\ot_{\C} D_0)\sub A\ot_{\C} E_0/(I\ot_{\C} D_0),$$
which projects to $1\ot D_0\sub (A/I)\ot_{\C} E_0$. 
We have a natural exact sequence of $\OO_{X_0}$-modules
$$0\to I\ot_{\C} (E_0/D_0)\to A\ot_{\C} E_0/(I\ot_{\C} D_0)\to (A/I)\ot_{\C} E_0\to 0$$
equipped with a spitting $1\ot v\mapsto 1\ot v$ over $1\ot E_0$. Hence, we can view $\ov{D}_A$ as an $\OO_{X_0}$-submodule
of $I\ot_{\C} (E_0/D_0)\oplus 1\ot D_0$, that intersects the first summand trivially. In other words, $\ov{D}_A$ is the graph of a homomorphism
$D_0\to I\ot_{\C} (E_0/D_0)$, which gives the claimed bijection.
\end{proof}

We consider the functor $\Def(X_0)$ from $\Art_{\C}$ to the category of sets associating with $A$ the set of isomorphism classes of
deformations of $X_0$ over $A$ as an affine supercurve. 
A deformation of a smooth supercurve $(X_0,\DD_0)$ over $A\in \Art_{\C}$ consists of a superscheme $X_A$, flat over $A$, and reducing to $X_0$ over $\C$,
as well as a distribution $\DD_A\sub \TT_{X_A/\Spec(A)}$ reducing to $\DD_0$. Note that $(X_A,\DD_A)$ will then automatically be a smooth supercurve over $A$ (since
both smoothness and surjectivity of the map \eqref{commutator-map} are open conditions).

Below we will use the fact that a group valued functor on local Artinian $\C$-algebras, satisfying (H1) and (H2) is automatically smooth (see \cite[Thm.\ 7.19]{FM}),
or rather a superanalog of this fact for group values functors on $\Art_{\C}$.

\begin{lemma}\label{Aut-smooth-lemma}
Let $X_0$ be a smooth supercurve over $\C$. Consider the functor on $\Art_{\C}$ that associates with $A$ the group $\Aut_0(X_A/A)$ of automorphisms
of $X_A=X_0\times\Spec(A)$ as a supercurve over $A$, reducing to the identity over $\C$. Then this functor satisfies (H1) and (H2), hence it is smooth.
\end{lemma}

\begin{proof}
It is well known that for any separated scheme $Y_0$ of finite type over $\C$ the functor of automorphisms of $Y_0\times\Spec(A)$ over $A$, reducing to the identity over $\C$ satisfies
(H1) and (H2) (see \cite[Ex.\ 7.2]{FM}). This result also holds for superschemes.
Since the functor $A\mapsto \Aut_0(X_A/A)$ is a subfunctor in such a functor for $X_0$, it suffices to check condition (H1) for the functor $\Aut_0$.
In the situation of (H1), set $B=A'\times_A A''$, and
suppose we have compatible automorphisms $\a_{A'}$ and $\a_{A''}$ of $X_{A'}$ and $X_{A''}$ as supercurves. Then we know that they come from an automorphism
$\a_B$ of $X_B$ as a superscheme over $B$. We claim that $\a_B$ automatically preserves the distribution $\DD_B:=B\ot_{\C}\DD_0\sub B\ot_{\C} \TT_{X_0}$.
Indeed, since $\DD_B=\DD_{A'}\times_{\DD_A}\DD_{A''}$,
this follows immediately from the fact that $\a_{A'}$ and $\a_{A''}$ preserve the distributions $\DD_{A'}$ and $\DD_{A''}$, respectively.
\end{proof}

\begin{lemma}\label{smooth-curve-trivial-def-lem} 
(i) Every deformation $X_A$ of $X_0$ over $A\in \Art_{\C}$ is isomorphic to a trivial deformation
$X_0\times \Spec(A)$, with the distribution induced by that on $X_0$.

\noindent
(ii) Given a surjection $A\to B$ in $\Art_{\C}$, and a deformation $X_A$ of $X_0$ over $A$, any superconformal automorphism
of $X_B=X_A\times_{\Spec A} \Spec B$ lifts to a superconformal automorphism of $X_A$.

\noindent
(iii) Analogs of (i) and (ii) hold for deformations of an affine supercurve $X_0$ with an NS-puncture $P_0\sub X_0$.

\noindent
(iv) For a surjection $A\to B$ in $\Art_{\C}$ any superconformal automorphism of $B(\!(z)\!)[\th]$ (with standard $\de$) lifts to a superconformal
automorphism of $A(\!(z)\!)[\th]$.
\end{lemma}

\begin{proof} 
(i) Let $X_0=\Spec(S_0)$, $X_A=\Spec(S_A)$.
%$S_0$ (resp., $S_A$) be the algebra of functions on $X_0$ (resp., $X_A$). 
Let also $\fm\sub A$ be the
maximal ideal (which is nilpotent). 
We have an exact sequence
$$0\to \fm S_A\to S_A\to S_0\to 0.$$
In particular, $S_A^+$ is a nilpotent extension of $S_0^+$ in the category of commutative $\C$-algebras. 
Since $S_0^+$ is a smooth finitely generated $\C$-algebra, by the infinitesimal lifting property, there exists a section
$$\si^+:S_0^+\to S_A^+,$$
which is a homomorphism of $\C$-algebras.
 
We know that $S_0^-$ is a locally free $S_0^+$-module of rank $1$. In particular, $S_0^-$ is projective,
so we can choose an $S_0^+$-module splitting 
$$\si^-:S_0^-\to S_A^-$$
of the projection $S_A^-\to S_0^-$, where we view $S_A^-$ as an $S_0^+$-module via the embedding $\si^+$. Furthermore, we claim that $\si^-(S_0^-)\cdot \si^-(S_0^-)=0$ in $S_A$.
Indeed, this follows from the fact that locally $S_0^-$ is generated by one element.
Hence, $\si=(\si^+,\si^-):S_0\to S_A$ is a homomorphism of superalgebras.

From this we get a homomorphism $f:S_0\otimes_{\C} A\to S_A$ of $A$-algebras which induces an isomorphism after tensoring with $A/\fm$.
Hence, $f$ induces isomorphisms
$$S_0\ot_{\C} \fm^i/\fm^{i+1}\to S_A\ot_A \fm^i/\fm^{i+1}$$
for $i\ge 0$.
Since $S_A$ is flat over $A$, we have isomorphisms $\fm^iS_A/\fm^{i+1}S_A\simeq S_A\ot_A \fm^i/\fm^{i+1}$. 
Since $\fm^N=0$ for some $N>0$, the descending induction on $i$ shows that $f$ induces an isomorphism
$$S_0\otimes_{\C} \fm^i\to \fm^iS_A.$$
Hence, $f$ is an isomorphism.
%the graded quotients of the $\fm$-adic filtrations, hence, is an isomorphism.

Now let $\DD_0\sub\TT_{X_0/k}$ be the distribution giving the supercurve
structure on $X_0$, and let us set $X_A=X_0\times \Spec(A)$. 
It remains to show that any distribution $\DD_A\sub \TT_{X_A/A}$ 
giving a supercurve structure over $A$, deforming $\DD_0$,
is isomorphic to the pull-back of $A\ot_{\C}\DD_0$ under some automorphism of $X_A$ trivial on $X_0$.

We can assume that for some ideal $I\sub A$, such that $\fm I=0$ (so $A\to A/I$ is a small extension), 
the assertion holds for $A/I$,
so let 
$$\DD_A\sub \TT_{X_A/A}\simeq A\ot_{\C}\TT_{X_0/k}$$
be a subbundle of rank $0|1$, which reduces to $A/I\ot \DD_0$ over $A/I$.
By Lemma \ref{subbun-def-lem}, $\DD_A$ corresponds to a homomorphism 
$$f:\DD_0\to I\ot_{\C} (\TT_{X_0/k}/\DD_0),$$
so that $\DD_A$ is generated over $A$ by sections of the form $1\ot x+\wt{f(x)}$, for $x\in \DD_0$, where $\wt{f(x)}\in I\ot_{\C} \TT_{X_0/k}$ 
is a representative of $f(x)$.

On the other hand, any $I$-valued vector field $v\in I\ot H^0(X_0,\TT_{X_0/k})$ induces
an automorphism $\a_v$ of $X_A$ trivial on $X_{A/I}$: its action on functions is $\phi\mapsto \phi+v(\phi)$.
The automorphism $\a_v$ acts on vector fields on $X_A$ by $w\mapsto w+[v,w]$, so it changes the distribution $A\ot_{\C}\DD_0$ to the distribution $\DD_A$ 
generated over $A$ by sections of the form $1\ot x+[v,x]$. In other words, $\DD_A$ corresponds to
the homomorphism $f_v:\DD_0\to I\ot_{\C} (\TT_{X_0/k}/\DD_0)$ induced by the Lie bracket with $v$. 
The condition that the map \eqref{commutator-map} is an isomorphism for $X_0$ is equivalent to the condition that the map 
$$\DD_0\to \und{\Hom}(\DD_0,\TT_{X_0/k}/\DD_0): v \mapsto x\mapsto [v,x]$$
is an isomorphism. Hence, there is a unique global section $v$ of $\DD_0\ot I$, such that $f_v=f$, so the automorphism $\a_v$ sends $A\ot_{\CC}\DD_0$ to $\DD_A$,
as required.
 
%Assume for a moment that it is free, and let $\th\in S_0^-$ be a generator.
%Then we 
\noindent
(ii) By (i) all deformations are trivial. Now the assertion follows from Lemma \ref{Aut-smooth-lemma}.
%It is easy to check that the functor associating with $A$ the group of superconformal automorphisms of $X\times\Spec(A)$
%is a deformation functor. Since we work over $\C$, by [Thm.\ 7.19]\cite{FM}, it is smooth.

%But superconformal automorphisms of the trivial deformation $X_A= are
%$A$-points of the group scheme of superconformal automorphisms of $X$ (here we use projectivity of $X$). 
%Since we work over a field of characteristic zero,
%the latter group scheme is smooth, and our assertion follows.

\noindent
(iii) Let $\phi:S_A\to A$ be the homomorphism corresponding to the NS-puncture, deforming
$\phi_0:S_0\to k$. By (i), we can assume that $S_A=S_0\otimes_{\C} A$. We claim that there exists an automorphism of $S_A$
over $A$, trivial on $S_0$, and compatible with the superconformal structure, 
transforming $\phi_0\ot A$ to $\phi$. 

It is enough to check this assuming that
$\phi=\phi_0\ot A \mod I$, where $\fm I=0$. Then $\phi-\phi_0\ot A$ is given by a $\phi_0$-derivation $S_0\to I$,
i.e., by an $I$-valued tangent vector at $P_0\in X_0$. It remains to extend this tangent vector to an $I$-valued 
superconformal vector field (see Lemma \ref{supeconf-field-pt})
and consider the corresponding automorphism of $X_A$.

The fact about automorphisms follows similarly to (ii) from smoothness of the corresponding group scheme. 

\noindent
(iv) The proof is similar to (ii): the functor associating with $A$ the group of superconformal automorphisms of $A(\!(z)\!)[\th]$ is a deformation
functor. Since we work over $\C$, it is smooth.
\end{proof}

\subsection{Neighborhood of Ramond puncture}\label{R-punct-def-sec}

Let $X_0$ be an affine supercurve over $\C$ with one Ramond puncture $R_0\sub X_0$, and let
$\DD_0\sub \TT_{X_0/k}$ be the structure distribution (so that \eqref{Ramond-D2-isom} is an isomorphism).

%We have formal coordinates $(z,\th)$, so that $R_0$ is given by $z=0$, the distribution on $X_0\setminus R_0$
%is given by $\pa_\th+\th z\pa_z$.

\begin{lemma}\label{Ramond-surj-lem}
The natural morphism of sheaves 
$$\TT_{X_0/k}\to\und{\Hom}_{\OO}(\DD_0,\TT_{X_0/k}/\DD_0): v\mapsto (v_0\mapsto [v,v_0]\mod\DD_0)$$
is surjective with respect to \'etale topology. Hence, we have an exact sequence of sheaves in classical topology
$$0\to \TT^{sc}_{X_0/k}\to \TT_{X_0/k}\to \und{\Hom}_{\OO}(\DD_0,\TT_{X_0/k}/\DD_0)\to 0$$
and the induced map on global sections
$$H^0(X_0,\TT_{X_0/k})\to \Hom_{\OO}(\DD_0,\TT_{X_0/k}/\DD_0)$$
is surjective
\end{lemma}

\begin{proof} Locally in \'etale topology we can assume that $\DD_0$ be generated by $v_0=\pa_\th+\th z\pa_z$ (see Lemma \ref{local-descr-smooth-lem}).
Then $\pa_z$ projects to a basis of $\TT_{X_0/k}$.
Thus, it is enough to check that for every function $a=a(z,\th)$ there exists $v$ with 
$$[v,v_0]=a\pa_z\mod\DD_0.$$
We can represent every $a$ in the form $a=c_0+c_1\th+bz$ for some function $b$ and some constants
$c_0,c_1$. It remains to note that
$$[(c_0\th+c_1)\pa_z+bv_0,v_0]=(\pm c_0+c_1\th+2bz)\pa_z.$$

The last statement is a consequence of the vanishing $H^1(X_0,\TT^{sc}_{X_0/k})=0$ which itself follows from the fact that
$\TT^{sc}_{X_0}$ is isomorphic to a coherent sheaf (see Proposition \ref{inf-aut-smooth-punct-prop}).
\end{proof}

We have the following analog of Lemma \ref{smooth-curve-trivial-def-lem}.

\begin{lemma} (i) Every deformation $(X_A,R_A)$ of $(X_0,R_0)$ over $A\in \Art_{\C}$ is isomorphic to the trivial deformation
$(X_0\times \Spec(A), R_0\times\Spec(A))$, with the distribution induced by that on $X_0$.

\noindent
(ii) Given a surjection $A\to B$ in $\Art_{\C}$, and a deformation $(X_A,R_A)$ of $(X_0,R_0)$ over $A$, any automorphism
of $(X_B,X_R)$ (obtained by the base change to $B$) lifts to an automorphism of $(X_A,R_A)$.
\end{lemma}

\begin{proof}
(i) 
%Let $S_0$ (resp., $S_A$) be the algebra of functions on $X_0$ (resp., $X_A$), and let $\II_0\sub S_0$ (resp.,
%$\II_A\sub S_A$) be the ideal corresponding to the Ramond puncture.
As we have seen in Lemma \ref{smooth-curve-trivial-def-lem}(i),
we can assume that $X_A=X_0\times \Spec(A)$. We have a distribution $\DD_A\sub \TT_{X_A/A}$ deforming $\DD_0$.
%such that the commutator induces an isomorphism
%\begin{equation}\label{Ramond-commutator-condition}
%\DD_A^{\ot 2}\rTo{\sim} \II_A\TT_{X_A/A}.
%\end{equation}
%We want to prove that every $(\DD_A,\II_A)$ 
It is enough to check that $\DD_A$ is obtained from $\DD_0\ot A$
by an automorphism of $X_A$, trivial on $X_0$. 
We can assume that for some ideal $I\sub A$, such that $\fm I=0$, the assertion holds for $A/I$,
so $\DD_A$ corresponds to a homomorphism 
$$f_\DD:\DD_0\to \TT_{X_0/k}/\DD_0\ot I.$$
Now the assertion follows from Lemma \ref{Ramond-surj-lem}.

\noindent
(ii) As before, this follows from (i) and from smoothness of the corresponding group scheme.
\end{proof}
%Family over $k[t,\tau]$ (where $t$ is even, $\tau$ is odd):
%coordinates $(z,\th,t,\tau)$, Ramond puncture $R$ is given by $z-t-\tau\th$, the distribution on $X\setminus R$
%is given by $D=\pa_\th+\th z\pa_z$.
%Formula for $\de$: 
%$$\de(f)=D(f)\cdot [\frac{dz}{z}|d\th].$$

\subsection{Unobstructedness}

Recall that a morphism of deformation functors $\Def_1\to \Def_2$ is called {\it smooth} if for every small extension $B\to A$ in $\Art_{\C}$,
the natural map
$$\Def_1(B)\to \Def_1(A)\times_{\Def_2(A)}\Def_2(B)$$
is surjective.

\begin{lemma}\label{aff-formal-node-smooth-lem}
(i) Let $X_0$ be an affine stable supercurve over $\C$ with a single NS node $q$ and let $\Def(\hat{\OO})$ denote the functor of deformations of
the completion of the local ring of $X_0$ at $q$,
$$\hat{\OO}\simeq \C[\![z_1,z_2]\!][\th_1,\th_2]/(z_1z_2,\th_1z_2,\th_2z_1,\th_1\th_2).$$  
Assume that there exists a function $f_0$ on $X_0$ such that in the formal neighborhood of $q$ one has
$f_0\equiv z_1+z_2\mod (z_1,z_2,\th_1,\th_2)^2$, and the principal open affine $D(f_0)\sub X_0$ coincides with $X_0\setminus\{q\}$.
Then the natural morphism of deformation functors
$$\Def(X_0)\to \Def(\hat{\OO})$$
is smooth.

\noindent
(ii) Similar assertions hold for am affine stable supercurve with a single Ramond node.
\end{lemma}

\begin{proof}
(i) First, let us introduce the following notation. For $A\in \Art_{\C}$ and $t\in \fm_A$, we set
$$\hat{\OO}_{A,t}:=A[\![z_1,z_2]\!][\th_1,\th_2]/(z_1z_2+t^2,z_1\th_2-t\th_1,z_2\th_1+t\th_2,\th_1\th_2).$$

\noindent
{\bf Step 1}. Let $A\in \Art_{\C}$, $t$ be an element such that $t^n=0$. Let us consider a homomorphism of $A$-algebras,
$$\ga:\hat{\OO}_{A,t}\to A(\!(z_1)\!)[\th_1]\oplus A(\!(z_2)\!)[\th_2]: \phi(z_1,z_2,\th_1,\th_2)\mapsto (\phi(z_1,-t^2/z_2,\th_1,t\th_1/z_1), 
\phi(-t^2/z_2,-t\th_2/z_2,z_2,\th_2)).$$
Note that this is well-defined since $t^n=0$. It is also easy to see that $\ga$ is injective,
and the elements
$$(z_1^{-i},z_1^{-i}\th_1,z_2^{-i},z_2^{-i}\th_2)_{i\le 0}$$ 
project to a basis of the quotient by the image of $\ga$, as an $A$-module.
We claim that for any element $f\in \hat{\OO}_{A,t}$ such that $f\equiv z_1+z_2\mod (z_1,z_2,\th_1,\th_2)^2$,
$\ga$ is identified with the localization map
$$\hat{\OO}_{A,t}\to \hat{\OO}_{A,t}[f^{-1}].$$

Indeed, since the multiplication by $f$ is invertible on $A(\!(z_1)\!)[\th_1]\oplus A(\!(z_2)\!)[\th_2]$, it is enough to check that
for every $(p,q)\in A(\!(z_1)\!)[\th_1]\oplus A(\!(z_2)\!)[\th_2]$ one has $f^N(p,q)\in \im(\ga)$ for sufficiently large $N$.
But this follows from the inclusion
$$z_1^nA[\![z_1]\!][\th_1]\oplus z_2^nA[\![z_2]\!][\th_2]\sub \im(\ga).$$
%Next, we need to show the existence of the derivation $\de$. Should use the fact that $\om_

%We also claim that there exists a unique derivation 
Note that this entire picture is obtained by the change of coefficients $\C[T]/(T^n)\to A$ from the standard picture described in Sec.\ \ref{nodes-def-sec}.
It follows that we have a structure derivation
$$\de:\hat{\OO}_{A,t}\to \om_{\hat{\OO}_{A,t}/A},$$
such that after inverting $z_1$ or $z_2$ it is given by the formula \eqref{NS-miniversal-de-formula}.

\noindent
{\bf Step 2}. Let $X_B$ be a deformation of $X_0$ over $B\in \Art_{\C}$. Then for any function $f\in \OO(X_B)$, lifting $f_0$, 
the principal open affine $D(f)\sub X_B$ is the trivial deformation of $D(f_0)=X_0\setminus\{q\}$ over $B$.
Furthermore, there exist an isomorphism of supercurves $D(f)\simeq \Spec(B)\times D(f_0)$, and a superconformal isomorphism
\begin{equation}\label{completion-XB-t-isom}
\hat{\OO}_{X_B,q}\simeq \hat{\OO}_{B,t},
\end{equation}
for some element $t$ of the maximal ideal in $B$, so that we have a cartesian diagram
\begin{diagram}
\OO(X_B)&\rTo{}& B\ot \OO(D(f_0))\\
\dTo{}&&\dTo{r_B}\\
\hat{\OO}_{B,t}&\rTo{\ga}&B(\!(z_1)\!)[\th_1]\oplus B(\!(z_2)\!)[\th_2]
\end{diagram}
Here $r_B$ comes from a certain isomorphism $B\ot \widehat{\OO(D(f_0))}\simeq B(\!(z_1)\!)[\th_1]\oplus B(\!(z_2)\!)[\th_2]$,
where we equip $\OO(D(f_0))$ with the topology using powers of $I_q$, the ideal of a node in $\OO(X_B)$, and consider the completion. 
%in which the lower horizontal arrow is a split embedding of $A$-modules.

Indeed, it is clear that $D(f)$ is a deformation of $D(f_0)$ as a supercurve.
Since $D(f_0)$ is a smooth affine supercurve, by Lemma \ref{smooth-curve-trivial-def-lem}(i) all its deformations are trivial.
The existence of a superconformal isomorphism \eqref{completion-XB-t-isom} follows from Theorem \ref{miniversal-Deligne-thm}.
Finally, by Step 1, we have
$$B(\!(z_1)\!)[\th_1]\oplus B(\!(z_2)\!)[\th_2]\simeq \hat{\OO}_{B,t}[f^{-1}]\simeq \hat{\OO}_{B,t}\ot_{\OO(X_B)}\OO(X_B)[f^{-1}]\simeq
\hat{\OO}_{B,t}\ot_{\OO(X_B)} (B\ot \OO(D(f_0))),$$
which is isomorphic to the completion of $B\ot \OO(D(f_0))$ with respect to the $I_q$-adic topology.

\noindent
{\bf Step 3}. Now let $A\to B$ be a small extension in $\Art_{\C}$. 
%First, since $\OO(D(f_0))$ is a smooth $\C$-algebra,
We claim that there exists a homomorphism 
$$r_A:A\ot \OO(D(f_0))\to A(\!(z_1)\!)[\th_1]\oplus A(\!(z_2)\!)[\th_2]$$
lifting $r_B$ and compatible with derivations. Now for an element
$\wt{t}\in A$ lifting $t\in B$, 
let us define the ring $\wt{\OO}$ as the fibered product
\begin{equation}\label{def-A-fibered-product-diagram}
\begin{diagram}
\wt{\OO}&\rTo{}& A\ot \OO(D(f_0))\\
\dTo{}&&\dTo{r_A}\\
\hat{\OO}_{A,\wt{t}}&\rTo{\ga}&A(\!(z_1)\!)[\th_1]\oplus A(\!(z_2)\!)[\th_2]
\end{diagram}
\end{equation}
Then we claim that $\wt{\OO}=\OO(X_A)$ for some deformation $X_A$ of $X_0$ inducing $X_B$ (so there is also a derivation $\de$ on $X_A$). 

First, we have a superconformal isomorphism
$$\a_0:\widehat{\OO(D(f_0))}\simeq \C(\!(z_1)\!)[\th_1]\oplus \C(\!(z_2)\!)[\th_2].$$
Hence, the map $r_B$ corresponds to the composition of  superconformal isomorphisms
$$B\ot \widehat{\OO(D(f_0))}\rTo{B\ot \a_0} B(\!(z_1)\!)[\th_1]\oplus B(\!(z_2)\!)[\th_2]\rTo{\a_B} B(\!(z_1)\!)[\th_1]\oplus B(\!(z_2)\!)[\th_2]$$
for some $\a_B$. Now we define $r_A$ using the composition of superconformal isomorphisms
$$A\ot \widehat{\OO(D(f_0))}\rTo{A\ot \a_0} A(\!(z_1)\!)[\th_1]\oplus A(\!(z_2)\!)[\th_2]\rTo{\a_A} A(\!(z_1)\!)[\th_1]\oplus A(\!(z_2)\!)[\th_2],$$
where $\a_A$ is some lift of $\a_B$. We know that such a lift exists by Lemma \ref{smooth-curve-trivial-def-lem}(iv).

Since the lower horizontal arrow in \eqref{def-A-fibered-product-diagram} is a split embedding of $A$-modules, so is the upper horizontal arrow.
Hence, $\wt{\OO}$ is a projective $A$-module. Applying the reduction $?\ot_A B$ to our diagram we recover the
cartesian diagram from Step 2, so $\wt{\OO}\ot_A B\simeq \OO(X_B)$. 

Finally, we observe that we have a cartesian diagram of the dualizing sheaves parallel to diagram \eqref{def-A-fibered-product-diagram}.
Since both $r_A$ and $\ga$ are compatible with the derivations, we get a
a well defined derivation $\de:\wt{\OO}\to \om_{\wt{\OO}/A}$.

\noindent
(ii) The proof is similar to (i).
\end{proof}

\begin{prop}\label{global-def-prop} 
Let $(X,P_\bullet,R_\bullet)$ be a stable supercurve with punctures over $\C$. Let $q_1,\ldots,q_m$ be all the nodes of $X$.
%For each $i=1,\ldots,m$, let $\Def(\hat{\OO}_{X,q_i})$ be the functor of deformations of the completion of the local ring at $q_i$. 
Then the morphism of deformation functors
\begin{equation}\label{def-curve-nodes-morphism}
\Def(X,P_\bullet,R_\bullet)\to \prod_{i=1}^m \Def(\OO_{X,q_i})
\end{equation} 
is smooth.
The induced morphism of tangent spaces fits into an exact sequence
\begin{equation}\label{def-curve-nodes-ex-seq}
0\to H^1(X,\AA_{X,P_\bullet,R_\bullet})\to T_{\Def(X,P_\bullet,R_\bullet)}\to \bigoplus_{i=1}^m T_{\Def(\hat{\OO}_{X,q_i})}\to 0.
\end{equation}
\end{prop}

\begin{proof}
Roughly speaking, the idea is that away from the nodes, locally all deformations are trivial, so deformations are classified by appropriate noncommutative $H^1$
sets. Now we use the vanishing of the relevant $H^2$.

Suppose $B\to B/I=A$ is a small extension in $\Art_{\C}$, 
and we are given a deformation $X_A$ of $X$ over $A$ (as a punctured supercurve), and
liftings of the induced deformations of $\hat{\OO}_{X,q_i}$ to deformations of singularities over $B$. To prove smoothness of the morphism \eqref{def-curve-nodes-morphism},
we need to construct a deformation $X_B$ over $B$ lifting $X_A$ and inducing the given deformations of the singularities over $B$.

Let us cover $X_A$ by affine opens $U_{iA}$ such that for $i=1,\ldots,m$, $U_i:=U_{iA}\times_{\Spec(A)}\Spec(\C)$ contains $q_i$ and does not contain other
nodes or punctures, while for $i>m$, $U_i$'s lies in the smooth locus and 
contains at most one R-puncture or NS-puncture.

By Lemma \ref{aff-formal-node-smooth-lem}, for $i=1,\ldots,m$, we can extend the deformation $U_{iA}$ over $A$ to a deformation
$U_{iB}$  over $B$, inducing the given deformation of the singularity at $q_i$.
% and
%such that the induced deformation $U_{iA}$ over $A$ comes from $X_A$. 
By the results of Sec.\ \ref{no-R-punct-def-sec} and \ref{R-punct-def-sec} all the local deformations 
$U_{iA}$, for $i>m$, are trivial, so we can
lift them to trivial deformations $U_{iB}$ over $B$. Furthermore, all the intersections $U_{ij}$ are smooth and contain
no punctures, so all their deformations are trivial, and we have some isomorphisms $\a_{ijB}$ between 
the deformations of $U_{ij}$ over $B$, induced from $U_{iB}$ and $U_{jB}$. 
The induced isomorphisms $\a_{ijA}$ may differ from the ones coming from $X_A$, but we can correct
each $\a_{ijB}$ using Lemma \ref{smooth-curve-trivial-def-lem}(ii) to make sure that $\a_{ijA}$ are the ones coming from $X_A$.

Thus, considering $\a_{kiB}\a_{jkB}\a_{ijB}$ over triple intersections we get a $2$-cocycle with values in 
$I\ot \AA_{X,P_\bullet,R_\bullet}$. Since the corresponding $H^2$ vanishes, this $2$-cocycle is a coboundary.
Hence, we can correct $\a_{ijB}$, so that they become compatible on triple intersections (without changing $\a_{ijA}$).
Thus, our data will give a deformation $X_B$ of $X$ (as a punctured supercurve).

By smoothness, we have surjectivity in sequence \eqref{def-curve-nodes-ex-seq}. It remains to observe that $H^1(X,\AA_{X,P_\bullet,R_\bullet})$
is identified with the space of locally trivial infinitesimal deformations, i.e., those infinitesimal deformations of $(X,P_\bullet,R_\bullet)$ that induce trivial deformations
of the singularities.
\end{proof}

Using the fact that the deformations of nodal singularities of supercurves are smooth (see Corollary \ref{smooth-def-nodes-cor}),
we derive the following

\begin{cor}\label{smooth-def-functor-cor}
The functor $\Def(X,P_\bullet,R_\bullet)$ is smooth.
\end{cor}

%\subsection{Formal universal family}

%As we know, every odd infinitesimal deformation of the nodes is trivial. This implies
%that any odd infinitesimal deformation of a supercurve is locally trivial.
%Hence, we get the following way of constructing a Kuranishi family for a supercurve $X$.

%Assume for simplicity that there are no punctures. 
%Note that even deformations of $X$ are exactly deformations of the underlying usual stable curve with a spin-structure 
%$(C,L)$. First, we can construct an even Kuranishi family $(\CC,\LL)\to B$, where the base is
%a (formal) polydisk. We view this as a supercurve $\XX_0$ over $B$. Then we can replace $B$ by
%$B\times \A^{0|m}$ where $m$ is the dimension of the odd part of $H^1(X,\AA_X)$, and extend the universal
%infinitesimal deformation of $\XX_0$ to a deformation with this base $\A^{0|m}$ (this is always possible).

\section{Square of the relative canonical bundle}\label{square-sec}

In this section we will define the line bundle $\om^2_{X/S}$ on certain families of stable supercurves, as an extension from the smooth locus.
We show that it is useful in finding an ample line bundle over $X$, as well as due to its relation to the sheaf on infinitesimal automorphisms.

%Projective embedding}

%Let $f:X\to S$ be a family of prestable supercurves, and let $X_0$ be the induced family over $S_0=S^{\red}$,
%so that $\OO_{X_0}=\OO_C\oplus L$, where $f_0:C\to S_0$ is a family of prestable curves, $L$ is a coherent sheaf on $C$.
%Then a line bundle $M$ over $X$, such that $f_*M$ is locally free, is relatively very ample if 
%$M|_C$ is relatively very ample and for every point $x\in C_s$,
%one has $H^1(C_s,M|_{C_s}\ot L\ot I_x)=0$.

%Hence, if a line bundle $M$ over $X$ is such that $M|_C$ is relatively ample then $M$ is relatively ample.

\subsection{Local freeness and ampleness}
%Relatively ample line bundle on a family of prestable supercurves}

%Let $f:X\to S$ be a family of prestable supercurves, and let $X_0$ be the induced family over $S_0=S^{\red}$,
%so that $\OO_{X_0}=\OO_C\oplus L$, where $f_0:C\to S_0$ is a family of prestable curves, $L$ is a coherent sheaf on $C$.

It is well known that for a stable curve with punctures $(C,p_1,\ldots,p_n)$ over $S_0$, the line bundle
$\om_{C/S_0}^{log}=\om_{C/S_0}(p_1+\ldots+p_n)$ is relatively ample. 
Our goal is to find an analogous construction for families of stable supercurves.
The problem is that for a stable supercurve $X/S$ the relative dualizing sheaf $\om_{X/S}$ is not necessarily locally free.

We show that this problem can be solved by extending $\om^2_{X/S}$ from the smooth locus.
%Let us first consider the case of stable supercurves without punctures.

Note if $X/S$ is a prestable supercurve and $q\in X_s$ is a node of some fiber then the completion $\hat{\OO}_{X,q}$ gives a formal deformation of the node
singularity in $X_s$ with the base $\hat{\OO}_{S,s}$, hence, it induces a map from the formal neighborhood of $s$ in $S$ to the base of the miniversal deformation of the node
(described in Theorem \ref{miniversal-Deligne-thm}). 
%In other words, this deformation comes from an element $t\in \hat{\OO}_{S,s}$ ???

\begin{theorem}\label{can-square-thm} 
Let $f:X\to S$ be a family of prestable supercurves inducing a smooth morphism to the base of the miniversal deformation of every node
in a fiber. Let $j:U\hra X$ be the open complement to the locus of nodes in fibers. Then for any $n\in \Z$, the sheaf
$\om_{X/S}^{2n}:=j_*\om_{U/S}^{\ot 2n}$ is a line bundle on $X$. For a morphism $u:S'\to S$ such that the induced family $X'\to S'$
also satisfies the above assumption, there is a natural isomorphism
\begin{equation}\label{om-2n-base-change-eq}
v^*\om_{X/S}^{2n}\rTo{\sim} \om_{X'/S'}^{2n},
\end{equation}
where $v:X'\to X$ is the induced morphism. In the case when $S$ is even, and $(C,L)$ is the underlying family of curves with spin structures, we have
$$\om_{X/S}^{2n}|_C\simeq \om_{C/S}^n\oplus \om_{C/S}^n\ot L=p^*\om_{C/S}^n,$$
where $p:X\to C$ is the natural projection.
In particular,
$$\om_{X/S}^{2n}|_C\simeq \om_{C/S}^n.$$
\end{theorem}

\begin{proof}
We claim that under our assumptions $X$ is Cohen-Macaulay near the nodes. Indeed, by assumption, $S$ is smooth near
every point $s$ with singular fiber $X_s$. Since the morphism $f$ is Cohen-Macaulay, the claim follows from
\cite[Lem.\ 37.21.4]{Stacks}. 
Since the complement to $U$ has codimension $\ge 2$ in $X$, by Lemma \ref{CM-str-sh-lem}(i),
we see that the natural map 
$$\OO_X\to j_*\OO_U$$
is an isomorphism. We will use this fact below.

Let us first prove that $j_*\om_{U/S}^{\ot 2}$ is a line bundle 
for the miniversal deformations of the nodes, i.e., when $S=\Spec k[t]$ and $X$ is given as in 
\eqref{univ-def-node-I} and \eqref{univ-def-node-II}.

For a Ramond node, $\om_{X/S}$ is a line bundle, hence, $\om_{X/S}^{\ot 2}$ is also a line bundle, and
$$\om_{X/S}^{\ot 2}\simeq j_*\om_{U/S}^{\ot 2}.$$

For an NS node, the total space $C$ of the underlying usual family of curves is the quadratic cone
$xy=-t^2$, and $U^{\red}$ is the complement to the singular point $q$ in $C$. Furthermore, we have
$$\OO_X=\OO_C\oplus L, \ \ \om_{X/S}\simeq L\oplus \om_{C/S},$$
where $L$ is a CM-sheaf of rank $1$ over $C$ (see Lemma \ref{even-case-rem}). Over $U$ we have
$$\om_{U/S}^{\ot 2}\simeq \om_{U^{\red}/S}\oplus \om_{U^{\red}/S}\ot L|_U.$$
We can compute the push-forward $j_*$ componentwise. Since the complement of $U$ has codimension $\ge 2$,
$\om_{C/S}$ is locally free and $L$ is a CM-sheaf with full support, by Lemma \ref{CM-str-sh-lem}, we get
$$j_*\om_{U/S}^{\ot 2}\simeq \om_{C/S}\oplus \om_{C/S}\ot L,$$
which is just the pull-back of $\om_{C/S}$ under the natural projection $X\to C$.
Hence, $j_*\om_{U/S}^{\ot 2}$ is locally free of rank $1$.

To prove the result in the general case, i.e., for an arbitary family $X/S$ inducing a smooth morphism to the base of the miniversal deformation of every node, 
we observe that the question is local, and by Lemma \ref{etale-local-nodes-lem}, an \'etale neighborhood $B$ of $X$ will have a smooth morphism $t:B\to X_0$
to the miniversal deformation of the node $X_0/S_0$, so that $B\cap U=t^{-1}(U_0)$, where $U_0\sub X_0$ is the complement to the node. Then the base change morphism
$$t^*j_{0*}\om_{U_0/S_0}^2\to j_*\om_{U/S}^2$$
is an isomorphism, hence, $j_*\om_{U/S}^2$ is locally free. 
The isomorphism \eqref{om-2n-base-change-eq} follows from this and from the compatilibity of the base change morphisms for the maps to $X_0$
from \'etale neighborhoods of nodes in $X'$ and in $X$.
\end{proof}

Now we recall that for every NS-puncture $P_i$,
there is a canonical divisor $D_i$ with the same support (see Sec.\ \ref{NS-correspondence-sec}). So, we can define the line bundle on $X$,
$$\LL_{X/S}:=\om_{X/S}^{\ot 2}(\sum D_i+\sum R_j).$$

Let $f:X\to S$ be a proper morphism of Noetherian superschemes. 

\begin{definition}
We say that a line bundle (of rank $1|0$) $L$ over $X$ is {\it strongly relatively ample over } $S$ 
%(or {\it $f$-ample})
if there exists $n>0$, 
%such that the map 
%$$X\to \Proj_S(S^\bullet(f_*L^n))$$
%is regular and is a closed embedding.
a supervector bundle $\EE$ over $S$ and a closed embedding $\phi:X\to \P(\EE)$ over $S$,
such that $L^n\simeq\phi^*\OO(1)$.
%We say that $L$ is {\it relatively ample} if the above condition holds locally over $S$.
\end{definition}

For a superscheme $S$ we denote by $S_{\bos}$ the usual scheme with the same underlying topological space as $S$ and with the sheaf of rings
$\OO_S/\NN_S$, where $\NN_S$ is the ideal generated by odd functions.

\begin{theorem}\label{ample-thm} 
Let $f:X\to S$ be a family of stable supercurves inducing a smooth morphism to the base of the miniversal deformation of every node
in a fiber. Then the line bundle $\LL_{X/S}$ is strongly relatively ample.
\end{theorem}

\begin{proof} Let $C\to S_0$ be the underlying usual family of stable curves over $S_0=S_{\bos}$.
Since $f$ is flat, by Proposition \ref{ample-crit-prop} of the Appendix, it suffices to prove that $L:=\LL_{X/S}|_C$ is strongly relatively ample over $S_0$.
Since $L$ is a line bundle, the natural map 
$$L_{X/S}\to j_{\bos *}j_{\bos}^*L_{X/S}$$ 
is an isomorphism, where $j_{\bos}:U_{\bos}\to C$ is the embedding of the complement to the nodes.
We have 
$$j_{\bos}^*L_{X/S}\simeq \om_{U_{\bos}/S_0}(\sum D_{i \bos}+\sum R_{j \bos}).$$
Hence,
$$L_{X/S}\simeq \om_{C/S_0}(\sum D_{i \bos}+\sum R_{j \bos}),$$
which is relatively ample by a classical result on stable curves (see \cite[Thm.\ 1.2]{DM}, \cite[Lem.\ 6.1]{ACG}).
\end{proof}

\subsection{The sheaf of infinitesimal symmetries}

Now we can extend Proposition \ref{inf-aut-smooth-punct-prop} to the case of stable supercurves.

\begin{theorem}\label{inf-aut-thm}
Let $(X,P_\bullet,R_\bullet)\to S$ be a family stable supercurve with punctures, where $(P_i)_{i\in I}$ are NS punctures
and $(R_j)_{j\in J}$ are Ramond punctures, inducing a smooth morphism to the base of the miniversal deformation of every node in a fiber. 
Then one has a natural isomorphism
$$\AA_{(X,P_\bullet,R_\bullet)/S}\simeq \om_{X/S}^{-2}(-\sum_{i\in I} D_i-2\sum_{j\in J} R_j),$$
where $\om_{X/S}^{-2}$ is the line bundle defined in Theorem \ref{can-square-thm} 
and $D_i\sub X$ is the divisor associated with the NS puncture $P_i$ (see Sec.\ \ref{NS-correspondence-sec}).
\end{theorem}

\begin{proof}
Over the smooth locus this holds by Proposition \ref{inf-aut-smooth-punct-prop}. 
%we have an isomorphism $\AA_{U/S}\simeq \TT_{U/S}/\DD\simeq \om_{U/S}^{-2}$.
%Also, the fact that the composition 
%$$\AA_{U/S}\to \TT_{U/S}\to \TT_{U/S}/\DD$$
%is an isomorphism easily implies that the compositions
%$$\AA_U\to \TT_U\to \TT_U/\DD,$$
%$$\AA_{U,U_0}\to \TT_{U,U_0}\to \TT_{U,U_0}/\DD$$
%are also isomorphisms.
Now let us show that the natural map
$$\AA_{X/S}\to j_*\AA_{U/S},$$
%, \ \ \AA_X\to j_*\AA_U, \ \ \AA_{X,X_0}\to j_*\AA_{U,U_0}$$
where $j:U\to X$ is the embedding of the smooth locus, is an isomorphism. 
Since the condition on $v\in \TT_{X/S}$ to belong to 
%$\AA_X$ (resp., $\AA_{X,X_0}$, resp., 
$\AA_{X/S}$ can be imposed only over the smooth locus,
it is enough to check that the natural map
$$\TT_{X/S}\to j_*\TT_{U/S}$$ 
is an isomorphism. But this follows immediately from the fact that $j_*\TT_{U/S}$ can be identified with relative
derivations of $j_*\OO_U$ and the isomorphism $j_*\OO_U\simeq \OO_X$.
\end{proof}

\begin{remark}
Note that neither $\TT_{X/S}$ nor $\AA_{X/S}$ are compatible with the base change. So in the situation of Theorem \ref{inf-aut-thm} the restriction of
$\AA_{X/S}$ to a fiber $X_s$, which is a stable supercurve with at least one node, is a line bundle on $X_s$, whereas $\AA_{X_s}$ is not
(since $T_C$ is not locally free for a nodal curve $C$).
\end{remark}

%\subsection{Hilbert scheme}
%
%Should go over Hilbert schemes, $\Hom$ and $\Isom$-schemes.
%
%In partucular, get representability of $\Aut$ by a group scheme.

\section{Proof of Theorem A}\label{proof-thm}

To prove Theorem $A$ we need to check that $\ov{\SS}_{g,n_{NS},n_R}$ is a stack with representable diagonal, with an \'etale atlas,
and that it is smooth and proper over $\C$.

\subsection{$\ov{\SS}_{g,n_{NS},n_R}$ is a limit preserving stack}

A family of stable supercurves is given by a superscheme $X\to S$, smooth of dimension $1|1$,
together with a morphism $\de:\Om_{X/S}\to \om_{X/S}$ and sections $p_i:S\to X$ (the R-punctures can
be recovered from $\de$). The fact that isomorphisms between two families over $S$ form a sheaf in \'etale topology
 is proved in the standard way (and also follows from the representability proved below). 
 The \'etale descent for such families follows as in the classical case from the existence of the relatively ample bundle which we proved
in Theorem \ref{ample-thm} (see  \cite[Sec.\ 4.3.3]{FGA-exp}; see also \cite[Sec.\ 7.2.1]{MZ}). 

The fact that our stack $\MM$ is limit preserving, i.e., if $S=\Spec(A)$
with $A=\varinjlim_i A_i$, then $\MM(A)\simeq \varinjlim_i \MM(A_i)$, is proved essentially by the same arguments as in the case of the
moduli stack of usual curves, see \cite[Lem.\ 7.23]{MZ} for details.

\subsection{Representability of the diagonal}

We need to check that for a pair of families of stable supercurves
$X\to S$ and $Y\to T$, there is an algebraic space locally of finite type over $\C$, classifying isomorphisms
between $X_s$ and $Y_t$ as abstract superschemes. Passing to the induced families over $S\times T$, we can assume that
we have two families $X\to S$ and $Y\to S$ over the same base, and we need to check
that the functor $\Isom(X,Y)$ is representable by an algebraic space over $S$. In fact, we will check that it is representable by a superscheme over $S$.

For this we can use relative projectivity of the morphisms $X\to S$ and $Y\to S$ which holds by Theorem \ref{ample-thm}
and the standard approach via the Hilbert (super)schemes (the construction of Hilbert superschemes for 
projective morphisms of superschemes is discussed in \cite[Sec.\ 4]{Hilbert-schemes} and in \cite[Sec.\ 7]{MZ}). 
Namely, as in the classical case, the idea is that to an isomorphism $X_s\to Y_s$ 
we can associate its graph, which is a closed subscheme in $X_s\times Y_s$.
In more detail, first, let $\HH$ be the relative Hilbert superscheme over $S$ parametrizing subschemes in $X_s\times Y_s$
(with the same Hilbert series as the Hilbert series of $X_s$ with respect to a large power of $\LL_{X/S}|_{X_s}$). 
Let 
$$Z\sub X\times_S\times Y\times_S \HH$$
be the universal subscheme.
Let us consider the projections
$$p_X:Z\to X\times_S \HH, \ \ p_Y:Z\to Y\times_S \HH.$$
Then there is a universal open subscheme $\HH_1\sub \HH$ 
over which $p_X$ and $p_Y$ become isomorphisms (this is proved as in \cite[Thm.\ 5.22]{FGA-exp}). 
Then it is easy to see that $\HH_1$ represents $\Isom_{\sSch/S}(X,Y)$.

Next, assume that our stable supercurves $X$ and $Y$ over the same base are equipped with the matching number of punctures of each type,
$(P_i^X, R_j^X)$ and $(P_i^Y, R_j^Y)$.
Let $Z\sub X\times_S\times Y\times_S\HH_1$ be the graph of the universal isomorphism.
Then we have the induced sections $Z\times_X P_i^X$ and $Z\times_Y P_i^Y$ of the family $Z\to \HH_1$,
and induced divisors $Z\times_X R_j^X$ and $Z\times_Y R_j^Y$ in $Z$, which are smooth of dimension $0|1$ over $\HH_1$. Let us set $P_X$ (resp.,$P_Y$) be the subscheme of $Z$ obtained as the disjoint union of these subschemes pulled back from $X$ (resp., $Y$). Then there exists the largest closed subscheme
$\HH_2\sub\HH_1$ such that $P_X$ and $P_Y$ coincide over $\HH_2$.
This is a consequence of the following easy Lemma.

\begin{lemma}\label{vanishing-closed-subscheme-lem}
(a) Let $f:F_1\to F_2$ be a morphism of coherent sheaves over a superscheme $Z$, and let $Z\to S$ be a morphism. 
Assume that $F_1$ and $F_2$ are flat over $S$. Then there exists the largest closed subscheme
$T\sub S$ such that $f=0$ over $T$.

\noindent
(b) Let $R$ and $R'$ be subschemes in a superscheme $Z$, and let $Z\to S$ be a flat morphism such that both $R$ and $R'$ are flat over $S$. 
Then there exists the largest closed subscheme
$T\sub S$ such that $R=R'$ over $T$.
\end{lemma}

\begin{proof}
(a) Let $Z_f\sub Z$ be the closed subscheme corresponding to the ideal $\ker(f:\OO_Z\to \und{\Hom}(F_1,F_2))$. Then $T$ is the schematic image of $Z_f$.

\noindent
(b) We apply (a) to the morphisms $I_R\to \OO_Z/I_{R'}$ and $I_{R'}\to \OO_Z/I_R$ and observe that these morphisms vanish if and only if $R=R'$.
\end{proof}

Finally, we can define a closed subscheme $\HH_{sc}\sub \HH$ that corresponds to superconformal isomorphisms.
To this end we observe that we can pull back to $Z$ both morphisms $\de_X:\Om_{X/S}\to \om_{X/S}(R_X)$ and $\de_Y:\Om_{Y/S}\to \om_{Y/S}(R_Y)$, 
and then apply Lemma \ref{vanishing-closed-subscheme-lem} to the difference.
Then $\HH_{sc}$ represents the sheaf of isomorphisms between families $X$ and $Y$ as stable supercurves with punctures.

\subsection{Construction of an \'etale atlas}

The main point is to use the existence of the (purely even) Deligne-Mumford stack $\ov{\SS}=\ov{\SS}_{g,n_{NS},n_R}$ parametrizing 
stable curves with spin-structures (and punctures), constructed in \cite{AJ}.

For every stable supercurve $X_0$ over $\C$, we can find a family $\pi:\XX_0\to B_0$ of stable supercurves (with punctures) over an affine even base $B_0$,
and a $\C$-point $b\in B_0$ such that we get $X_0$ as a fiber of $\XX_0$ over $b$, 
such that the corresponding map $B_0\to \ov{\SS}$ is \'etale.

Let us consider the corresponding bundle over $B_0$,
$$\EE:=R^1\pi_*(\AA^-)=R^1\pi_*(\om_{\CC/B_0}^{-1}\ot \LL(-\sum p_i-\sum r_i)),$$
where $\CC\sub \XX_0$ is the corresponding usual curve over $B_0$ with the relative spin-structure $\LL$
(recall that $\AA^-$ denotes the sheaf of odd infinitesimal automorphisms of $\XX_0/B_0$). Note that the fact that $\EE$ is
locally free follows from the vanishing of $\pi_*(\AA^-)$, i.e., from the absence of odd infinitesimal automorphisms of $\XX_0/B_0$.

We define the superbase $B$ as
$$B=\Spec_{B_0}({\bigwedge}^\bullet \Pi\EE^\vee).$$
In other words, we view $\EE^\vee$ as extra odd coordinates.
Our goal is to extend $\XX_0$ to a family $\XX\to B$ (possibly after changing $B_0$ to an \'etale neighborhood of $b$).

Let $q_1,\ldots,q_N$ be a finite number of points in $C_0$, including all nodes, $q_1,\ldots,q_n$,
such that there is an ample effective divisor $D$ supported on $\{q_1,\ldots,q_N\}$, and $q_i$ are distinct from all the punctures.
Let also $q'_1,\ldots,q'_m$ be a set of smooth points in $C_0$, distinct from all the punctures and from $q_i$'s, such that
$D'=q'_1+\ldots+q'_m$ is ample. 

We claim that replacing $B_0$ by an \'etale neighborhood of $b$, we can assume
the existence of relatively ample effective Cartier divisors $\DD$ and $\DD'$ in $\CC_0$, extending $D$ and $D'$,
such that $\DD\cap \DD'=\emptyset$. 

Therefore, we have an affine cover of $\XX_0$ by two open sets, $\UU_0$ defined as the complement to $D'$ and $\UU_1$ defined as the complement to $D$.
Now, let us consider the tautological cohomology class
$$[c]\in H^0(B_0,\EE^\vee\ot R^1\pi_*(\AA^-))\simeq H^1(\XX_0,\pi^*\EE^\vee\ot \AA^-).
$$
We can represent it by a Cech cocycle 
$$c\in H^0(\UU_0\cap \UU_1,\pi^*\EE^\vee\ot \AA^-).$$

Now we observe that there is a natural morphism of sheaves
$$\om_{\CC/B_0}^{-1}\ot \LL(-\sum p_i-\sum r_i)=\AA^-\to \TT_{\XX_0/B_0,P_\bullet,R_\bullet}^-.$$
Hence, we can view odd sections of $\pi^*\EE^\vee\ot \AA^-$ as an even derivation 
$$\OO_{\XX_0}\to \Pi\pi^*\EE^\vee$$ 
(relative to $B_0$). Thus, $c$ gives rise to such a derivation over $\UU_0\cap \UU_1$.
Hence, we have a homomorphism 
$$\exp(c):\OO\to {\bigwedge}^\bullet_{\OO} \Pi\pi^*\EE^\vee,$$
over $\UU_0\cap \UU_1$, reducing to identity modulo ${\bigwedge}^{\ge 1}$. 
We can extend $\exp(c)$ to the automorphism $\wt{c}$ of
${\bigwedge}^\bullet_{\OO_{\XX_0}} \Pi\pi^*\EE^\vee$, identical on $\pi^{-1}{\bigwedge}^\bullet_{\OO_{B_0}}\Pi\EE^\vee$.
%$\OO_B$. 
Now we define $\XX$ by gluing 
$$\UU_i\times_{B_0} B\simeq \Spec({\bigwedge}^\bullet_{\OO} \Pi\pi^*\EE^\vee|_{\UU_i}), \ \ i=0,1,$$ 
using $\wt{c}$ as an automorphism of $(\UU_0\cap\UU_1)\times_{B_0} B$.

Note that since on a smooth locus $c$ acts by a derivation preserving the structure distribution $\DD\sub \TT$,
the same is true for $\wt{c}$. Hence, $\XX$ has a natural structure of a stable supercurve over $B$ (see Lemma \ref{isom-stable-lem}).

We claim that the corresponding map from $B$ to the moduli stack $\SS\MM$ of supercurves is \'etale near $b$.
Indeed, it is enough to check that the induced map on tangent spaces at $b$ is an isomorphism.
The tangent space $T_bB$ to $B$ at $b$ is given by
$$(T_bB)^+=T_bB_0, \ \ (T_bB)^-=\EE_b=H^1(X_0,\AA^-_{X_0}).$$
The fact that the map on even tangent spaces is an isomorphism follows from the assumption that $B_0\to \SS$ is \'etale.
The map on odd tangent spaces corresponds to the natural map
$$H^1(X_0,\AA^-_{X_0})\to (T_b\SS\MM)^-.$$
The fact that it is an isomorphism follows from the exact sequence \eqref{def-curve-nodes-ex-seq} since the last term of this sequence is purely even.

\subsection{Properties of the stack of stable supercurves}

Corollary \ref{smooth-def-functor-cor} shows that the stack of stable supercurves is smooth. Since it is of finite type (for each fixed genus and fixed number of punctures),
the fact that it is proper can be checked for its even part, i.e., for the stack of stable curves with spin structures. But this is known due to works of Cornalba \cite{Cornalba}
and Jarvis \cite{Jarvis}.

\section{Kodaira-Spencer map}\label{KS-map-sec} 

In this section we will study the behavior of the Kodaira-Spencer map in degenerating families of supercurves where the limiting curve acquires one NS or Ramond 
node. This will later help us to calculate the canonical line bundle of the moduli stack of stable supercurves.
We begin with the classical case of a degenerating family of usual curve and then consider separately the cases of NS and Ramond nodes.

\subsection{Classical case}\label{KS-class-sec}

Let $\pi:C\to S$ be a family of stable curves over a smooth affine base $S$, equipped with a smooth morphism $t:S\to \A^1$.
%$S=\Spf(\C[\![t,t_1,\ldots,t_d]\!])$ with one special coordinate $t$.
We denote by $S_0\sub S$ the divisor $t=0$ and by $\pi_0:C_0\to S_0$ the induced family over $S_0$.
We assume that there is a section $q:S_0\to C_0$ such that the map $C\to S$ is smooth away from $q(S_0)$, and that the structure sheaf of the
completion of $C$ along $q(S_0)$
is isomorphic to $\OO_S[\![x,y]\!]/(xy-t)$ (so that the section $q$ corresponds to $x=y=t=0$). 
Below we will write simply $q$ to denote the relative node $q(S_0)\sub C_0$.
%which is a nodal curve with one node $q$, and such that the generic fiber is smooth.
%We assume in addition that the completion of $C$ at $q$ is isomorphic to $k[\![x,y,t]\!]/(xy-t)$.

We would like to discuss the Kodaira-Spencer map for such a family.

Let us denote by $\TT_{S,S_0}$ the sheaf of derivations of $\OO_S$ preserving the ideal generated by $t$.
Let also $\TT_{C,C_0}$ denote the sheaf of derivations of $\OO_C$ preserving the divisor $C_0$, i.e., 
%Locally sections of $\TT_{C,C_0}$ correspond to derivations 
preserving the ideal generated by $t$ in $\OO_C$.
Finally, we denote by $\TT_{C/S}\sub \TT_{C,C_0}$ the relative tangent sheaf.

Note that $\pi^*\TT_S$ can be identified with derivations $\pi^{-1}\OO_S\to\OO_C$ (since $\Om_S$ is locally free). 
The subsheaf $\pi^*\TT_{S,S_0}$ corresponds to derivations $\pi^{-1}\OO_S\to \OO_C$ sending $t$ to $\OO_C\cdot t$.
Thus, we have a natural morphism $\TT_{C,C_0}\to \pi^*\TT_{S,S_0}$ sending a derivation of $\OO_C$ to its restriction to $\pi^{-1}\OO_S$.
Similarly, we have a map $\TT_{C_0}\to \pi_0^*\TT_{S_0}$ sending a derivation of $\OO_{C_0}$ to its restriction to $\pi^{-1}\OO_{S_0}$.

\begin{lemma}\label{KS-even-lem} (i) There are exact sequences
\begin{equation}\label{tangent-ex-seq}
0\to \TT_{C/S}\to \TT_{C,C_0}\to \pi^*\TT_{S,S_0}\to 0
\end{equation}

\begin{equation}\label{tangent-C0-ex-seq}
0\to \TT_{C_0/S_0}\to \TT_{C_0}\to \pi_0^*\TT_{S_0}\to 0
\end{equation}

\noindent
(ii) For the closed point $x_0\sub S_0$, the natural map
\begin{equation}\label{T-C0-S0-x0-map}
\TT_{C_0/S_0}|_{C_{x_0}}\to \TT_{C_{x_0}}
\end{equation}
is an isomorphism.

\noindent
(iii) Let $j:C\setminus\{q\}\to C$ denote the open embedding. Then the natural map $\om_{C/S}\to j_*\om_{C\setminus\{q\}/S}$ is an isomorphism.
The corresponding natural map
$$\Om_{C/S}\to j_*\Om_{C\setminus\{q\}/S}\simeq j_*\om_{C\setminus\{q\}/S}\simeq \om_{C/S}$$
is injective with the cokernel isomorphic to $\OO_q$.
Hence, the dual map 
$$\om_{C/S}^{-1}\to \TT_{C/S}$$ 
is an isomorphism.
Also, the coherent sheaf $\TT_{C,C_0}$ on $C$ is locally free, and the sheaf $\TT_{C_0/S_0}$ on $C_0$ is flat over $S$.

\noindent
(iv) Let us consider the map 
$$\TT_C\to \OO_C/(t)=i_*\OO_{C_0},$$
where $i:C_0\to C$ is the natural embedding, 
sending $v$ to $v(t)\mod (t)$.
Then its image is $i_*\II_q$, where $\II_q\sub \OO_{C_0}$ is the ideal sheaf of $q$, so we have
an exact sequence
\begin{equation}\label{tangent-C-C0-ex-seq}
0\to \TT_{C,C_0}\to \TT_C\to i_*\II_q\to 0
\end{equation}

\noindent
(v) There is an exact sequence
$$0\to \TT_C/\TT_{C,C_0}\rTo{t} \TT_{C,C_0}/t\TT_{C,C_0}\to i_*\TT_{C_0}$$
\end{lemma}

\begin{proof}
(i) For the first sequence, we have to check that the map $\TT_{C,C_0}\to \pi^*\TT_{S,S_0}$ is surjective. This is clear away from $q$.
It remains to check this in the formal neighborhood of $q$. But we can extend the derivation $t\partial_t$ to the derivation of $k[\![x,y,t]\!]/(xy-t)$ induced by
$x\partial_x+t\partial_t$.

It follows that the composition 
$$\TT_{C,C_0}\to \pi^*\TT_{S,S_0}\to i_*\pi_0^*\TT_{S_0}$$
is surjective, which implies surjectivity of the map $\TT_{C_0}\to \pi_0^*\TT_{S_0}$, and hence, exactness of the second sequence.

\noindent
(ii) This follows easily by a local computation near the node. Namely, locally a derivation $v$ in $\TT_{C_0/S_0}$ is described
by a pair of functions $v(x)$ and $v(y)$, satisfying $v(x)y+xv(y)=0$, or equivalently $v(x)\in (x)$, $v(y)\in (y)$, so $\TT_{C_0/S_0}$
is locally isomorphic to 
$$\II_q\simeq x\cdot\OO_{C_0}\oplus y\cdot \OO_{C_0}.$$ 
Similarly $\TT_{C_{x_0}}$ is isomorphic to 
$$\II_{q_0}\simeq x\cdot \OO_{C_{x_0}}\oplus y\cdot \OO_{C_{x_0}},$$ 
where $q_0$ is the node in $C_{x_0}$. 
The natural map \eqref{T-C0-S0-x0-map} corresponds under these isomorphisms to the natural map
$$(x\cdot\OO_{C_0}\oplus y\cdot \OO_{C_0})|_{\OO_{C_{x_0}}}\to x\cdot \OO_{C_{x_0}}\oplus y\cdot \OO_{C_{x_0}},$$
which is an isomorphism.

\noindent
(iii) The first assertion follows from the fact that $\om_{C/S}$ is locally free.
The injectivity of $\Om_{C/S}\to \om_{C/S}$ only has to be checked in a formal neighborhood of $q$.
Then we can identify $\Om_{C/S}$ with $(\OO dx\oplus \OO dy)/(xdy+ydx)$.
It is easy to see that any element of $\Om_{C/S}$ can be written uniquely as
$$f(t,x,y)dx+ g(y)dy.$$
The map $\Om_{C/S}\to \om_{C/S}$ is given by $dx\mapsto x\cdot \bb$, $dy\mapsto -y\cdot \bb$, where $\bb$ is a generator of $\om_{C/S}$.
Hence, it sends the above element to 
$$([f(t,x,y)x-g(y)y)\cdot \bb,$$
which is zero only if $f=g=0$. 
The last assertion follows by dualizing the sequence 
$$0\to \Om_{C/S}\to \om_{C/S}\to \OO_q\to 0$$
and using the vanishing $\und{\Hom}(\OO_q,\OO_C)=\und{\Ext}^1(\OO_q,\OO_C)=0$.

The fact that $\TT_{C,C_0}$ is locally free follows from the exact sequence \eqref{tangent-ex-seq}.
The sheaf $\TT_{C_0/S_0}$ is locally free away from the node. Thus, to check its flatness it is enough to consider it in the
formal neighborhood of the node. Then it can be identified with the subsheaf of $\OO_{C_0}\oplus \OO_{C_0}$
consisting of $(f,g)$ such that $f\in(x)$ and $g\in (y)$, so it is locally free as $\OO_{S_0}$-module.

\noindent
(iv) Since the map $\pi:C\to S$ is smooth away from $q$, it is enough to check the assertion in the formal neighborhood of the node.
Then we can extend $x,y$ to formal coordinates $(x,y,t_1,\ldots,t_d)$ on $C$. Since $t=xy$, the map in question $\TT_C\to \OO_C/(t)$ is given by 
$$f\partial_x+g\partial_y+\sum f_i\partial_{t_i}\mapsto fy+gx \mod (t),$$
so its image is $\II_q\sub \OO_{C_0}$.

\noindent
(v) This follows easily from the fact that the kernel of the projection $\TT_{C,C_0}\to i_*\TT_{C_0}$ is $t\TT_C\sub \TT_{C,C_0}$.
\end{proof}

It follows that we can rewrite the exact sequence \eqref{tangent-ex-seq} as
\begin{equation}\label{tangent-ex-seq-bis}
0\to \om^{-1}_{C/S}\to \TT_{C,C_0}\to \pi^*\TT_{S,S_0}\to 0.
\end{equation}
Hence, applying the functor $R\pi_*(\cdot)$ and using the isomorphism $\pi_*\OO_C\simeq \OO_S$, we get the induced coboundary map
\begin{equation}\label{KS-even-map}
KS:\TT_{S,S_0}\to R^1\pi_*\om^{-1}_{C/S}
\end{equation}
which we call the {\it Kodaira-Spencer map} for our family, since it restricts to the usual Kodaira-Spencer map over $S\setminus S_0$.
Note that we also have a similar map coming from the sequence \eqref{tangent-C0-ex-seq},
\begin{equation}\label{KS0-even-map}
KS_0:\TT_{S_0}\to R^1\pi_{0*}\TT_{C_0/S_0}.
\end{equation}
Let $x_0\in S_0$ denote the origin, and let $C_{x_0}$ be the corresponding curve with one node. 
Then by Lemma \ref{KS-even-lem}(ii), we have $\TT_{C_0/S_0}|_{C_{x_0}}\simeq \TT_{C_{x_0}}$,
so the exact sequence \eqref{tangent-C0-ex-seq} restricts to the standard exact sequence associated with the
embedding $C_{x_0}\hookrightarrow C_0$ (with the trivial normal bundle):
$$0\to \TT_{C_{x_0}}\to \TT_{C_0}|_{C_{x_0}}\to T_{x_0}S_0\ot \OO_{C_{x_0}}\to 0.$$
We can look at the corresponding coboundary map
$$\kappa_{x_0}: T_{x_0}S_0\to H^1(C_{x_0},\TT_{C_{x_0}}).$$

\begin{prop}\label{KS-even-prop} 
%Assume that the arithmetic genus $g$ is $\ge 2$ and $C_0$ is stable.
Assume that the map $\kappa_{x_0}$ is injective. 
Then $R^1\pi_*\om^{-1}_{C/S}$ is a vector bundle on $S$
and the map \eqref{KS-even-map} is an embedding of a subbundle.
\end{prop}

\begin{proof}
Since $C_{x_0}$ is a stable curve, we have $H^0(C_{x_0},\om^{-1})=0$, which implies by semicontinuity 
that $\pi_*\om^{-1}_{C/S}=0$, and by the base change theorem that
$R^1\pi_*\om^{-1}_{C/S}$ is locally free.
Thus, we have a long exact sequence
$$0\to \pi_*\TT_{C,C_0}\to \TT_{S,S_0}\rTo{KS} R^1\pi_*\om^{-1}_{C/S}\to R^1\pi_*\TT_{C,C_0}\to R^1\pi_*\OO_C\ot \TT_{S,S_0}\to 0.$$
Note that $R^1\pi_*\OO_C$ is locally free of rank $g$. Also $\TT_{C,C_0}$ is a vector bundle on $C$ as follows from the exact sequence
\eqref{tangent-ex-seq}. Thus, if we prove that $H^0(C_{x_0},\TT_{C,C_0}|_{C_{x_0}})=0$ then it would follow that $\pi_*\TT_{C,C_0}=0$ and
$R^1\pi_*\TT_{C,C_0}$ is locally free, and our assertion would follow.

Let us consider the morphism of exact sequences on $C_0$,
\begin{diagram}
0&\rTo{}&\TT_{C/S}|_{C_0}&\rTo{}&\TT_{C,C_0}|_{C_0}&\rTo{}&\pi_0^*\TT_{S,S_0}|_{S_0}&\rTo{}& 0\\
&&\dTo{}&&\dTo{}&&\dTo{}\\
0&\rTo{}&\TT_{C_0/S_0}&\rTo{}&\TT_{C_0}&\rTo{}&\pi_0^*\TT_{S_0}&\rTo{}&0
\end{diagram}
Since the rightmost terms are locally free, the restrictions of these sequences to $C_{x_0}$ are still exact, so we get
a commutative square of the coboundary maps
\begin{diagram}
\TT_{S,S_0}|_{x_0}&\rTo{\kappa'_{x_0}}& H^1(C_{x_0},\om_{C_{x_0}}^{-1})\\
\dTo{r}&&\dTo{}\\
T_{x_0}S_0&\rTo{\kappa_{x_0}}& H^1(C_{x_0},\TT_{C_{x_0}})
\end{diagram}
We claim that the map $\kappa'_{x_0}$ is injective. Indeed, by assumption, $\kappa_{x_0}$ is injective.
Also, the map $r$ is surjective with the $1$-dimensional kernel spanned by $t\partial_t$.
Thus, our claim reduces to the assertion that the restriction of $\kappa'_{x_0}$ to $\ker(r)$ is injective.

To prove the last assertion let us consider the embedding $D\sub S$ of the formal $1$-dimensional disk with the coordinate $t$, so that $D_0=D\cap S_0=\{x_0\}$.
Then we can identify $\ker(r)$ with the image of the natural embedding
$$\TT_{D,x_0}|_{x_0}\to \TT_{S,S_0}|_{x_0}.$$
Applying Lemma \ref{KS-even-lem}(iv)(v) to the induced family of curves $\pi_D:C_D\to D$, we get an exact sequence of sheaves on $C_{x_0}$
$$0\to \II_{q_{x_0}}\rTo{t} \TT_{C_D,C_{x_0}}|_{C_{x_0}}\to \TT_{C_{x_0}}.$$
where $q_{x_0}\in C_{x_0}$ is the node.
Since $H^0(C_{x_0},\TT_{C_{x_0}})=H^0(C_{x_0},\II_{q_{x_0}})=0$, this implies that
$$H^0(C_{x_0},\TT_{C_D,C_{x_0}}|_{C_{x_0}})=0.$$

Therefore, looking at the exact sequence of sheaves on $C_{x_0}$,
$$0\to \om_{C_{x_0}}^{-1}\to \TT_{C_D,C_{x_0}}|_{C_{x_0}}\to (\pi_D^*\TT_{D,x_0})|_{C_{x_0}}\to 0$$
we get that the corresponding coboundary map
\begin{equation}
\label{C-x0-TD-coboundary-eq}
H^0(C_{x_0},\TT_{D,x_0}|_{x_0}\ot\OO_{C_{x_0}})\simeq\TT_{D,x_0}|_{x_0}\to H^1(C_{x_0},\om_{C_{x_0}}^{-1})
\end{equation}
in injective. 

Now, restricting to $C_{x_0}$ the natural morphism of exact sequences of vector bundles
\begin{diagram}
0&\rTo{}&\om_{C_D/D}^{-1}&\rTo{}&\TT_{C_D,C_{x_0}}&\rTo{}&\pi_D^*\TT_{D,x_0}&\rTo{}& 0\\
&&\dTo{\sim}&&\dTo{}&&\dTo{}\\
0&\rTo{}&\om_{C/S}^{-1}|_{C_D}&\rTo{}&\TT_{C,C_0}|_{C_D}&\rTo{}&\pi_D^*\TT_{S,S_0}|_D&\rTo{}&0
\end{diagram}
and considering the morphism between the corresponding long exact sequences of cohomology on $C_{x_0}$,
we deduce that the coboundary map \eqref{C-x0-TD-coboundary-eq} is equal to the restriction of $\kappa'_{x_0}$ to $\ker(r)$,
which proves our claim.

Finally, by stability of $C_{x_0}$, we have $H^0(C_{x_0},\om_{C_{x_0}}^{-1})=0$.
Hence, the long exact sequence
$$0\to H^0(C_{x_0},\om_{C_{x_0}}^{-1})\to
H^0(C_{x_0},\TT_{C,C_0}|_{C_{x_0}})\to \TT_{S,S_0}|_{x_0}\rTo{\kappa'_{x_0}}H^1(C_{x_0},\om_{C_{x_0}}^{-1})$$
shows the vanishing of
$H^0(C_{x_0},\TT_{C,C_0}|_{C_{x_0}})$.
\end{proof}

\begin{cor}\label{det-KS-class-sec}
In the situation of Proposition \ref{KS-even-prop} assume in addition that the map from $S$ to the moduli space of stable curves induces an isomorphism
on tangent spaces at the origin.
Then the map \eqref{KS-even-map} is an isomorphism.
Hence, the determinant of the usual Kodaira-Spencer map on $S'=S\setminus S_0$,
$$\det KS: \Det \TT_{S'}\to \Det R^1\pi_*\om^{-1}_{C/S}|_{S'}$$
is of the form $u/t$, where $u$ is a unit (with respect to trivializations regular on $S$).
\end{cor}

\begin{proof}
For the second statement, we can argue formally locally. Then we observe that $\TT_{S,S_0}$ has a basis $t\partial_t,\partial_{t_1},\ldots,\partial_{t_d}$, where $(t,t_1,\ldots,t_d)$ are 
formal coordinates on $S$. Hence the determinant of the natural morphism
$\TT_{S,S_0}\to \TT_S$ with respect to the standard bases is equal to $t$.
\end{proof}

%Since $\TT_{S,S_0}$ is generated by $t\partial_t$,
%It follows that the Kodaira-Spencer map for the induced family over the punctured formal disk has a pole of order $1$ at $0\in S$.

\subsection{Super case, NS node}\label{super-KS-NS-sec}

First, let us review the Kodaira-Spencer map for a family of smooth supercurves $\pi:X\to S$.
We have the standard exact sequence
$$0\to \TT_{X/S}\to \TT_X\rTo{d\pi} \pi^*\TT_S\to 0.$$
We also have a relative distribution $\DD\sub \TT_{X/S}$.
Let us set 
%$\TT_\pi:=(d\pi)^{-1}\sub \TT_X$, and 
$$\AA_X:=\{v\in \TT_X \ |\ [v,\DD]\sub \DD\}.$$
Then there is an exact sequence (see \cite[Sec.\ 2]{LR})
\begin{equation}\label{A-X-smooth-case-eq}
0\to \AA_{X/S}\to \AA_X\to \pi^{-1}\TT_S\to 0
\end{equation}
where 
$$\AA_{X/S}:=\AA_X\cap \TT_{X/S}.
$$
Indeed, let $D$ be a local generator of $\DD$. Since $\DD\sub \TT_{X/S}$,
the condition $[v,D](f)=0$, for $f\in \pi^{-1}\OO_S$, is equivalent to $D(v(f))=0$, i.e., to $v(f)\in \pi^{-1}\OO_S$.
Furthermore, locally any $v\in \pi^{-1}\TT_S$ can be extended to a section of $\AA_X$.
%the surjectivity of the projection $\AA_X\to \pi^*\TT_S$ can be easily checked locally.
%(note that there is a mistake in \cite{LR} in that they do not restrict to $\TT_\pi$ when defining $\AA_\pi$).
Also, as we have seen before,
%the local computation shows that the projection
$$\AA_{X/S}\simeq \TT_{X/S}/\DD\simeq \DD^2\simeq \om_{X/S}^{-2}.$$
Now the Kodaira-Spencer map is the coboundary map
\begin{equation}
KS: \TT_S\to \pi_*\pi^{-1}\TT_S\to R^1\pi_*\AA_{X/S}\simeq R^1\pi_*\om_{X/S}^{-2}.
\end{equation}

We want to study the behavior of this map near the component of the boundary divisor of the moduli space where a supercurve acquires a NS node.
So let us consider a family of stable supercurves $\pi:X\to S$ over a smooth affine base with a smooth map $t:S\to \A^1$.
We denote by $S_0\sub S$ the divisor $t=0$ and by $\pi_0:X_0\to S_0$ the induced family.
We assume that there is a section $q:S_0\to X_0$ such that $q(S_0)$ is the relative node of $X_0$ and the map $X\to S$ is smooth away from $q(S_0)$.
Furthermore, we assume that the structure sheaf of $X$ completed along $q(S_0)$ is generated over $\OO_S$ by even generators $z_1,z_2$ and
odd generators $\th_1,\th_2$, subject to the relations
\begin{equation}\label{NS-node-def-relations}
z_1z_2=-t^2, \ \ z_1\th_2=t\th_1, \ \ z_2\th_1=-t\th_2, \ \ \th_1\th_2=0
\end{equation}
(so that $q(S_0)$ corresponds to $z_1=z_2=0$, $\th_1=\th_2=0$).
In this case, by Lemma \ref{etale-local-nodes-lem}, arguing \'etale locally, we can assume that 
the complement to the node is covered by two charts $U_1$ and $U_2$ where $z_i$ is invertible on $U_i$,
and there exist odd sections $s_i$ (resp., even section $s_0$) of $\om_{X/S}$
% (defined in the formal neighborhood of the node) 
such that
\begin{equation}\label{s1-s2-s0-generators}
\begin{array}{l}
s_1=\begin{cases} [dz_1 | d\th_1] & \text{ on } U_1, \\ -\frac{t}{z_2}[dz_2 | d\th_2] & \text{ on } U_2;\end{cases}, \\
s_2=\begin{cases} \frac{t}{z_1}[dz_1 | d\th_1] & \text{ on } U_1, \\ [dz_2 | d\th_2] & \text{ on } U_2;\end{cases}, \\
s_0=\begin{cases} \frac{\th_1}{z_1}[dz_1 | d\th_1] & \text{ on } U_1, \\ -\frac{\th_2}{z_2}[dz_2 | d\th_2] & \text{ on } U_2;\end{cases}
\end{array}
\end{equation}
(see \cite[Sec.\ 2.3]{Deligne}). 
%The fact that $(s_1,s_2,s_0)$ are regular sections generating $\om_{X/S}$ follows from Lemma \ref{om-stable-lem}.
Note that the relative derivation $\de:\OO_X\to \om_{X/S}$ satisfies
$$\de(z_i)=\th_i s_i, \ \ \de(\th_i)=s_i,$$
for $i=1,2$. 

\begin{lemma}\label{NS-om-lem} 
In the above situation, in the formal neighborhood of the node, $\om_{X/S}$ is generated as an $\OO_X$-module
by global sections $s_1,s_2$ and $s_0$ subject to defining relations
$$z_1s_2=ts_1, \ \ z_2s_1=-ts_2, \ \ \th_1s_2=\th_2s_1=ts_0, \ \ z_1s_0=\th_1s_1, \ \ z_2s_0=-\th_2s_2, \ \ \th_1s_0=\th_2s_0=0.$$
It has a topological basis over $\OO_S$,
$$z_1^ns_1, \ \ z_2^ns_2,  \ \ z_1^n\th_1s_1, \ \ z_2^n\th_2s_2, \ \ s_0,
$$
where $n\ge 0$.
\end{lemma}

\begin{proof}
It is enough to check this when $S=\C[t]$, 
%is the formal spectrum of $k[\![t]\!]$, 
in which case this is discussed in \cite[Sec.\ 2.3]{Deligne}.
In more detail, $X$ is split, with $\OO_X=\OO_C\oplus L$, where $C$ is the curve $z_1z_2=-t^2$, and $L$ is the $\OO_C$-module
generated by $\th_1$,$\th_2$ (which is a maximal CM-module over $\OO_C$). Then $\om_X=\om_C\oplus \und{\Hom}(L,\om_C)$,
and we have an isomorphism $\de:L\rTo{\sim}\und{\Hom}(L,\om_C)$. 
The section $s_0$ corresponds to the standard generator of $\om_C$, while $s_1$ and $s_2$ are the images of the generators $\th_1,\th_2$ of $L$
under $\de$. Another way to prove this is to use Lemma \ref{om-stable-lem}.
\end{proof}

Note that the line bundle $\om_{X/S}^2:=j_*\om_{U/S}^2$ is generated near the node by the section
\begin{equation}\label{e-sec-eq}
e=\begin{cases} \frac{1}{z_1} [dz_1 |d\th_1]^2 & \text{ on } U_1, \\ \frac{1}{z_2} [dz_2 |d\th_2]^2 & \text{ on } U_2.\end{cases}
\end{equation}

As before, we denote by $\TT_X$ the sheaf of derivations of $\OO_X$
and consider the subsheaf $\AA_X\sub \TT_X$ consisting of $v$ such that $[v,\DD]\sub \DD$ on the smooth locus
of $\pi:X\to S$. We also set $\AA_{X/S}=\AA_X\cap \TT_{X/S}$. We use a similar definition for $\AA_{X_0/S_0}\sub \AA_{X_0}$.

Now similarly to the nodal even case we consider the subsheaf
$\AA_{X,X_0}\sub \AA_X$ consisting of $v$ such that $v(t)\in (t)$. We have a natural projection
$$\AA_{X,X_0}\to \pi^{-1}\TT_{S,S_0}.$$

Note that some of the above sheaves do not have $\OO_X$-module, only the $\pi^{-1}\OO_S$-module structure.
Nevertheless, for a subscheme $S'\sub S$ we have a natural operation of restriction to $\pi^{-1}(S')$: we set for an $\pi^{-1}\OO_S$-module $\FF$,
$$\FF|_{\pi^{-1}(S')}:=\FF\ot_{\pi^{-1}\OO_S}\pi^{-1}\OO_{S'}.$$

Below we will prove an analog of Lemma \ref{KS-even-lem}.
Note that the slight difference from the case of the usual curves is that the sheaves $\AA_X$, $\AA_{X,X_0}$ and $\AA_{X/S}$
are not $\OO_X$-submodules of $\TT_X$. However, as we know from Theorem \ref{inf-aut-thm}, there is an isomorphism
of $\pi^{-1}\OO_S$-modules
$$\AA_{X/S}\simeq \om_{X/S}^{-2}:=j_*\om_{U/S}^{-2},$$ 
where $j:U=X\setminus\{q\}\to X$ denotes the open embedding and $\om_{X/S}^{-2}$ is a line bundle on $X$ (see Theorem \ref{can-square-thm}). 

%Then the natural map $\om_{X/S}\to j_*\om_{X\setminus\{q\}/S}$ is an isomorphism.
%The corresponding natural map
%$$\Om^1_{X/S}\to j_*\Om^1_{X\setminus\{q\}/S}\simeq j_*\om_{C\setminus\{q\}/S}\simeq \om_{C/S}$$
%is injective with the cokernel isomorphic to $\OO_q$.
%$$\AA_X\simeq j_*\TT_U/\DD, \ \ \AA_{X,X_0}\simeq j_*\TT_{U,U_0}/\DD$$
%and $U_0=U\cap X_0$.
%All of the sheaves in the right-hand sides are coherent, and
%$\AA_{X,X_0}$ (resp., $\AA_{X_0/S_0}$) is flat over $\pi^{-1}\OO_S$ (resp., over $\pi_0^{-1}\OO_{S_0}$).

%in part (iii) of the following lemma we show that $\AA_{X/S}$ is isomorphic to a
%coherent sheaf of $\OO_X$-modules.

\begin{lemma}\label{KS-super-NS-lem} (i) There are exact sequences
\begin{equation}\label{tangent-super-ex-seq}
0\to \AA_{X/S}\to \AA_{X,X_0}\to \pi^{-1}\TT_{S,S_0}\to 0
\end{equation}

\begin{equation}\label{tangent-X0-ex-seq}
0\to \AA_{X_0/S_0}\to \AA_{X_0}\to \pi_0^{-1}\TT_{S_0}\to 0
\end{equation}

\noindent
(ii) In the formal neighborhood of a node, derivations in $\AA_{X_0/S_0}$ are in bijection with pairs of functions $a_1,a_2$
in $\OO_{X_0}$ such that $a_i\in (z_i)$: the corresponding derivation $v$ is given by 
\begin{equation}\label{v-ai-Di-eq}
v(z_i)=a_i+(-1)^{|v|}\frac{1}{2}D_i(a_i)\th_i, \ \ v(\th_i)=(-1)^{|v|}\frac{1}{2}D_i(a_i)
\end{equation}
for $i=1,2$, where $D_i=\partial_{\th_i}+\th_i\partial_{z_i}$.

For the closed point $s_0\in S_0$, the natural morphism
$$\AA_{X_0/S_0}|_{X_{s_0}}\to \AA_{X_{s_0}/k}$$
is an isomorphism.

%\noindent
%(iii) 

\noindent
(iii) One has 
$$\AA_{X,X_0}=\AA_X$$
%where the map from $\AA_X$ is induced by the map 
%$$\AA_X\to \OO_X/(t)=i_*\OO_{X_0}: v\mapsto v(t) \mod (t).$$
%where $i:X_0\to X$ is the natural embedding, 
%sending $D$ to $D(t)\mod (t)$.
%Then its image is $i_*\II_q$, where $\II_q\sub \OO_{X_0}$ is the ideal sheaf of $q$, so we have
%an exact sequence
%\begin{equation}\label{tangent-X0-ex-seq}
%0\to \AA_{\pi,X_0}\to \AA_\pi\to i_*\II_q\to 0
%\end{equation}
and the natural morphism
$$\AA_{X,X_0}/t\AA_{X,X_0}\to i_*\AA_{X_0}$$
is injective.
%$$0\to \AA_X/\AA_{X,X_0}\rTo{t} \AA_{X,X_0}/t\AA_{X,X_0}\to i_*\AA_{X_0}$$
\end{lemma}

\begin{proof}
(i) In both cases it is enough to prove surjectivity of the last arrow in the formal neighborhood of $q$.
Let us first check this for $\AA_{X,X_0}\to \pi^{-1}\TT_{S,S_0}$.
It is enough to extend $t\partial_t$ to an even derivation $v$ in the formal neighborhood of $q$ in $X$,
such that $v|_{U_i}$ preserves $\OO_X\cdot (\partial_{\th_i}+\th_i\partial_{z_i})$, for $i=1,2$. For this, we can take $v$ given by
$$v(z_i)=z_i, \ \ v(\th_i)=\frac{1}{2}\th_i, \ \ v(t)=t.$$
This shows that sequence \eqref{tangent-super-ex-seq} is exact.

The composed arrow
$$\AA_{X,X_0}\to \pi^{-1}\TT_{S,S_0}\to i_*\pi_0^{-1}\TT_{S_0}$$
is still surjective and factors through $i_*\AA_{X_0/S_0}$. This implies exactness
of \eqref{tangent-X0-ex-seq}.

\noindent
(ii) Recall that $\OO_{X_0}$ is the quotient
of $\OO_{S_0}[\![z_1,z_2,\th_1,\th_2]\!]$ by the relations 
$$z_1z_2=0, \ \ z_1\th_2=0, \ \ z_2\th_1=0, \ \ \th_1\th_2=0.$$
Thus, a derivation $v$ in $\AA_{X_0/S_0}$ is described by the functions
$v(z_1)$, $v(z_2)$, $v(\th_1)$ and $v(\th_2)$. Furthermore, there should exist
a pair of functions $a_1(z_1,\th_1)\in \OO_{X_0}[z_1^{-1}]$, $a_2(z_2,\th_2)\in \OO_{X_0}[z_2^{-1}]$ of the same parity as $v$ such that
$$v(z_i)=a_i+(-1)^{|v|}\frac{1}{2}D_i(a_i)\th_i, \ \ v(\th_i)=(-1)^{|v|}\frac{1}{2}D_i(a_i) \ \text{ in } \OO_{X_0}[z_i^{-1}],$$
for $i=1,2$, where $D_i=\partial_{\th_i}+\th_i\partial_{z_i}$.
%Note that $a_i=v(z_i)-v(\th_i)\th_i$, so $a_i$ comes from $\OO_{X_0}$.
%Furthermore, 

The condition $v(z_1z_2)=0$ implies that 
$$v(z_i)\in (z_i,\th_i)\sub \OO_{X_0}$$
for $i=1,2$. Hence, $a_i$ is the image of the element $v(z_i)-v(\th_i)\th_i\in (z_i,\th_i)\sub \OO_{X_0}$
under the map $\OO_{X_0}\to \OO_{X_0}[z_i^{-1}]$.
Note that the restriction of the latter map to 
$$(z_i,\th_i)=\OO_S[z_i]z_i\oplus \OO_S[z_i]\th_i\sub \OO_{X_0}$$
is an embedding. Thus,
$v(z_i)$ is determined by its image in $\OO_{X_0}[z_i^{-1}]$ and the above formulas for $v(z_i)$ hold in $\OO_{X_0}$ with some
$$a_i=f_i(z_i)z_i+g_i(z_i)\th_i,$$
for $i=1,2$, with $f_i,g_i\in\OO_S[z_i]$.

It follows also that $v(z_1)\th_2=v(z_2)\th_1=0$.
Hence, the conditions $v(z_1\th_2)=v(z_2\th_1)=0$ are equivalent to $v(\th_i)\in (z_i,\th_i)$, for $i=1,2$.
Since $D_i(a_i)\equiv g_i(z_i)\mod (\th_i)$, this is equivalent to $g_i(z_i)\in (z_i)$. 
The condition $v(\th_1\th_2)$ is then automatically satisfied.
Note also that the condition $v(\th_i)\in (z_i,\th_i)$ implies that $v(\th_i)$ is determined by its image in $\OO_{X_0}[z_i^{-1}]$.
Thus, the condition on $v$ to define a section of $\AA_{X_0/S_0}$ is that
$a_i\in (z_i)$, for $i=1,2$, which is equivalent to our assertion.

For the last assertion, we note that it is enough to check it in the formal neighborhood of the node. Then the statement follows immediately from
the above explicit description of $\AA_{X_0/S_0}$ and from a similar description of $\AA_{X_{s_0}/k}$.

%\noindent
%(iii) 

%We already know that $j_*\TT_{U/S}/\DD\simeq j_*\om_{U/S}^{-2}$ is a line bundle on $X$.
%Note also that $j_*\OO_U\simeq \OO_S$. Thus, since $\pi^*\TT_S$ and $\pi^*\TT_{S,S_0}$ are trivial vector bundles
%on $X$, the exact sequences
%$$0\to j_*\TT_{U/S}/\DD\to j_*\TT_U/\DD \to j_*\pi^*\TT_S|_U$$
%$$0\to j_*\TT_{U/S}/\DD\to j_*\TT_{U,U_0}/\DD\to j_*\pi^*\TT_{S,S_0}|_U$$
%show that the sheaves $j_*\TT_U/\DD$ and $j_*\TT_{U,U_0}/\DD$ are coherent.

%The fact that $\AA_{X,X_0}$ is flat over $\pi^{-1}\OO_S$ follows from the exact sequence
%\eqref{tangent-super-ex-seq}.
%Finally, to check that $\AA_{X_0/S_0}$ is flat over $\OO_{S_0}$, it is enough to restrict to the formal neighborhood of $q$.
%Then the assertion follows immediately from (ii).

%Thus, we see that $\AA_{X_0/S_0}$ has a free basis over $\OO_{S_0}$.

\noindent
(iii) We claim that in fact any derivation $v$ in $\TT_X$ satisfies $v(t)\in (t,z_1,z_2,\th_1,\th_2)$.
Indeed, we have
$$v(z_1)z_2+z_1v(z_2)=-2tv(t),$$
so $tv(t)\in (z_1,z_2)$.
Let us write $tv(t)=z_1f+z_2f'$, and decompose $f$ and $f'$ with respect to the topological $\OO_S$-basis 
$$1, (z_1^n)_{n\ge 1}, (z_2^n)_{n\ge 1}, (z_1^n\th_1)_{n\ge 0}, (z_2^n\th_2)_{n\ge 0}.$$
Then we get the equation of the form
\begin{align*}
&tv(t)=z_1[a_0+\sum_{n\ge 1} a_n z_1^n+\sum_{n\ge 1} b_n z_2^n+\sum_{n\ge 0} c_nz_1^n\th_1+\sum_{n\ge 0}d_nz_2^n\th_2]\\
&+z_2[a'_0+\sum_{n\ge 1} a'_n z_1^n+\sum_{n\ge 1} b'_n z_2^n+\sum_{n\ge 0} c'_nz_1^n\th_1+\sum_{n\ge 0}d'_nz_2^n\th_2]
\end{align*}
The relations imply that the free term of the right-hand side is $-(b_1+a'_1)t^2$, hence the free term of $v(t)$ is divisible by $t$, as claimed.
%It remains  to construct a derivation $v\in \AA_X$ with $v(t)$ of any given form
%$$v(t)=f_1(z_1)z_1+f_2(z_2)z_2+g_1(z_1)\th_1+g_2(z_2)\th_2\mod(t)$$ 
%???

It follows that for any $v\in \AA_X$, we have $v(t)\in t\pi^{-1}\OO_S$, so $v\in \AA_{X,X_0}$.
For the last assertion we observe that the kernel of
the restriction map
$$\AA_{X,X_0}\to i_*\AA_{X_0}$$
consists of derivations $v$ in $\AA_X$ that are in $t\cdot \TT_X$. Hence, this kernel is $t\AA_X=t\AA_{X,X_0}$.
\end{proof}

Now, as in the even case, we consider the {\it Kodaira-Spencer maps}
associated with sequence \eqref{tangent-super-ex-seq} 
%and \eqref{tangent-X0-ex-seq}:
\begin{equation}\label{KS-super-NS-map}
KS:\TT_{S,S_0}\to R^1\pi_*\om^{-2}_{X/S},
\end{equation}
%\begin{equation}\label{KS0-super-NS-map}
%KS_0:\TT_{S_0}\to R^1\pi_{0*}\AA_{X_0/S_0}.
%\end{equation}
%Note that the coboundary maps associated with our sequences give similar maps with the source
%$\TT_{S,S_0}\ot \pi_*\OO_X$ (resp., $\TT_{S_0}\ot \pi_{0*}\OO_{X_0}$), so to get $KS$ (resp., $KS_0$)
%we use the natural maps $\OO_S\to \pi_*\OO_X$ and $\OO_{S_0}\to \pi_*\OO_{X_0}$.

Let $s_0\in S_0$ denote the origin, and consider the restriction of the exact sequence \eqref{tangent-X0-ex-seq}
to $X_{s_0}$, which can be identified with
$$0\to \AA_{X_{s_0}}\to \AA_{X_0}|_{X_{s_0}}\to T_{s_0}S_0\ot \und{\C}_{X_{s_0}}\to 0.$$
We have the corresponding coboundary map
$$\kappa_{s_0}:T_{s_0}S_0\to H^1(X_{s_0},\AA_{X_{s_0}}).$$

Now we can prove the following analog of Proposition \ref{KS-even-prop}.

\begin{prop}\label{KS-super-NS-prop}
Assume that the map $\kappa_{x_0}$ is injective.
%, and that $H^0(X_{x_0},\OO)^-=0$.
%Assume also the existense of the fiberwise normalization as above. 
Then $R^1\pi_*\om^{-2}_{X/S}$ is a vector bundle on $S$
and the map \eqref{KS-super-NS-map} is an embedding of a subbundle.
\end{prop}

\begin{proof} First, let $(C_{s_0},L_{s_0})$ be the usual stable curve with a spin structure, underlying $X_{s_0}$.
We observe that
$$\om^{-2}_{X/S}|_{X_{s_0}}=\om_{C_{s_0}}^{-1}\oplus \om_{C_{s_0}}^{-1}\ot L_{s_0}.$$
Indeed, this follows easily from Theorem \ref{can-square-thm} applied to
the induced family of stable supercurves $X_D\to D$ over the formal disk $D\sub S_0$ with the coordinate $t$.
By stability of $C_{s_0}$, we get
$$H^0(C_{s_0},\om^{-1})=H^0(C_{s_0},\om^{-1}\ot L_{s_0})=0.$$
In other words, we get the vanishing
\begin{equation}\label{X-s0-om-van}
H^0(X_{s_0},\om^{-2}_{X/S}|_{X_{s_0}})=0.
\end{equation} 
It follows that
$$\pi_*\AA_{X/S}=\pi_*\om^{-2}_{X/S}=0,$$
and that $R^1\pi_*\om^{-2}_{X/S}$ is a vector bundle on $S$.

%Note that our assumption $H^0(X_{x_0},\OO)^-=0$ implies that the natural maps
%$$\OO_S\to \pi_*\OO_X, \ \ \OO_S\to \pi_{0*}\OO_{X_0}$$ are isomorphisms and $R^1\pi_*\OO_X$ is locally free
%(see ???).
%We have a long exact sequence
%\begin{equation}\label{KS-super-long-ex-seq}
%0\to \pi_*\AA_{X,X_0}\to \TT_{S,S_0}\rTo{KS} R^1\pi_*\AA_{X/S}\to R^1\pi_*\AA_{X,X_0}\to R^1\pi_*\OO_X\ot \TT_{S,S_0}\to 0.
%\end{equation}
%We claim that it is enough to check that $H^0(X_{s_0},\AA_{X,X_0}|_{X_{s_0}})=0$.
%Indeed, since $\AA_{X,X_0}$ is isomorphic to a coherent sheaf on $X$, flat over $S$, this would imply that
%$\pi_*\AA_{X,X_0}=0$ while $R^1\pi_*\AA_{X,X_0}$ is locally free. Since $R^1\pi_*\OO_X$ is also locally free, our assertion would follow.

Thus, $KS$ is a map of vector bundles on $S$, so it is enough to show that the corresponding map of fibers at $s_0$,
$$KS(s_0):\TT_{S,S_0}|_{s_0}\to (R^1\pi_*\om_{X/S}^{-2})|_{s_0}\simeq H^1(X_{s_0},\om_{X/S}^{-2}|_{X_{s_0}}),$$
is injective. 

To this end, let us consider the morphism of exact sequences on $X_0$,
\begin{diagram}
0&\rTo{}&\AA_{X/S}|_{X_0}&\rTo{}&\AA_{X,X_0}|_{X_0}&\rTo{}&\pi_0^{-1}\TT_{S,S_0}|_{S_0}&\rTo{}& 0\\
&&\dTo{}&&\dTo{}&&\dTo{}\\
0&\rTo{}&\AA_{X_0/S_0}&\rTo{}&\AA_{X_0}&\rTo{}&\pi_0^{-1}\TT_{S_0}&\rTo{}&0
\end{diagram}
Since the rightmost terms are locally free, the restrictions of these sequences to $X_{s_0}$ are still exact, so we get
a commutative square of the coboundary maps
\begin{diagram}
\TT_{S,S_0}|_{s_0}&\rTo{KS(s_0)}& H^1(X_{s_0},\om_{X/S}^{-2}|_{X_{s_0}})\\
\dTo{r}&&\dTo{}\\
T_{s_0}S_0&\rTo{\kappa_{s_0}}& H^1(X_{x_0},\AA_{X_{s_0}})
\end{diagram}
Since $\kappa_{s_0}$ is injective, arguing as before, we see that it is enough to check
that the restriction of $KS(s_0)$ to the $1$-dimensional subspace
$$\TT_{D,s_0}|_{s_0}\hra \TT_{S,S_0}|_{s_0}$$
is injective.

Considering the induced family $\pi_D:X_D\to D$ and applying Lemma \ref{KS-super-NS-lem}(iv), we get
an inclusion of sheaves on $X_{x_0}$,
$$\AA_{X_D,X_{s_0}}|_{X_{s_0}}\hra \AA_{X_{s_0}}.$$
We know that $H^0(X_{s_0},\AA_{X_{s_0}})=0$ (see Prop.\ \ref{no-inf-aut-prop}(ii)), so we get that
$$H^0(X_{s_0},\AA_{X_D,X_{s_0}}|_{X_{s_0}})=0.$$

Next, restricting the sequence \eqref{tangent-super-ex-seq} for the family $X_D/D$ to $X_{s_0}$, we get an exact sequence
$$0\to \om_{X_D}^{-2}|_{X_{s_0}}\to \AA_{X_D,X_{s_0}}|_{X_{s_0}}\to (\pi_D^{-1}\TT_{D,s_0})|_{X_{s_0}}\to 0.$$
From the above vanishing we get that the corresponding coboundary map
\begin{equation}
\label{X-s0-TD-super-coboundary-eq}
H^0(X_{s_0},\TT_{D,s_0}|_{s_0}\ot\OO_{X_{s_0}})\simeq\TT_{D,s_0}|_{s_0}\to H^1(X_{s_0},\om_{X_D}^{-2}|_{X_{s_0}})
\end{equation}
is injective. 

Now, restricting to $X_{s_0}$ the natural morphism of exact sequences 
\begin{diagram}
0&\rTo{}&\om_{X_D/D}^{-2}&\rTo{}&\AA_{X_D,X_{s_0}}&\rTo{}&\pi_D^{-1}\TT_{D,s_0}&\rTo{}& 0\\
&&\dTo{\sim}&&\dTo{}&&\dTo{}\\
0&\rTo{}&\om_{X/S}^{-2}|_{X_D}&\rTo{}&\AA_{X,X_0}|_{X_D}&\rTo{}&\pi_D^{-1}\TT_{S,S_0}|_D&\rTo{}&0
\end{diagram}
and considering the morphism between the corresponding long exact sequences of cohomology on $X_{s_0}$,
we deduce that the coboundary map \eqref{X-s0-TD-super-coboundary-eq} is equal to the restriction of $KS(s_0)$ to $\ker(r)$,
which proves our claim.
%Finally, the long exact sequence
%$$0\to H^0(X_{s_0},\om_{X/S}^{-2}|_{X_{s_0}})\to
%H^0(X_{s_0},\AA_{X,X_0}|_{X_{s_0}})\to \TT_{S,S_0}|_{s_0}\rTo{\kappa'_{s_0}}H^1(X_{s_0},\om_{X/S}^{-2}|_{X_{s_0}}),$$
%the vanishing \eqref{X-s0-om-van} and injectivity of $\kappa'_{s_0}$ imply the vanishing of
%$H^0(X_{s_0},\AA_{X,X_0}|_{X_{s_0}})$.
\end{proof}

Note that the Berezinian of the natural morphism $\TT_{S,S_0}\to \TT_S$ is equal to $t$ for an appropriate choice of bases.
As in Corollary \ref{det-KS-class-sec}, this leads to the following result.

\begin{cor}\label{KS-super-NS-node-cor}
Assume that the map $\kappa_{x_0}$ is an embedding and the superdimension of $S$ is $3g-3|2g-2$.
Then the map \eqref{KS-super-NS-map} is an isomorphism. Hence, the Berezinian of the usual Kodaira-Spencer map on $S'=S\setminus S_0$,
$$\ber(KS):\Ber(\TT_{S'})\to \Ber R^1\pi_*\om_{X'/S'}^{-2}$$
is of the form $u/t$, where $u$ is invertible on $S$ (with respect to trivializations regular on $S$).
\end{cor}

%\begin{rem} If $H^0(X_{s_0},\OO)^-\neq 0$ but all other assumptions of Corollary \ref{KS-super-NS-node-cor} are satisfied 
%then the map \eqref{KS-super-NS-map} is not an isomorphism.
%Indeed, it is easy to see that the natural map $\OO_S\to \pi_*\OO_X$ is still an isomorphism, hence, we still have an exact sequence
%\eqref{KS-super-long-ex-seq}, and it would imply that the map
%$$R^1\pi_*\AA_{X,X_0}\to R^1\pi_*\OO_X\ot \TT_{S,S_0}$$
%is an isomorphism.
%???
%\end{rem}

\subsection{Super case, Ramond node}\label{super-KS-R-sec}

Now we consider a family $X\to S$ over a smooth affine base $S$ with the relative Ramond node $q:S_0\to X_0=\pi^{-1}(S_0)$ over the divisor $S_0=(t=0)$ 
(where $t:S\to \A^1$ is a smooth morphism), 
such that $X\setminus q(S_0)$ is smooth over $S$, and
the completion of $\OO_X$ along $q(S_0)$ is generated over $\OO_S$ by generators $z_1,z_2,\th$ subject to the single relation
$$z_1z_2=t.$$
The distribution is generated by $\partial_{\th}+\th z_i\partial_{z_i}$ over $z_i\neq 0$.

Note that in this case the relative dualizing sheaf $\om_{X/S}$ is a line bundle on $X$.

\begin{lemma}\label{KS-super-R-lem} (i) The sequences
\eqref{tangent-super-ex-seq} and \eqref{tangent-X0-ex-seq} are still exact.

\noindent
(ii) In the formal neighborhood of a node, even (resp., odd) derivations in $\AA_{X_0/S_0}$ are in bijection with pairs of even (resp., odd) functions 
$a_i=f_i(z_i)+g_i(z_i)\th\in \OO_S[z_i,\th_i]$, $i=1,2$,
such that $g_1(0)=g_2(0)$.
The corresponding derivation $v$ is given by 
$$v(z_i)=a_iz_i+(-1)^{|v|}\frac{1}{2}D_i(a_i)z_i\th, \ \text{ for } i=1,2,$$
$$v(\th)=(-1)^{|v|}\frac{1}{2}[(-1)^{|v|}(g_1+g_2-g_2(0))+\th z_1\partial_{z_1}(f_1)+\th z_2\partial_{z_2}(f_2)].$$
where $D_i=\partial_{\th}+\th z_i\partial_{z_i}$.

For a closed point $s_0\in S_0$, the natural morphism
$$\AA_{X_0/S_0}|_{X_{s_0}}\to \AA_{X_{s_0}/k}$$
is an isomorphism.

%\noindent
%(iii) 
%Then the natural map $\om_{X/S}\to j_*\om_{X\setminus\{q\}/S}$ is an isomorphism.
%The corresponding natural map
%$$\Om^1_{X/S}\to j_*\Om^1_{X\setminus\{q\}/S}\simeq j_*\om_{C\setminus\{q\}/S}\simeq \om_{C/S}$$
%is injective with the cokernel isomorphic to $\OO_q$.
%There is a natural isomorphism of $\pi^{-1}\OO_S$-modules
%$$\AA_X\simeq j_*\TT_U/\DD, \ \ \AA_{X,X_0}\simeq j_*\TT_{U,U_0}/\DD$$
%$$\AA_{X/S}\simeq \om_{X/S}^{-2}.$$ 
%and $U_0=U\cap X_0$.
%All of the sheaves in the right-hand sides are coherent, and
%$\AA_{X,X_0}$ (resp., $\AA_{X_0/S_0}$) is flat over $\pi^{-1}\OO_S$ (resp., over $\pi_0^{-1}\OO_{S_0}$).

\noindent
(iii) One has 
$$\AA_{X,X_0}=\AA_X$$
%where the map from $\AA_X$ is induced by the map 
%$$\AA_X\to \OO_X/(t)=i_*\OO_{X_0}: v\mapsto v(t) \mod (t).$$
%where $i:X_0\to X$ is the natural embedding, 
%sending $D$ to $D(t)\mod (t)$.
%Then its image is $i_*\II_q$, where $\II_q\sub \OO_{X_0}$ is the ideal sheaf of $q$, so we have
%an exact sequence
%\begin{equation}\label{tangent-X0-ex-seq}
%0\to \AA_{\pi,X_0}\to \AA_\pi\to i_*\II_q\to 0
%\end{equation}
and the natural morphism
$$\AA_{X,X_0}/t\AA_{X,X_0}\to i_*\AA_{X_0}$$
is injective.
%$$0\to \AA_X/\AA_{X,X_0}\rTo{t} \AA_{X,X_0}/t\AA_{X,X_0}\to i_*\AA_{X_0}$$
\end{lemma}
 
\begin{proof}
(i) To show surjectivity of the morphism $\AA_{X,X_0}\to \pi^{-1}\TT_{S,S_0}$
we observe that the derivation $v\in \AA_{X,X_0}$, given by
$$v(t)=2t, \ \ v(z_i)=z_i, \ \ v(\th)=0,$$
extends $2t\partial_t$.

\noindent
(ii)
Here is the description of $\AA_{X_0/S_0}$.
The condition $v(z_1)z_2+v(z_2)z_1=0$ implies that $v(z_i)\in (z_i)$, in particular $v(z_i)$ is determined by its image in $\OO_{X_0}[z_i^{-1}]$.
Also, we should have $a_i\in \OO_S[z_i,z_i^{-1},\th]$ such that
$$v(z_i)=a_iz_i+(-1)^{|v|}\frac{1}{2}D_i(a_i)z_i\th, \ \ v(\th)=(-1)^{|v|}\frac{1}{2} D_i(a_i) \ \text{ in } \OO_{X_0}[z_i^{-1}].$$
Let us write $a_i=f_i(z_i)+g_i(z_i)\th$. Then $a_1$ and $a_2$ could be arbitrary elements of $\OO_S[z_1,\th]$ and $\OO_S[z_2,\th]$ (of the same parity) such that 
$g_1(0)=g_2(0)$.
 
%\noindent
%(iii) The proof is similar to that of Lemma \ref{KS-super-NS-lem}(iii).

\noindent
(iii) The proof is similar to that of Lemma \ref{KS-super-NS-lem}(iv):
we use the condition $v(t)\in (z_1,z_2)\sub \OO_X$ and the
topological $\OO_S$-basis of $\OO_X$
$$1, (z_1^n)_{n\ge 1}, (z_2^n)_{n\ge 1}, (z_1^n\th)_{n\ge 0}, (z_2^n\th)_{n\ge 0}.$$
\end{proof}
 
Using the same arguments as before we derive the following assertion.  
 
\begin{prop}\label{KS-super-R-prop}
The statements of Proposition \ref{KS-super-NS-prop} and Corollary \ref{KS-super-NS-node-cor} hold 
in the case of a Ramond degeneration as well.
\end{prop}

\subsection{Kodaira-Spencer in the presence of NS and Ramond punctures}\label{Kodaira-punct-sec}

Everything in Sec.\ \ref{super-KS-NS-sec} and \ref{super-KS-R-sec} has an analog in the case of families of stable supercurves with punctures.
Namely, for such a family $(X,P_\bullet,R_\bullet)/S$, we should replace the sheaves $\AA_X$ and $\AA_{X/S}$ 
with their intersections $\AA_{X,P,R}$ and $\AA_{(X,P,R)/S}$ 
with the subsheaf $\TT_{X,P,R}\sub \TT_X$ of derivations preserving all the punctures (i.e., the corresponding ideals in $\OO_X$).
This will not change the local picture near the nodes, however, we need to make some changes in the global statements.

Let us assume that we have a family $(X,P,R)$ of stable supercurves with punctures over $S$, acquiring a single node over $S_0=(t=0)\sub S$,
so that forgetting the punctures we are in the situation of either Sec.\ \ref{super-KS-NS-sec} or Sec.\ \ref{super-KS-R-sec}.
We will still have the equality $\AA_{X,P,R}=\AA_{X,X_0,P,R}$ as in Lemma \ref{KS-super-NS-lem}, and the analogs of sequences \eqref{tangent-super-ex-seq}
and \eqref{tangent-X0-ex-seq} are
$$0\to \AA_{(X,P,R)/S}\to \AA_{X,P,R}\to \pi^{-1}\TT_{S,S_0}\to 0,$$
$$0\to \AA_{(X_0,P_0,R_0)/S}\to \AA_{X_0,P_0,R_0}\to \pi_0^{-1}\TT_{S_0}\to 0.$$
Indeed, the only extra statement is the surjectivity of the right arrows near the punctures, which follows from the standard local description of the punctures
(see Sec.\ \ref{local-descr-sec}).

%The analog of the isomorphism of Lemma \ref{KS-super-NS-lem}(iii) is 
%obtained by combining the same local arguments with the case of smooth supercurves with punctures (see Proposition 
%\ref{inf-aut-smooth-punct-prop}), namely, for a family of stable supercurves we have an isomorphism
Recall also that by Theorem \ref{inf-aut-thm}, we have an isomorphism
$$\AA_{(X,P_\bullet,R_\bullet)/S}\simeq \LL(X,P_\bullet,R_\bullet):=\om_{X/S}^{-2}(-\sum_{i\in I} D_i-2\sum_{j\in J} R_j),$$
where $D_i\sub X$ are the divisors associated with the NS nodes $P_i\sub X$, and $\om_{X/S}^{-2}$ is defined as in Theorem \ref{can-square-thm}.

The analog of Corollary \ref{KS-super-NS-node-cor} states that under the assumption that the map 
$$\kappa_{x_0}:T_{s_0}S_0\to H^1(X_{s_0}, \AA_{X_{s_0},P_{s_0},R_{s_0}})$$
is injective and the superdimension of $S$ is $3g-3|2g-2$, the Kodaira-Spencer map induces an isomorphism of line bundles on $S$,
$$\Ber(\TT_S)\rTo{\sim} \Ber(R^1\pi_*\LL(X,P_\bullet,R_\bullet))(S_0).$$

\section{The boundary divisor}\label{boundary-str-sec}

In this section we discuss the definition of the boundary of our compactification as an effective Cartier divisor. We also study the NS and the Ramond components
of the boundary divisor.

\subsection{Modified sheaf of differentials}

Let $\pi:X\to S$ be a family of stable supercurves such that the map from $S$ to the deformation space of each node is smooth.
Let $j:U\to X$ be the embedding of the smooth locus. We want to study the sheaf $j_*\Om_{U/S}$ and the map 
$$\Om_{X/S}\to j_*\Om_{U/S}.$$

Since $U/S$ is a smooth supercurve, we have an exact sequence
$$0\to \om^2_{U/S}\rTo{\kappa} \Om_{U/S}\rTo{\de} \om_{U/S}\to 0$$
where in standard coordinates
$$\de(dz)=\th [dz|d\th], \ \ \de(d\th)=[dz|d\th],$$
$$\kappa([dz|d\th]^{\ot 2})=dz-\th d\th.$$
Applying the functor $j_*$, we get an exact sequence
$$0\to \om^2_{X/S}\rTo{\kappa} j_*\Om_{U/S}\rTo{j_*\de} \om_{X/S},$$
where $\om^2_{X/S}=j_*\om_{U/S}$ is a line bundle on $X$ (see Theorem \ref{can-square-thm}).

\begin{lemma} The map $j_*\de$ is surjective, so we have an exact sequence
\begin{equation}\label{j-Om-ex-seq}
0\to \om^2_{X/S}\rTo{\kappa} j_*\Om_{U/S}\rTo{j_*\de} \om_{X/S}\to 0
\end{equation}
\end{lemma}

\begin{proof} We know that $j_*\de$ is surjective over $U$, so it is enough to check that it is surjective near the nodes.

Near a Ramond node, the map $\de:\Om_{X/S}\to \om_{X/S}=j_*\om_{U/S}$ is surjective. Since it factors through
$j_*\Om_{U/S}$, we see that $j_*\de$ is surjective near a Ramond node.

To check surjectivity near an NS-node, we can work \'etale locally and apply Lemma \ref{etale-local-nodes-lem}.
Hence, it remains to make a computation for the standard deformation of the NS-node.
Recall the generators $s_1,s_2,s_0$ of $\om_{X/S}$ (see Lemma \ref{NS-om-lem}).
We know that $s_1$ and $s_2$ are in the image of $\de:\Om_{X/S}\to \om_{X/S}$.
On the other hand, the sections $dz_1/z_1$ and $-dz_2/z_2$ of $\Om_{U/S}$ over $z_1\neq 0$ and $z_2\neq 0$ glue into a global section of
$j_*\Om_{U/S}$, which maps to $s_0$ under $j_*\de$.
\end{proof}

Now let us consider a family of stable supercurves $X/S$ over a smooth affine base $S$ with a smooth map $t:S\to \A^1$, such that over $S_0=(t=0)\sub S$,
we have a relative NS node $q:S_0\to X_0=S_0\times_S X$, so that $X\setminus q(S_0)$ is smooth over $S$, 
the completion of $\OO_X$ along $q(S_0)$ is given by the standard generators and relations 
\eqref{NS-node-def-relations}, and the derivation $\de$ is given by the standard formula.

\begin{lemma}\label{differential-generators-lem}
(i) In the above situation, \'etale locally along $q(S_0)$, the sheaf $j_*\Om_{U/S}$ is generated as an $\OO_X$-module by global sections
%\begin{equation}\label{j*Om-NS-generators-eq}
$$e=\frac{dz_1}{z_1}=-\frac{dz_2}{z_2}, \ \ d\th_1, \ \ d\th_2, \ \ f=\frac{\th_1d\th_1}{z_1}=-\frac{\th_2d\th_2}{z_2},$$
%\end{equation}
with defining relations
\begin{eqnarray}\label{j*Om-NS-relations-eq}
%\begin{array}{l}
td\th_1-z_1d\th_2=t\th_1 e, \ \ z_2d\th_1+td\th_2=-t\th_2e, \ \ \ \ \ \ \ \ \ \ \ \ \ \ \nonumber\\
tf=\th_2d\th_1=\th_1d\th_2, \ \ z_1f=\th_1d\th_1, \ \ z_2f=-\th_2d\th_2, \ \ \th_1f=\th_2f=0.
\end{eqnarray}
%\end{equation}

\noindent
(ii) We have an exact sequence on $X_0$,
$$0\to \KK\to \Om_{X_0/S_0}\to j_*\Om_{U/S}|_{X_0}\to \CC_0\to 0$$
where $\KK$ has an $\OO_{S_0}$-basis 
\begin{equation}\label{KK-basis-eq}
z_1dz_2=-z_2dz_1, \ \ z_1d\th_2=-\th_2dz_1, \ \ z_2d\th_1=-\th_1dz_2, \ \ \th_1d\th_2=\th_2d\th_1,
\end{equation}
and $\CC_0$ has an $\OO_{S_0}$-basis the images of
\begin{equation}\label{CC-basis-eq}
e, \ \  f, \ \ \th_1e, \ \ \th_2e.
\end{equation}
Furthermore, the map of $\OO_{S_0}$-modules $\pi_*\KK\to \pi_*\Om_{X_0/S_0}$ is an embedding of a direct summand.
\end{lemma}

\begin{proof}
(i) \'Etale locally we can think of sections of $j_*\Om_{U/S}$ as compatible sections of $\Om_{U_1/S}$ and $\Om_{U_2/S}$, where $U_i$ is the open subset
where $z_i$ is invertible. Thus, $e$ is a well defined section of  $j_*\Om_{U/S}$.
Using relations \eqref{NS-node-def-relations}, we get
$$z_2d\th_1=-td\th_2-\th_1dz_2,$$
so over $U_1\cap U_2$,
$$\frac{\th_1}{z_1}d\th_1=\frac{\th_2}{t}(-\frac{t}{z_2}d\th_2)=-\frac{\th_2}{z_2}d\th_2,$$
so $f$ is a well defined section of $j_*\Om_{U/S}$.

Let us denote by $\FF$ the $\OO_X$-module given by generators $e$, $f$, $d\th_1$ and $d\th_2$ and relations \eqref{j*Om-NS-relations-eq}
It is straightforward to check (by inverting $z_1$ or $z_2$) that these relations are satisfied in $j_*\Om_{U/S}$, so we have a well defined morphism
$\FF\to j_*\Om_{U/S}$.
The map $\kappa:\om^2_{X/S}\to j_*\Om_{U/S}$ sends a generator to $e-f$. Thus, we 
%denote by $\FF$ the $\OO_X$-module given by our generators and relations then 
we have a commutative diagram with exact rows
\begin{equation}\label{j-Om-F-diagram}
\begin{diagram}
&&\OO_X &\rTo{e-f}& \FF&\rTo{}&\FF/\OO_X(e-f)&\rTo{}&0\\
&&\dTo{\sim}&&\dTo{} &&\dTo{}\\
0&\rTo{}&\om^2_{X/S}&\rTo{\kappa}&j_*\Om_{U/S} &\rTo{j_*\de}&\om_{X/S}&\rTo{}&0
\end{diagram}
\end{equation}
One can check using Gr\"obner basis technique that 
$$z_1^ne, \ \ z_1^n\th_1e, \ \  z_1^nd\th_1, \ \ z_1^n\th_1d\th_1, \ \ z_2^ne, \ \ z_2^n\th_2e, \ \ z_2^nd\th_2, \ \ z_2^n\th_2d\th_2, \ \ f,
$$
where $n\ge 0$, is an $\OO_S$-basis of $\FF$. This easily implies
that the map $\OO_X\rTo{e-f}\FF$ is injective (by looking at the coefficients of $e$).

The quotient $\FF/\OO_X(e-f)$ is generated by $e,d\th_1,d\th_2$ subject to relations
$$td\th_1-z_1d\th_2=0, \ \ z_2d\th_1+td\th_2=0,$$
$$te=\th_2d\th_1=\th_1d\th_2, \ \ z_1e=\th_1d\th_1, \ \ z_2e=-\th_2d\th_2, \ \ \th_1e=\th_2e=0.$$
Note that 
$$\de(d\th_1)=s_1,\ \ \de(d\th_2)=s_2, \ \ \de(e)=s_0.$$
Comparing these with generators and relations of $\om_{X/S}$ (see Lemma \ref{NS-om-lem})
we deduce that the map $\FF/\OO_X(e-f)\to \om_{X/S}$ is an isomorphism.
Now the fact that $\FF\to j_*\Om_{U/S}$ is an isomorphism follows from the diagram \eqref{j-Om-F-diagram}.
%exact sequence \eqref{j-Om-ex-seq}.

\noindent
(ii) From (i) we see that $j_*\Om_{U/S}$ is a coherent sheaf (since this can be checked \'etale locally along $q(S_0)$).
Hence, the kernel $\KK$ and the cokernel $\CC_0$ of the map $\Om_{X_0/S_0}\to j_*\Om_{U/S}|_{X_0}$ are coherent sheaves supported at the node.
Also, we get an explicit description of the completion of $j_*\Om_{U/S}|_{X_0}$ at the node by generators and relations. 
Taking the quotient by the image of $\Om_{X_0/S}$,
we get that $\CC_0$ (which coincides with its completion at the node) is generated by the images $(\ov{e},\ov{f})$ of $(e, f)$, with the defining relations
$$z_1\ov{e}=0, \ \ z_2\ov{e}=0, \ \ z_1\ov{f}=0, \ \ z_2\ov{f}=0, \ \ \th_1\ov{f}=0, \ \ \th_2\ov{f}=0.$$
This implies the assertion that the elements \eqref{CC-basis-eq} constitute an $\OO_{S_0}$-basis of $\CC$.

The kernel $\KK$ coincides with the kernel of the restriction $\Om_{X_0/S_0}\to \Om_{X_1/S_0}\oplus \Om_{X_2/S_0}$. The four elements \eqref{KK-basis-eq}
can be extended to an $\OO_{S_0}$-basis of $\Om_{X_0/S_0}$: we have to add 
$$z_i^ndz_i, \ z_i^n\th_idz_i, \ z_i^nd\th_i, \ z_i^n\th_id\th_i,$$
where $i=1,2$, $n\ge 0$. Since the latter elements project to independent elements of $\Om_{X_1/S_0}\oplus \Om_{X_2/S_0}$,
the assertion follows.
\end{proof}

Next let us consider the case of a Ramond node, i.e., consider similar data $(X/S, t:S\to \A^1, q:S_0\to X_0)$, such that $q(S_0)$ is
a relative Ramond node, and \'etale locally along $q(S_0)$, $\OO_X$ is given by the standard generators $z_1,z_2,\th$ subject 
to $z_1z_2=t$, and with the distribution generated by $\partial_{\th}+\th z_i\partial_{z_i}$ over $z_i\neq 0$.

 %\eqref{NS-node-def-relations}.

\begin{lemma}\label{differential-generators-R-lem}
(i) In the above situation the sheaf $j_*\Om_{U/S}$ is freely generated in a formal neighborhood of $q(S_0)$, as an $\OO_X$-module, by global sections
$$e=\frac{dz_1}{z_1}=-\frac{dz_2}{z_2}, \ \ d\th.$$

\noindent
(ii) We have an exact sequence on $X_0$,
$$0\to \KK\to \Om_{X_0/S_0}\to j_*\Om_{U/S}|_{X_0}\to \CC_0\to 0$$
where $\KK\simeq q_*\OO_{S_0}$ is generated by $z_1dz_2=-z_2dz_1$
and $\CC_0\simeq q_*\OO_{S_0}$ is generated by the image of $e$.
Furthermore, the map of $\OO_{S_0}$-modules $\pi_*\KK\to \pi_*\Om_{X_0/S_0}$ is an embedding of a direct summand.
\end{lemma}

\begin{proof} The proof is similar to that of Lemma \ref{differential-generators-lem} but is much easier since both $\om_{X/S}$ and $\om_{X/S}^2$ are locally free near
the Ramond node. Namely, $\om_{X/S}$ is freely generated by the section glued from $\frac{1}{z_1}[dz_1|d\th]$ and $-\frac{1}{z_2}[dz_2|d\th]$,
The map $\kappa$ sends a generator of $\om_{X/S}^2$ to $e-\th d\th$.
Since $\de$ sends $d\th$ to a generator of $\om_{X/S}$, we deduce that $e$ and $d\th$ freely generate $j_*\Om_{U/S}$ over $\OO_X$. 

The map $\Om_{X_0/S_0}\to j_*\Om_{U/S}|_{X_0}$ sends $dz_1$ to $z_1e$, $dz_2$
to $-z_2e$, and $d\th$ to $d\th$. This easily implies that its kernel is generated by $z_2dz_1=-z_1dz_2$, while its cokernel is generate by the image of $e$.
\end{proof}

\subsection{Boundary Cartier divisor}\label{boundary-equation-sec}

The boundary divisor of the compactified moduli superspace has codimension $1|0$, however, since the supermoduli space is not reduced,
the structure of the Cartier divisor on the boundary is not automatically given. We claim however that there is a natural such structure.

Let $\pi:X\to S$ be a family of stable supercurves, inducing a surjection to the miniversal space of deformations of each node.
We are going to study the natural $2$-term complex
$$[\Om_{X/S}\to j_*\Om_{U/S}]$$
placed in degrees $-1$ and $0$, where $j:U\to X$ is the open embedding of the complement to the nodes.
Note that this complex is acyclic over $U$. Furthermore, both terms are flat over $S$, so
the line bundle $\Ber\pi_*[\Om_{X/S}\to j_*\Om_{U/S}]$ is well defined. Similarly
to the theory in the even case (see \cite{KM}), we can define a canonical trivialization of this line bundle away from the locus of nodal supercurves.

\begin{prop}\label{Ber-boundary-prop}
The canonical trivialization $c$ of $\BB:=\Ber\pi_*[\Om_{X/S}\to j_*\Om_{U/S}]$ away from the nodal locus extends to a regular section of $\BB$ on $S$ that has form $t\cdot b$ near
a point corresponding to a stable supercurve with one NS or Ramond node, for some trivializing local section $b$ of $\BB$ and some local function $t$ whose reduction
modulo nilpotents gives an equation of the corresponding reduced divisor.
%\noindent
%(ii) More generally, in the formal neighborhood of a point of the moduli space corresponding to a supercurve with $k$ nodes, we will get a rational
%trivialization of $\Ber\pi_*[\Om_{X/S}\to j_*\Om_{U/S}]$, away from the special point, which has form $(t_1\ldots t_k)\cdot b$,
%where $t_i$ are equations of the branches of the boundary divisor.
\end{prop}

\begin{proof}
%It is enough to replace $S$ be the formal neighborhood of a point corresponding to a stable supercurve with one node.
We can work in an \'etale neighborhood of a point on $S$ corresponding to a stable supercurve with one node.
Let us first consider the case of an NS node.
We can choose a sufficiently positive relative Cartier divisor $D$ (disjoint from the node) such that $\pi_*[\Om_{X/S}\to j_*\Om_{U/S}]$ will be represented by
the compex of supervector bundles
$$\pi_*(\Om_{X/S}(D))\rTo{\iota} \pi_*(j_*\Om_{U/S}(D)).$$
To understand this map, we first restrict it to $X_0$.
We have an exact sequence on $S$,
$$0\to \pi_*\KK\to \pi_*(\Om_{X_0/S_0}(D))\to \pi_*(j_*\Om_{U/S}(D)|_{S_0})\to \pi_*\CC_0\to 0,$$
where both $\pi_*\KK$ and $\pi_*\CC_0$ are locally free of rank $2|2$.
Furthermore, $\pi_*\KK$ embeds into $\pi_*(\Om_{X_0/S_0}(D))$ as a subbundle, and we have an $\OO_{S_0}$-basis \eqref{KK-basis-eq} of
$\pi_*\KK$. Note that the elements of this basis are defined in terms of the coordinates which only exist locally near the node. However,
they are killed by the ideal $I_q$ of the node, so they can be viewed as sections of $\pi_*\Om_{X_0/S_0}(D)$.

%Thus, we can choose splittings 
%$$\pi_*(\Om_{X/S}(D))=\wt{\KK}\oplus \VV, \ \ \pi_*(j_*\Om_{U/S}(D))=\wt{\CC}\oplus \WW,$$
%so that $\wt{\KK}$ reduces over $S_0$ to the subbundle $\pi_*\KK$, $\iota(\VV)=\WW$, and the composition 
%$$\wt{\CC}|_{X_0}\to \pi_*(j_*\Om_{U/S}(D))|_{S_0}\to \pi_*\CC$$
%is an isomorphism.

We can choose liftings of the $\OO_{S_0}$-basis \eqref{KK-basis-eq} in $\pi_*\KK$ to $\OO_S$-independent sections $b_1,b_2,b_3,b_4$ in $\pi_*(\Om_{X/S}(D))$, so
that
$$b_1\equiv z_1dz_2 \mod(t), \ b_2\equiv \th_2dz_1 \mod(t), \ b_3\equiv \th_1dz_2 \mod(t), \ b_4\equiv \th_2d\th_1 \mod(t).$$
Then we can extend it to a basis $(b_1,\ldots,b_n)$ of $\pi_*(\Om_{X/S}(D))$ so that the elements $(\iota(b_5),\ldots,\iota(b_n))$
form a basis of a free $\OO_S$-submodule in $\pi_*(j_*\Om_{U/S}(D))$, which projects modulo $(t)$ to a basis of the image of
$\pi_*\Om_{X_0/S_0}(D)$.

Next, we adjust our choice of $b_1$ so that
$$b_1\equiv z_1dz_2 \mod(t^2).$$
%where $I$ is the ideal of the node $q(S_0)\sub C$.
Here the right-hand side can be viewed as a section of 
$$\pi_*(\Om_{X/S}(D))/(t^2)=\pi_*(\Om_{X_0^{(1)}/S_0^{(1)}}(D)),$$
where we consider the base change $X_0^{(1)}\to S_0^{(1)}$, where $S_0^{(1)}\sub S$ is given by the ideal $t^2$.
Namely, we observe that near the node one has $I_q^2\sub (z_1,z_2)$, which implies that $z_1dz_2=-z_2dz_1$ is killed by $I_q^2$ on $X_0^{(1)}$
(due to the relation $z_1z_2=-t^2$).
Thus, we can lift $z_1dz_2 \mod(t^2)$ to a section $b_1\in \pi_*(\Om_{X/S}(D))$.

Now let us analyse the images of $b_1,\ldots,b_4$ under $\iota$.
We have $\iota(b_i)=tc_i$ for some $c_i\in \pi_*(j_*\Om_{U/S}(D))$, for $i=2,3,4$.
In the case of $b_1$ we know that $\iota(b_1)=t^2c_1$ for some $c_1\in \pi_*(j_*\Om_{U/S}(D))$.

Furthermore, the restriction of $\iota(b_i)$ to a formal neighborhood $\hat{X}$ of the node is the image
of the restriction of $b_i$ to this formal neighborhood under the similar map for $\hat{X}$.
Since $b_1$ restricts to $z_1dz_2+t^2x$, it maps to 
$$-z_1z_2e+t^2x=t^2(e+x),$$
where $x$ comes from $\Om_{\hat{X}/S}$.
Hence, 
$$c_1\equiv e \mod \Om_{\hat{X}/S}.$$
Similarly, we compute
$$\th_2dz_1=\th_2z_1e=t\th_1e,$$
$$\th_1dz_2=-\th_1z_2e=t\th_2e,$$
$$\th_2d\th_1=tf.$$
Hence, we get
$$c_2\equiv \th_1e\mod \Om_{\hat{X}/S}, \ \ c_3\equiv \th_2e \mod \Om_{\hat{X}/S}, \ \ c_4\equiv f \mod \Om_{\hat{X}/S}.$$
It follows that $c_1,c_2,c_3,c_4$ project to an $\OO_S$-basis of $\pi_*\CC_0$.
Hence, $c_1,c_2,c_3,c_4$ are linearly indpendent over $\OO_S$, and 
can be extended to an $\OO_S$-basis $(c_1,\ldots,c_n)$ of $\pi_*(j_*\Om_{U/S}(D))$ by $c_i=\iota(b_i)$ for $i>4$.
Computing the Berezinian in this basis we get $t$.

The case of Ramond node is considered in a similar way using Lemma \ref{differential-generators-R-lem}.
In this case we can choose bases $(b_1,\ldots,b_n)$ of $\pi_*(\Om_{X/S}(D))$ and $(c_1,\ldots,c_n)$ of $\pi_*(j_*\Om_{U/S}(D))$, such that 
$b_1\equiv z_1dz_2 \mod (t^2)$, $\iota(b_1)=-tc_1$ (since $z_1dz_2=-z_1z_2e=-te$ in $j_*\Om_{U/S}(D)$), and $\iota(b_i)=c_i$ for $i>1$.
This shows that the Berezinian is $t$.
%\noindent
%(ii) ???
\end{proof}

The above Proposition gives a natural definition of the boundary divisor in the moduli space of stable supercurves $\ov{\SS}$ as an effective Cartier divisor.
More precisely, let $k:\SS'\hra \ov{\SS}$ be the complement to the locus of stable supercurves with more than $1$ node. Since
the even codimension of the latter locus is $>1$ and $\ov{\SS}$ is smooth, for any vector bundle $\VV$ over $\ov{\SS}$, the natural map
$\VV\to k_*k^*\VV$ is an isomorphism.
Let us consider the line bundle
$$\LL:=\Ber\pi_*[\Om_{X/\ov{\SS}}\to j_*\Om_{U/\ov{\SS}}]$$
over $\ov{\SS}$.
By Proposition \ref{Ber-boundary-prop}, the canonical trivialization $c$ of $\LL$ over the smooth locus gives a regular global section of $k^*\LL$
and hence of $\LL\simeq k_*k^*\LL$.
We define the effective Cartier divisor $\De$ to be the vanishing divisor of this global section.
Note that by definition, we have an isomorphism
\begin{equation}\label{boundary-line-bundle}
\OO(\De)\simeq \LL\simeq \Ber R\pi_*(j_*\Om_{U/S})\ot \Ber^{-1} R\pi_*(\Om_{X/S}).
\end{equation}

The proof of Proposition \ref{Ber-boundary-prop} also yields the following characterization of the divisor $\De$.

\begin{lemma}\label{Cartier-char-lem} 
The effective Cartier divisor $\De$ is a unique Cartier divisor supported on the locus of nodal supercurves with the following two properties:
\begin{itemize}
\item for any stable supercurve $X_0$ with a single NS node $q$ there exists an \'etale neighborhood $S$ of the corresponding
point $[X_0]$ in the moduli space and an \'etale neighborhood $X_S$ of
$q$ in the family $X\to S$ induced by the universal family such that $\OO_{X_S}$ is generated over $\OO_S$ by $z_1,z_2,\th_1,\th_2$ subject to relations 
\eqref{NS-node-def-relations}, where $t$ is a local equation of $\De$ on $S$;
\item for any stable supercurve $X_0$ with a single Ramond node $q$ there exists an \'etale neighborhood $S$ of $[X_0]$ in the moduli space and an \'etale neighborhood $X_S$ of
$q$ in the family $X\to S$ induced by the universal family such that $\OO_{X_S}$ is generated over $\OO_S$ by $z_1,z_2,\th$ subject to the relation
$z_1z_2=t$, where $t$ is a local equation of $\De$ on $S$;
\end{itemize}
\end{lemma} 

\begin{proof} These conditions clearly characterize $\De\cap \SS'$. To show uniqueness of an extension to $\ov{\SS}$, we observe that if $\II\sub \OO_{\ov{\SS}}$ is
an invertible ideal sheaf then $\II$ is identified with $k_*k^*\II\sub k_*\OO_{\SS'}=\OO_{\ov{\SS}}$.
\end{proof}

%We also claim that the pull-back of $\De$ to the reduced stack
%$\ov{\SS}^{\red}$ corresponds to the reduced scheme structure on its boundary divisor. ???
  
\subsection{Another definition of the boundary divisor and the normal crossing property}\label{normal-cross-sec}

Let $X/\ov{\SS}$ be the universal stable supercurve.

\begin{lemma}
The natural morphism $\AA_X\to \pi^*\TT_{\ov{\SS}}$ of sheaves on $X$
factors through a morphism $\AA_X\to \pi^{-1}\TT_{\ov{\SS}}$.
\end{lemma}

\begin{proof}
This follows the fact that $\AA_X=j_*\AA_U$, where $U\sub X$ is the complement to the nodes (see the proof of Theorem \ref{inf-aut-thm}), 
and from the corresponding statement for the smooth locus $U\to \ov{\SS}$ (see Eq. \eqref{A-X-smooth-case-eq}).
\end{proof}

Let us consider the natural morphism 
\begin{equation}\label{boundary-eq}
\pi_*(\AA_X/\AA_{X/\ov{\SS}})\to \TT_{\ov{\SS}}
\end{equation}
induced by the map $\AA_X\to \pi^{-1}\TT_{\ov{\SS}}$.

\begin{prop} The $\OO_{\ov{\SS}}$-module $\pi_*(\AA_X/\AA_{X/\ov{\SS}})$ is locally free of rank $(3g-3|2g-2)$ over $\ov{\SS}$, and the Cartier divisor
associated with the Berezinian of the map \eqref{boundary-eq} coincides with $\De$. Furthermore, $\De$ is a normal crossing divisor.
\end{prop}

\begin{proof}
%We claim that the Cartier divisor $B_0=(t)$ on $B$ is uniquely defined by the above property. 
%Indeed,  But $B_0$ is determined from this subsheaf as
%the vanishing locus of the map $\Ber(\TT_{B,B_0})\to \Ber(\TT_B)$.
%The same argument works for the boundary divisor corresponding to the Ramond node, using 
Note that \eqref{boundary-eq} is an isomorphism over the locus of smooth supercurves.
By Lemma \ref{KS-super-NS-lem}, in an \'etale neighborhood $S$ of a point $[X_0]$, where $X_0$ has a single NS node, the image of the morphism
$\AA_X\to \pi^{-1}\TT_S$ is $\pi^{-1}\TT_{S,S_0}$, where $S_0$ is the divisor $(t)$. 
Hence, in this case \eqref{boundary-eq} is the embedding $\TT_{B,B_0}\hra \TT_B$, and the Berezinian of this morphism is $t$.
By Lemma \ref{KS-super-R-lem}, the similar statement holds in a neighborhood of a curve with a single Ramond node.
This implies that $\pi_*(\AA_X/\AA_{X/\ov{\SS}})$ is locally free over $\SS'\sub \ov{\SS}$, and the divisor of the Berezinian of \eqref{boundary-eq} over $\SS'$ is
$\De\cap \SS'$.
 
%In other words, at the generic point of the boundary divisor of the moduli space $\ov{\SS}$, where only one node is present, the $\OO_{\ov{\SS}}$-module
%$\pi_*(\AA_X/\AA_{X/\ov{\SS}})$ (where $\pi:X\to \ov{\SS}$ is the universal supercurve) is locally free, and the Berezinian of the natural map
%\begin{equation}\label{boundary-eq}
%\pi_*(\AA_X/\AA_{X/\ov{\SS}})\to \TT_{\ov{\SS}}
%\end{equation}
%gives a local equation of the boundary divisor. 

%More generally, we claim that this is true everywhere on $\ov{\SS}$.

Nex, let us study the situation near a point $[X_0]$ of $\ov{\SS}$, such that $[X_0]$ has several nodes $q_1,\ldots,q_k$.
Using the fact that the map from deformations of $[X_0]$ to the product of deformation spaces of the nodes is smooth, we see that there exists
an \'etale neighborhood $B$ of $[X_0]$, together with a smooth map $(t_1,\ldots,t_k):B\to \A^k$ such that the function $t_i$ corresponds the induced deformation of
the node $q_i$.

Let us consider the normal crossing divisor $D=(t_1\ldots t_k=0)$ in $B$, and let $\TT_{B,D}\sub \TT_B$ be the subsheaf of derivations preserving $D$.
Note that $\TT_{B,D}=\cap \TT_{B,D_i}$ where $D_i$ is the divisor $t_i=0$.

Let us consider the restriction $\pi:X\to B$ of the universal family and set 
$$\ov{\AA}_X:=\AA_X/\AA_{X/B}\sub \pi^*\TT_B.$$
Let $U_i$ be an open neighborhood of $q_i$ in $X$, and let $\pi_i=\pi|_{U_i}$. Then, as we have seen above,
we have
$$\ov{\AA}_X|_{U_i}=\pi_i^{-1}\TT_{B,D_i}.$$
Thus, for each $i=1,\ldots,k$, we have an inclusion
$$\pi_*(\ov{\AA}_X)\sub \pi_{i*}(\ov{\AA}_X|_{U_i})=\TT_{B,D_i}\sub \TT_B.$$
Hence, we deduce the inclusion
$$\pi_*(\ov{\AA}_X)\sub \TT_{B,D}.$$

On the other hand, we claim that there is an inclusion
$$\pi^{-1}\TT_{B,D}\sub \ov{\AA}_X.$$
Indeed, it is enough to check this over each $U_i$. But then we have
$$\pi^{-1}\TT_{B,D}|_{U_i}\sub \pi_i^{-1}\TT_{B,D_i}=\ov{\AA}_X,$$
as required.
Hence, passing to $\pi_*(?)$, we derive the inclusion
$$\TT_{B,D}=\pi_*\pi^{-1}\TT_{B,D}\sub \pi_*\ov{\AA}_X.$$
as claimed. 

Thus, we get $\pi_*\ov{\AA}_X=\TT_{B,D}$. 
Hence, $\pi_*\ov{\AA}_X$ is locally free and the Berezinian of the morphism $\TT_{B,D}\to \TT_B$ has the required form.
\end{proof}

%\subsection{Immersions of codimension $1|0$}

%Let $f:M\to N$ be a morphism of smooth supervarieties inducing injective maps on tangent spaces, where
%$\dim M=m|n$, $\dim N=(m+1)|n$. 
%Let us consider the ideal $\II_f$ of functions $\varphi\in \OO_N$ such that $f^*\varphi=0$.

%\begin{lemma} The ideal $\II_f$ is locally generated by one even function $\varphi_0$ with everywhere nonvanishing differential.
%\end{lemma}

%\begin{proof} 
%???
%\end{proof}

%Let us consider the scheme image $M'\sub N$ of $f$. Then 
%It is enough to consider a surjection of local rings of some points, 
%$$f: A[\th_1,\ldots,\th_n]\to B[\psi_1,\ldots,\psi_n]$$
%induced by $f$. Here $A$ and $B$ are regular local rings, $\dim A=m+1$, $\dim B=m$.
%Let $I\sub A[\th_1,\ldots,\th_n]$ denote the kernel (i.e., the stalk of $\II_f$).
%Since the induced morphism on odd cotangent spaces is an isomorphism,
%we can assume that $f(\th_i)=\psi_i$.
%We have an exact sequence

\subsection{NS and Ramond boundary components as effective Cartier divisors}\label{De-NS-R-sec}

Recall that we denote by $\SS'\sub \ov{\SS}$ the open locus of stable supercurves with at most one node.
Note that $\De\cap \SS'$ is the disjoint union of two Cartier divisors, one supported on stable supercurves with one NS node and another supported on
those with a Ramond node. We want to extend these two divisors on $\SS'$ to effective Cartier divisors $\De_{NS}$ and $\De_R$ on $\ov{\SS}$ such that
$\De=\De_{NS}+\De_R$.

First, we are going to define the divisor $\De_{NS}$ giving the NS component.
For this let us consider the structure map 
$$\de:\Om_{X/\ov{\SS}}\to \om_{X/\ov{\SS}}(R).$$
As we have seen before, it is surjective away from the nodes.
It is also surjective on Ramond nodes.
Let us consider the ideal sheaf on $\ov{\SS}$,
$$\II_\de:=\Ann \pi_*\coker(\de),$$
supported on the locus of stable supercurves with at least one NS node.

\begin{prop} 
%Assume that the map from $S$ to the moduli space assocaited with the family $X/S$ is a submersion
%(surjective on tangent spaces). Let $\II_b$ denote the ideal of the effective Cartier divisor supported on the locus of nodal curves in $S$,
%defined in Sec.\ \ref{boundary-equation-sec}.
(i) The ideal sheaf $\II_\de$ defines an effective Cartier divisor $\De_{NS}$ on $\ov{\SS}$, which coincides with $\De$
in a neighborhood of any point $[X_0]$ corresponding to a stable supercurve $X_0$ with only NS nodes.

\noindent
(ii) There exists a unique effective Cartier divisor $\De_R$ supported on the locus of stable supercurves with at least one Ramond node,
such that 
$$\De=\De_{NS}+\De_R.$$
If $X_0$ is a stable supercurve with NS nodes $q_1,\ldots,q_r$ and Ramond nodes $q_{r+1},\ldots,q_k$, then there exists an \'etale neighborhood $B$ of
$[X_0]$ in $\ov{\SS}$ with a smooth morphism $t_1,\ldots,t_k:B\to \A^k$ such that $t_1\ldots t_r$ is an equation of $\De_{NS}$ and $t_{r+1}\ldots t_k$ is an equation
of $\De_R$.
%and the ideal of the vanishing locus $\div(b_0)$ of the canonical section $b_0$ of
%$\Ber\pi_*[\Om_{X/S}\to j_*\Om_{U/S}]$ defined in Proposition \ref{Ber-boundary-prop}.
%Then the ideal sheaf $\II_\de$ is locally generated by an even non-zero-divisor,
%so it defines a Cartier divisor supported on the locus in $S$ corresponding to curves with NS nodes.
\end{prop}

\begin{proof} Let $X_0$ be a stable supercurve with NS nodes $q_1,\ldots,q_r$ and Ramond nodes $q_{r+1},\ldots,q_k$.
Consider an \'etale neighborhood $B$ of $[X_0]$, equipped with a smooth map $(t_1,\ldots,t_k):B\to \A^k$,
such that $t_i$ gives the universal deformation of the node $q_i$.  
%Hence, it is enough to consider a family of stable supercurves $X/S$ over the formal super-polydisk with a special even coordinate $t$ such that over $S_0=(t=0)\sub S$,
%we have a relative NS node $q:S_0\to X_0=S_0\times S X$, so that the completion of $\OO_X$ along $q(S_0)$ is given by the standard generators and relations 
%\eqref{NS-node-def-relations}. 
Then the ideal of $\De$ is generated by $t_1\ldots t_k$. 

On the other hand, using the description of $\om_{X/S}$ by generators and relations (see Lemma \ref{NS-om-lem}), we see that in the formal neighborhood of each NS node $q_i$,
$\coker(\de)$ is generated by $s_0$ subject to the relations 
$$\th_1s_0=\th_2s_0=z_1s_0=z_2s_0=t_is_0=0.$$
It is easy to deduce from this that the annihilator of $\pi_*\coker(\de)$ is generated by $t_1\ldots t_r$.
This implies all the assertions.
%Hence, $\coker(\de)$ is isomorphic to the structure sheaf of the node $q(S_0)$, so $\pi_*\coker(\de)$ is isomorphic to $\OO_{S_0}$, so $\II_d=(t)$.
%Furthermore, one can check by the explicit computation that the assertion holds for the universal deformations
%of both types of nodes.
%Indeed,
%see from the local description of the nodes that for each node $q$ of $X_s$
%one has $\coker(\de)|_{X_s}\simeq \OO_{X_s}/\II_q$ near $q$.
%???
\end{proof}

\subsection{The NS node boundary components}\label{NS-boundary-gluing-sec}

We will use the following construction of {\it gluing two superschemes along a closed subscheme}.

Assume that $X$, $X'$ and $Y$ are superschemes, $i:Y\to X$ and $i:Y\to X'$ are closed embeddings.
We want to define a new superscheme $Z$ by gluing $X$ and $X'$ along $Y$. As a topological space we
can define $Z$ to the usual gluing of the topological spaces of $X$ and $X'$ along $Y$, so that we have
closed embeddings $j:X\to Z$, $j':X'\to Z$, so that $k:=j_1\circ i=j_2\circ i$. Hence, we have two homomorphisms
of sheaves of rings
$$\phi:j_*\OO_X\to k_*\OO_Y, \ \ \phi':j'_*\OO_{X'}\to k_*\OO_Y,$$
and we define $\OO_X$ to be the subsheaf of $j_*\OO_X\oplus j'_*\OO_{X'}$, namely, the preimage of the diagonal
$k_*\OO_Y\sub k_*\OO_Y\oplus k_*\OO_Y$ under $\phi\oplus \phi'$. 

We can apply this to gluing two stable supercurves along NS-punctures.

Suppose $X/S$ and $X'/S$ is a pair of stable supercurves with NS-punctures $P\sub X$ and $P'\sub X'$.
We have an isomorphism $P\simeq S\simeq P'$, so we can define a new superscheme $Z/S$ by gluing $X$ and $X'$
along $P\simeq P'$.
%Equivalently, we can define $Z$ as a closed subscheme of $X\times_ S X'$ given by the ideal 
%$\II_{P\times_S X'}\II_{X\times_S P'}$.

\begin{lemma}
The glued superscheme $Z/S$ has a natural stable supercurve structure such that the derivation $\de$ on $Z$ is defined as the composition
$$\OO_Z\to j_*\OO_X\oplus j'_*\OO_{X'}\rTo{(j_*\de,j'_*\de')} j_*\om_{X/S}\oplus j'_*\om_{X'/S}\to \om_{Z/S}.$$
\end{lemma}

\begin{proof}
If $U\sub X$, $U'\sub X'$ are smooth loci, then $(U\setminus P)\sqcup (U'\setminus P')$ is an open subset of $Z$, which is a smooth supercurve over $S$.
Next, we need to check that in the case when $S$ is a point, $\de^-$ induces an isomorphism of $\OO_Z^-$ with $\om_Z^-$.
Since the base is even, we have identifications
$$\OO_X=\OO_C\oplus L, \ \ \OO_{X'}=\OO_{C'}\oplus L',$$
and the smooth marked points $P\sub C$, $P'\sub C'$, so that the embedding of $P$ into $X$ corresponds to the projection
$$\OO_X\to \OO_C\to \OO_P.$$
Note that $L$ and $L'$ are generalized spin-structures, i.e., we have an isomorphism $L\rTo{\sim}\und{\Hom}(L,\om_C)$ (which corresponds to $\de^-$ on $X$), 
and similarly, for $L'$. Furthermore, we know that $L$ (resp., $L'$) is locally free near $P$ (resp., $P'$).

The glued superscheme $Z$ has the underlying nodal curve $Z_0$, which is glued from $C$ and $C'$ along $P\simeq P'$,
and 
$$\OO_Z\simeq \OO_{Z_0}\oplus j_*L\oplus j'_*L'.$$
It is well known that the natural map
$$j_*L\oplus j'_*L'\to \und{\Hom}(j_*L\oplus j'_*L',\om_{Z_0})$$
is an isomorphism, i.e., $j_*L\oplus j'_*L'$ is a generalized spin-structure.
It is easy to see that above map is precisely $\de^-$, so this proves that $Z/S$ is a stable supercurve.
\end{proof}

Let $\ov{\SS}_{g;m,n}$ denote the moduli superspace of stable supercurves of genus $g$ with $m$ NS marked points and $n$ Ramond marked points.
Then we can apply the above gluing construction to the pair of families of stable supercurves over
$\ov{\SS}_{g_1;m_1+1,n_1}\times \ov{\SS}_{g_2;m_2+1,n_2}$, pulled back from each factor and using the last NS puncture on each of them.
This leads to a morphism
\begin{equation}\label{NS-sep-node-morphism}
\ov{\SS}_{g_1;m_1+1,n_1}\times \ov{\SS}_{g_2;m_2+1,n_2}\to \ov{\SS}_{g_1+g_2;m_1+m_2,n_1+n_2}.
\end{equation}

Similarly, if $X/S$ is a stable supercurve with two disjoint NS-punctures $P,P'\sub X$, then we can glue $P$ with $P'$ and get a new stable supercurve with
a non-separating node. This leads to a morphism 
\begin{equation}\label{NS-nonsep-node-morphism}
\ov{\SS}_{g;m+2,n}\to \ov{\SS}_{g+1;m,n}
\end{equation}

\begin{lemma}\label{NS-boundary-tangent-emb-lem}
%(i) 
Both morphisms \eqref{NS-sep-node-morphism} and \eqref{NS-nonsep-node-morphism} induce embeddings of codimension $1|0$ on tangent spaces and factor
through the divisor $\De_{NS}$. Furthermore, near a stable supercurve with a single separating (resp., non-separating)
NS node, the divisor $\De_{NS}$ coincides with the schematic image of \eqref{NS-sep-node-morphism} (resp., \eqref{NS-nonsep-node-morphism}).
%\noindent
%(ii)The morphism \eqref{NS-sep-node-morphism} is generically a closed embedding unless $g_1=g_2$, $m_1=m_2$ and $n_1=n_2$,
%in which case it is generically $2:1$. The morphism \eqref{NS-nonsep-node-morphism} is generically $2:1$.
\end{lemma}

\begin{proof} 
%(i) 
The idea is to use exact sequences \eqref{def-curve-nodes-ex-seq}. 
%Let us consider the case of a separating node first, so
Let $X$ be a stable supercurve with punctures and a fixed NS-node $q\in X$, and let $\rho:\wt{X}\to X$ be the normalization at $q$, equipped with the two NS punctures over $q$.
Then by Proposition \ref{no-inf-aut-prop}(i), we have
$$\AA_X\simeq \rho_*\AA_{\wt{X}}$$
(where we take into account all the punctures on both sides).
Hence, the natural map
$$H^1(\wt{X},\AA_{\wt{X}})\to H^1(X,\AA_X)$$
is an isomorphism.

Let $q_1,\ldots,q_m$ be the nodes of $X$ different from $q$. Then we have a morphism of exact sequences
\begin{diagram}
0&\rTo{}&H^1(\wt{X},\AA_{\wt{X}})&\rTo{}& T_{\Def(\wt{X})} &\rTo{}& \bigoplus_{i=1}^m T_{\Def(\OO_{X,q_i})} &\rTo{}& 0\\
&&\dTo{}&&\dTo{}&&\dTo{}\\
0&\rTo{}&H^1(X,\AA_X)&\rTo{}& T_{\Def(X)} &\rTo{}& T_{\Def(\OO_{X,q})}\oplus\bigoplus_{i=1}^m T_{\Def(\OO_{X,q_i})} &\rTo{}& 0
\end{diagram}
in which the right vertical arrow is the natural inclusion of codimension $1|0$. Since the left vertical arrow is an isomorphism, the assertion about the
map of tangent spaces follows.

To check that our morphisms factor through $\De_{NS}$, we have to check that the pull-back of the morphism $\OO\to \OO(\De_{NS})$ is zero.
Since the sources are smooth, it is enough to check this generically, so we can consider a neighborhood of the point corresponding to a supercurve with a single
NS node. Then we know that the local equation of $\De_{NS}$ will be $(t=0)$, where $t$ corresponds to the map to the universal deformation of the node.
Since the deformation of the node given by the source of the maps \eqref{NS-sep-node-morphism} and \eqref{NS-nonsep-node-morphism} is trivial, this proves that the pull-back of $t$ with respect to one of these maps is zero. Using the fact that the morphism on tangent spaces is an embedding with the image which is the orthogonal to $dt$, we deduce
that near this point $\De_{NS}$ is the schematic image of the morphism.
%Near a generic point of the image of either \eqref{NS-sep-node-morphism} and \eqref{NS-nonsep-node-morphism}
%\noindent
%(ii) ???
\end{proof}

\subsection{The Ramond node boundary components}\label{R-boundary-gluing-sec}

Gluing along two Ramond punctures is more subtle.

First, we point out a certain ``residual structure" we have on each Ramond puncture.
Let $R\sub X$ be a Ramond puncture in a stable supercurve $X/S$.

\begin{lemma}\label{R-puncture-structure-lem}
The map
$$\de|_R:\Om_{X/S}|_R\to \om_{X/R}(R)|_R\simeq \om_{R/S}$$
factors through the canonical projection $\Om_{X/S}|_R\to \Om_{R/S}$, and induces
an isomorphism
$$\ov{\de}_R:\Om_{R/S}\to \om_{R/S}.$$ 
This induces a trivialization of $(\Om_{R/S})^{\ot 2}$ (or equivalently, of $\om_{R/S}^{\ot 2}$)
which is locally given by $(d\th)^2$, for $\th$ such that
$\OO_R=\OO_S[\th]$, 
\end{lemma}

\begin{proof}
Recall that in appropriate \'etale local coordinates $(z,\th)$, $R$ is given by $(z=0)$ and we have 
$$\de(dz)=z\th [\frac{dz}{z}|d\th], \ \ \de(d\th)=  [\frac{dz}{z}|d\th].$$
It follows that $\de|_R(dz)=0$, and
the induced map $\ov{\de}_R$ is given by
$$\ov{\de}_R(d\th)=\ov{b},$$
where $\ov{b}$ is the generator of $\om_{R/S}$ corresponding to $[\frac{dz}{z}|d\th]$.

Since the relative dimension of $R$ is $0|1$, we have in fact an isomorphism 
$\om_{R/S}^{-1}\simeq \Om_{R/S}$. Thus, we can view $\ov{\de}_R$ as a trivializing section of $(\Om_{R/S})^{\ot 2}$.
%Now our trivialization $[d\th]\in \om_{R/S}$ should satisfy $[d\th]^2=\ov{\de}_R$. 
\end{proof}

The above Lemma implies that every Ramond puncture $R$ has a preferred system of \'etale local relative coordinates $\th_i$, such that over intersections one has
$\th_j=\pm \th_i+a_{ij}$, where $a_{ij}$ are functions on the base. Namely, we require that $(d\th_i)^2$ is the canonical trivialization of $(\Om_{R/S})^{\ot 2}$.
In other words, we have a canonical principal bundle $\PP_R\to S$ with the structure supergroup $\Z/2\ltimes \A^{0|1}$.

%there exists an \'etale double covering $\wt{S}\to S$ such that the pull-back family $\wt{R}\to \wt{S}$ has a structure of a principal $\A^{0|1}$-bundle.

We can restate the above structure in more invariant terms. Let $\pi:R\to S$ denote the projection.
First, we observe that 
\begin{equation}\label{PhiR-def}
\Phi_R:=\pi_*\OO_R/\OO_S\simeq \Ber^{-1}\pi_*\OO_R
\end{equation}
is a line bundle of rank $0|1$ on $S$. Furthermore, we have a canonical isomorphism of odd line bundles on $R$,
$$\pi^*\Phi_R\rTo{\sim} \Om_R$$
induced by the de Rham differential $\OO_S/\OO_R\to \Om_R$. Furthermore, the trivialization of $(\Om_R)^{\ot 2}$
in Lemma \ref{R-puncture-structure-lem} actually comes from the canonical trivialization of
$\Phi_R^2$ on $S$.

Let $\Theta\sub \tot(\Phi_R)$ be the $\Z/2$-torsor over $S$ corresponding to $\Phi_R$.
Then the principal bundle $\PP_R\to S$ can be identified with the preimage of $\Theta$ in $\tot(\pi_*\OO_R)$ under
the natural projection $p:\pi_*\OO_R\to \pi_*\OO_R/\OO_S$.

\begin{lemma}\label{Aut-PhiR-lem}
The group scheme $\Aut(\Phi_R)\to S$ can be identified with the group of automorphisms of $\pi_*\OO_R$, which are
identity on $\OO_S$, and induce $\pm \id$ on $\pi_*\OO_R/\OO_S$. This group scheme is an extension of $\Z/2$ by
the line bundle $\Phi_R^{-1}$.
\end{lemma}

%\begin{proof}
%\end{proof}

Now suppose $X/S$ is a (possibly disconnected) stable supercurve  with two R-punctures $R\sub X$ and $R'\sub X$,
equipped with an isomorphism $R\simeq R'$ over $S$. 
Then we can glue $X$ with itself along $R\simeq R'$ into a superscheme $Z$ over $S$ equipped with a finite morphism $j:X\to Z$.

\begin{lemma}\label{Ramond-gluing-lem} 
Assume that the isomorphism $\a:R\rTo{\sim} R'$ is such that
the induced isomorphism
$$\a^*\om_{R'/S}^{\ot 2}\rTo{\sim} \om_{R/S}^{\ot 2}$$
is equal to $-1$, where we use the trivializations of $\om_{R/S}^{\ot 2}$ and
$\om_{R'/S}^{\ot 2}$ coming from the supercurve structure on $X$.
Then the glued superscheme $Z/S$ has a natural stable supercurve structure, such that the derivation $\de$ is
uniquely determined from the commutative diagram
\begin{diagram}
\OO_Z &\rTo{\de}& \om_{Z/S}\\
\dTo{r}&&\dTo{}\\
j_*\OO_X&\rTo{j_*\de}& j_*\om_{X/S}(R)
\end{diagram}
\end{lemma}

\begin{proof} To check that $\de$ is well defined, we need to check that the image of the composition
$j_*\de\circ r$ belongs to $\om_{Z/S}$.
For this we can argue locally. Let $(z,\th)$ be the standard local coordinates on $X$ near the puncture $R$,
so that $R$ is given by $(z=0)$, and
$$\de(dz)=z\th [\frac{dz}{z}|d\th], \ \ \de(d\th)=  [\frac{dz}{z}|d\th].$$
Let also $(z',\th')$ be similar coordinates near $R'$. Then $\a^*(d\th')^2=-d\th$, so changing $\th$ to $\pm\th+a$, we can assume
that $i\a^*(\th')=\th$.
Let us set $\ov{\th}=i\th'$. Then we have
$$\de(dz')=z'\ov{\th} [\frac{dz}{z}|d\ov{\th}], \ \ \de(d\ov{\th})=-[\frac{dz'}{z'}|d\ov{\th}].$$

Then we have relative coordinates on the glued scheme $Z$, $(z,z',\th)$, where $zz'=0$, $\th$ restricts to $\th$ near $R$ and to $\ov{\th}$ near $R'$.
Since $\de(z)$ has no pole at $R$, we see that $(\de(z),0)$ belongs to $\om_{Z/S}$. Similarly, $(0,\de'(z'))$ belongs to $\om_{Z/S}$.
Finally, 
$$(\de(\th),\de'(\ov{\th}))=([\frac{dz}{z}|d\th],-[\frac{dz'}{z'}|d\ov{\th}]),$$
which comes from a section of $\om_{Z/S}$.

Over the point, our construction can be recast as follows. We start with a Ramond spin curve $(C,p,p',L)$, together with an identification
$$\a:L|_p\simeq L|_{p'},$$
such that $\a^2=-1$ (where we use trivializations of $\om_C(p)|_p$ and $\om_{C}(p')|_{p'}$).
We glue $p$ with $p'$ and get a nodal curve $Z_0$, and then descend $L$ to a line bundle $\LL$ over $Z_0$ using $\a$.
Then we have a natural isomorphism $\LL^{\ot 2}\simeq\om_{Z_0}$, and one can easily check that the corresponding isomorphism 
$\LL\rTo{\sim} \und{\Hom}(\LL,\om_{Z_0})$ is induced by $\de^-$, where $\de$ is defined as above.
\end{proof}

The choice of an isomorphism $\a$ can be interpreted in terms of the principal $\Z/2\ltimes\A^{0|1}$-bundles $\PP_R$ and $\PP_{R'}$ as follows.
For every $c\in \C^*$ let us denote by $[c]$ the automorphism of $\Z/2\ltimes\A^{0|1}$ given by the rescaling by $c$ on $\A^{0|1}$ (and trivial on $\Z/2$).
Then a choice of $\a$ is equivalent to a choice of an isomorphism of $\Z/2\ltimes\A^{0|1}$-bundles,
$$\PP_R\to [i]_*\PP_{R'},$$
where $[i]_*\PP_{R'}$ is the push-out of $\PP_{R'}$ with respect to the automorphism $[i]$.

We can apply the above construction to the two last Ramond punctures $R_{n+1}$ and $R_{n+2}$ on the universal stable supercurve over
$\ov{\SS}_{g;m,n+2}$.
Let $\PP_{n+1}\to \ov{\SS}_{g;m,n+2}$ and $\PP_{n+2}\to \ov{\SS}_{g;m,n+2}$ be the corresponding
principal $\Z/2\ltimes \A^{0|1}$-bundles.
Let $\Isom(\PP_{n+1},[i]_*\PP_{n+2})\to \ov{\SS}_{g;m,n+2}$ denote the bundle of isomorphisms between these $\Z/2\ltimes \A^{0|1}$-bundles.
Note that it is also a principal bundle with the group $\Z/2\times \A^{0|1}$.

Then the above gluing construction gives a family of stable supercurves with a Ramond node over $\Isom(\PP_{n+1},[i]_*\PP_{n+2})$, so we get a morphism
\begin{equation}\label{Ramond-nonsep-node-morphism}
\Isom(\PP_{n+1},[i]_*\PP_{n+2})\to \ov{\SS}_{g+1;m,n}
\end{equation}

Similarly, we can glue the last Ramond punctures $R,R'$ on the stable supercurves $X,X'$ over the moduli space 
$\ov{\SS}_{g_1;m_1,n_1+1}\times \ov{\SS}_{g_2;m_2,n_2+1}$ pulled back from each factor.
This gives a morphism
\begin{equation}\label{Ramond-sep-node-morphism}
\Isom(\PP_R,[i]_*\PP_{R'})\to \ov{\SS}_{g_1+g_2;m_1+m_2,n_1+n_2}.
\end{equation}
Note that in this case $n_1$ and $n_2$ are odd, so $n_1+n_2\ge 2$.

\begin{lemma}
Both morphisms \eqref{Ramond-nonsep-node-morphism} and \eqref{Ramond-sep-node-morphism} induce embeddings of codimension $1|0$ on tangent spaces
and factor through the divisor $\De_R$. Furthermore, near a stable supercurve with a single non-separating Ramond node the divisor $\De_R$ coincides with the schematic image
of \eqref{Ramond-nonsep-node-morphism}.
\end{lemma}

\begin{proof}
Let $X$ be the stable supercurve obtained by gluing two Ramond punctures $R, R'$ on a stable supercurve $\wt{X}$, via an isomorphism
$\a:R\to R'$ as in Lemma \ref{Ramond-gluing-lem}. 
Note that we have an exact sequence
$$0\to \Pi\C\to T_{\Def(\wt{X},\a)}\to T_{\Def(\wt{X})}\to 0,$$
where the odd line $\Pi\C$ corresponds to infinitesimal deformations of $\a$.
By Proposition \ref{no-inf-aut-prop}(i), we have
$$\AA_X^+\simeq \rho_*\AA_{\wt{X}}^+.$$
As in Lemma \ref{NS-boundary-tangent-emb-lem} we deduce from this that
$$T^+_{\Def(\wt{X},\a)}\to T^+_{\Def(X)}$$
is an embedding of codimension $1$.

Next, we claim that the map between the spaces of odd infinitesimal deformations,
$$T^-_{\Def(\wt{X},\a)}\to T^-_{\Def(X)}$$
is an isomorphism. Indeed, we know that every odd deformation of $X$ is locally trivial, i.e., 
$$T^-_{\Def(X)}\simeq H^1(X,\AA_X^-).$$
Also, by Proposition \ref{no-inf-aut-prop}(i),
we have an exact sequence
$$0\to \C\to H^1(X,\AA_X^-)\to H^1(\wt{X},\AA_{\wt{X}}^-)\to 0$$
so 
$$\dim T^-_{\Def(X)}= \dim H^1(\wt{X},\AA_{\wt{X}}^-)+1=\dim T^-_{\Def(\wt{X},\a)}.$$
%Hence, the map between the spaces of odd infinitesimal deformations is an isomorphism.
It remains to observe that every locally trivial deformation of $X$
can be lifted to a deformation of $(\wt{X},\a)$.
\end{proof}

\section{Canonical line bundle on the moduli space of stable supercurves}\label{canonical-bundle-sec}

In this section we will calculate the canonical line bundle of $\ov{\SS}_{g,n_{NS},n_R}$, eventually proving Theorem B.

\subsection{What Kodaira-Spencer isomorphism gives for the canonical bundle}

Set $\ov{\SS}=\ov{\SS}_{g,n_{NS},n_R}$. We are interested in the canonical line bundle $K_{\ov{\SS}}=\Ber^{-1}(\TT_{\ov{\SS}})$.
Over the smooth locus $\SS\sub \ov{\SS}$, from the Kodaira-Spencer isomorphism we get an isomorphism (see Sec.\ \ref{Kodaira-punct-sec})
\begin{equation}\label{smooth-can-KS-formula-eq}
\ber(KS_{\ov{\SS}}):K_{\SS}\rTo{\sim} \Ber^{-1}(R^1\pi_*\LL(X,P_\bullet,R_\bullet))\simeq \Ber(R\pi_*\LL(X,P_\bullet,R_\bullet)),
\end{equation}
where
$$\LL(X,P_\bullet,R_\bullet):=\om_{X/S}^{-2}(-\sum_{i\in I} D_i-2\sum_{j\in J} R_j).$$
Recall that 
here $\om_{X/S}^{-2}$ is the line bundle defined in Theorem \ref{can-square-thm}, and
$D_i$ are the divisors on the universal curve $X$ associated with the NS punctures.

Our study in Sec.\ \ref{KS-map-sec}
of the behavior of the Kodaira-Spencer map in degenerating families of supercurves leads to the following identification of $K_{\ov{\SS}}$.

\begin{prop}\label{stable-can-KS-formula-prop}
One has a natural isomorphism 
$$\ber(KS_{\ov{\SS}})^{-1}:K_{\ov{\SS}}\rTo{\sim} \Ber(R\pi_*\LL(X,P_\bullet,R_\bullet))(-\De),$$
where $\De$ is the boundary divisor.
\end{prop}

\begin{proof}
We just have to check that isomorphism \eqref{smooth-can-KS-formula-eq} acquires simple poles at all generic points of $\De$.
But this follows from the results of Sec. \ \ref{KS-map-sec} and our local description of $\De$ at points corresponding to stable 
supercurves with single nodes.
\end{proof}

Next, we will study the line bundle $\Ber(R\pi_*\LL(X,P_\bullet,R_\bullet))$. We begin with the case 
when there are no punctures. Then the above identification of the canonical bundle becomes
%\begin{equation}\label{stable-can-KS-formula-eq}
$$\ber(KS_{\ov{\SS}})^{-1}:K_{\ov{\SS}}\rTo{\sim} \Ber(R\pi_*(\om_{X/S}^{-2})).$$
%\end{equation}

\subsection{Isomorphisms for Berezinian bundles}\label{super-Deligne-sec}

Let $\pi:X\to S$ be a smooth proper morphism of relative dimension $1|1$. Then we can define Deligne's symbol
$\lan L_1$, $L_2\ran$ of a pair of line bundles of rank $1|0$
over $X$ similarly to the even case (considered in \cite{Deligne-det}), so that for a relative effective Cartier divisor $D$
one has
$$\lan \OO_X(D),L\ran\simeq \Ber(\pi_*(L|_D))/\Ber(\pi_*\OO_D).$$

If $\BB(L)$ denotes the Berezinian of the derived push-forward of $L$ then one has, as in the classical case,
$$\BB(L_1\ot L_2)\ot\BB(\OO)\simeq \BB(L_1)\ot\BB(L_2)\ot\lan L_1,L_2\ran.$$
However, in the supercase,
for any line bundles $L$ and $L'$ of rank $1|0$,
there is a canonical isomorphism
\begin{equation}\label{Ber-L-D-eq}
\a_D:\Ber(\pi_*L|_D)\rTo{\sim}\Ber(\pi_*(L'|_D)),
\end{equation}
induced by any local isomorphism $L\to L'$ (the point is that the Berezinian of a scalar automorphism of a linear space of rank $1|1$ is trivial). Hence, the Deligne's symbol is canonically trivial, so we get
\begin{equation}\label{super-Deligne-eq}
\BB(L_1\ot L_2)\ot\BB(\OO)\simeq \BB(L_1)\ot\BB(L_2)
\end{equation}
(see \cite{Voronov}).

More explicitly, locally over $S$ we can pick a relative positive divisor $D$ such that $L_1L_2(D)$, $L_2(D)$ and $\OO_X(D)$ have no $R^1\pi_*$,
and assume also that we have an even section $s\in H^0(X,L_2(D))$, fiberwise regular, vanishing on a relative divisor $E$.
Then we have the resolutions
\begin{align*}
&R\pi_*(\OO_X): & [\pi_*L_2(D)\to \pi_*L_2(D)|_E], \\
&R\pi_*(L_1): & [\pi_*L_1L_2(D)\to \pi_*L_1L_2(D)|_E], \\
&R\pi_*(L_2): & [\pi_*L_2(D)\to \pi_*L_2(D)|_D], \\
&R\pi_*(L_1L_2): & [\pi_*L_1L_2(D)\to \pi_*L_1L_2(D)|_D].
\end{align*}
Using these resolutions we get an isomorphism
\begin{align*}
&\BB(L_1L_2)\ot \BB(\OO_X)=\BB(L_1L_2(D))\ot \BB(L_2(D))\ot \Ber^{-1}(\pi_*L_1L_2(D)|_D)\ot \Ber^{-1}(\pi_*L_2(D)|_E)\\
&\rTo{\id\ot\id\ot\a^{-1}_D\ot\a^{-1}_E}
\BB(L_1L_2(D))\ot \BB(L_2(D))\ot \Ber^{-1}(\pi_*L_2(D)|_D)\ot \Ber^{-1}(\pi_*L_1L_2(D)|_E)=\BB(L_1)\ot \BB(L_2).
\end{align*}

Note that isomorphism \eqref{super-Deligne-eq} can be used to calculate $\BB(L_1L_2)$ when $L_1$ and $L_2$ are not necessarily of
rank $1|0$: for a line bundle $L$ of rank $0|1$ we use the isomorphism
$$\BB(\Pi L)\simeq \BB(L)^{-1}.$$

For example, for a line bundle $L$ of rank $1|0$ one has
$$\BB(L^3)\simeq \BB(L)^3/\BB(\OO)^2.$$
Applying this to $L=\Pi \om_X$, where $X/S$ is a smooth supercurve,
one gets the analog of Mumford isomorphism in the supercase (see \cite{Voronov}),
\begin{equation}\label{super-Mumford-eq}
\BB(\om_{X/S}^3)\simeq \BB(\om_{X/S})^3\ot \BB(\OO_X)^2\simeq \BB(\om_{X/S})^5,
\end{equation}
since $\BB(L)\simeq \BB(\om_{X/S}\ot L^{-1})$ by Grothendieck-Serre duality.
Thus, denoting $\Ber_i:=\BB(\om_{X/S}^i)$, we can write this isomorphism as
$$\Ber_3\simeq \Ber_1^5.$$

We will need an extension of isomorphism \eqref{super-Deligne-eq} to families of stable supercurves $\pi:X\to S$ and to the case when $L_1$ is replaced by a not necessarily
locally free sheaf. For a coherent sheaf $\FF$ over $X$, flat over $S$, we denote by $\BB(\FF)$ the Berezinian of the perfect complex $R\pi_*(\FF)$.

\begin{lemma}\label{super-Deligne-nonfree-lem}
Let $\pi:X\to S$ be a stable supercurve, $\FF$ a coherent sheaf on $X$ flat over $S$.
Assume that $\FF$ is locally free of rank $1|0$ over the smooth locus of $\pi$. Then for any line bundle $L$ of rank $1|0$ on $X$, one has
a canonical isomorphism
\begin{equation}\label{super-Deligne-sheaf-eq}
\BB(\FF\ot L)\simeq \BB(\FF)\ot \BB(L)\ot \BB(\OO_X)^{-1}.
\end{equation}
\end{lemma}

\begin{proof} We can use the similar recipe as above to construct this isomorphism locally over $S$: we simply replace $L_1$ by $\FF$, $L_2$ by $L$ and
take $D$ with support in the smooth locus of $\pi$, and a section $s\in H^0(X,L(D))$ with the vanishing divisor $E$, also contained in the smooth locus of $\pi$.
\end{proof}

\subsection{Behavior of the super Mumford isomorphism near the NS boundary divisor}\label{super-Mum-NS-sec}

Now we want to look at the behavior of the super Mumford isomorphism \eqref{super-Mumford-eq}
near the generic point of boundary divisors. In the case of a Ramond node it extends to an isomorphism, since $\om_{X/S}$ is
still a line bundle. Thus, it remains to study locally the boundary component where one NS node appears. So let
$\pi:X\to S$ be a family of stable supercurves as in Sec.\ \ref{super-KS-NS-sec}, so we have a smooth morphism $t:S\to \A^1$, a relative NS node $q:S_0\to X_0\sub X$ over 
the divisor $S_0=(t=0)\sub S$, and $\OO_X$ has the standard description along $q(S_0)$.

Let $S'=S\setminus S_0$, and let $X'=\pi^{-1}(S')$ be the corresponding family of smooth supercurves.
First, we need to explain how to extend the line bundles $\BB(\om_{X'/S'}^3)$ and $\BB(\OO_{X'})$ to $S$.
For the second one this is straightforward: the extension is given by $\BB(\OO_X)$.

For any integer $n$ let us set
$$\om_{X/S}^n:=j_*\om_{U/S}^n,$$
where $U\sub X$ is the smooth locus of $\pi$.
Recall that $\om_{X/S}^2$ is a line bundle on $X$ (see Theorem \ref{can-square-thm}). Thus, we have 
$$\om_{X/S}^3=j_*\om_{U/S}^3\simeq \om_{X/S}^2\ot j_*\om_{U/S}\simeq \om_{X/S}^2\ot \om_{X/S}.$$ 
This is a coherent sheaf, flat over $S$, so we can take $\BB(\om_{X/S}^3)$ as the desired extension of
$\BB(\om_{X'/S'}^3)$.

Note that by Grothendieck-Serre duality, we have an isomorphism of line bundes on $S$,
$$\BB(\om_{X/S})\simeq \BB(\OO_X).$$

Since $\om_{X/S}$ is a coherent sheaf flat over $S$, we deduce from Lemma \ref{super-Deligne-nonfree-lem} an isomorphism
$$\BB(\Pi\om_{X/S}^3)\simeq \BB(\Pi\om_{X/S})\ot \BB(\om_{X/S}^2)\ot \BB(\OO_X)^{-1}\simeq \BB(\om_{X/S}^2)\ot \BB(\OO_X)^{-2},$$
or equivalently,
\begin{equation}\label{BB-om3-X-Mum-eq}
\BB(\om_{X/S}^3)\simeq \BB(\om_{X/S}^2)^{-1}\ot \BB(\OO_X)^2.
\end{equation}

Thus, our problem reduces to studying the behavior of the isomorphism on $X'$,
$$\BB(\om_{X'/S'}^2)\simeq \BB(\Pi\om_{X'/S'})^2\ot \BB(\OO_{X'})^{-1}\simeq \BB(\OO_{X'})^{-3}$$
near the divisor $S_0\sub S$. To this end we look more carefully at the recipe for this isomorphism outlined in Sec.\ \ref{super-Deligne-sec}.

We can pick a sufficiently positive relative divisor $D$ supported in the smooth locus of $\pi$, and also an even section $s$ of $\Pi\om_{X/S}(D)$
with the zero locus $E$. Furthermore, we want to make a special choice of $s$ as described below.

% can assume that the intersection of $E$ with the smooth locus of $\pi$ is a relative divisor, 

\begin{lemma}\label{choice-of-section-lem}
There exists an exact sequence of $\OO_X$-modules,
$$0\to \om^{\reg}_{X/S}\to \om_{X/S}\to q_*\OO_{S_0}\to 0$$
such that in the formal neighborhood of $q(S_0)$, $\om^{\reg}_{X/S}$ is the $\OO_X$-submodule generated by the sections $s_1$ and $s_2$
given by \eqref{s1-s2-s0-generators}. 
Furthermore, we can choose an even section $s$ of $\Pi\om^{\reg}_{X/S}(D)$ in such a way that in the formal neighborhood of $q$ one has
$$s=v\cdot (s_1+u\cdot s_2),$$
for some invertible functions $u$ and $v$ that are congruent to $1$ modulo the ideal of $q(S_0)$.
By making a change of variables 
$$z_1\mapsto u^{-1}z_1,  \ \ z_2\mapsto uz_2, \ \ \th_1\mapsto \th_1, \ \ \th_2\mapsto u\th_2$$
we can achieve that $u=1$.
\end{lemma}

\begin{proof}
In the formal neighborhood of $q$, the quotient of $\om_{X/S}/(\OO_X s_1+\OO_X s_2)$ is generated by $s_0$ and is isomorphic to $q_*\OO_{S_0}$.
This implies the first assertion. Furthermore, using Lemma \ref{NS-om-lem}, we see that
$\om^{\reg}_{X/S}\ot q_*\OO_{S_0}$ is a free $q_*\OO_{S_0}$-module generated by the images of $s_1$ and $s_2$.
For sufficiently positive $D$, the map
$$H^0(X,\om^{\reg}_{X/S}(D))\to H^0(X,\om^{\reg}_{X/S}\ot q_*\OO_{S_0})$$
is surjective, so we can choose an odd section $s\in H^0(X,\om^{\reg}_{X/S}(D))$ such that 
$$s\equiv s_1+s_2\mod J_{q(S_0)}.$$
Thus, in the formal neighborhood of $q$, we can write $s=u_1s_1+u_2s_2$ where $u_i\equiv 1\mod J_{q(S_0)}$.

It is easy to check that the given change of variables transforms $s$ in the desired way.
\end{proof}

Now we define coherent sheaves $\FF$ and $\GG$ on $X$ by the exact sequences
$$0\to \OO_X\rTo{s}\Pi \om_{X/S}(D)\to \FF\to 0,$$
$$0\to \Pi \om_{X/S} \rTo{\cdot s} \om_{X/S}^2(D)\to \GG\to 0.$$
Let $\hat{\FF}$ and $\hat{\GG}$ denote
the completions of $\FF$ and $\GG$ at $q$.

\begin{lemma}\label{F-G-completions-lem}
(a) The space $\hat{\FF}$ is a free $\OO_S$-module with the basis $s_1, s_0$. 

\noindent
(b) We have an isomorphism $\hat{\GG}\simeq \OO_Z$, where $Z\sub X$ is a sub-superscheme given by
$$z_1=-t, \ \ z_2=t, \ \ \th_1+\th_2=0.$$
In particular, $\hat{\GG}$ is a free $\OO_S$-module with the basis $e, \th_1e$,
where $e$ is the local generator of $\om_{X/S}^2$ given by \eqref{e-sec-eq}.
The sheaf $\hat{\FF}$ is also scheme-theoretically supported on $Z$.

\noindent
(c) The Berezinian of the map of $\OO_S$-modules
$$s_1\cdot ?: \hat{\FF}\to \hat{\GG}$$
is equal to $ft$, where $f$ is an invertible function.
\end{lemma}

\begin{proof}
(a) Lemma \ref{NS-om-lem} immediately shows that the map 
$$\OO_S\to \Pi\om_{X/S}: f\mapsto f\cdot (s_1+s_2)$$
is injective.
Furthermore, the submodule $\OO_X(s_1+s_2)$ has the following topological basis over $\OO_S$:
\begin{align*}
&s_1+s_2, \ \ z_1^{n+1}(s_1+s_2)=z_1^n(z_1+t)s_1, \ \
z_2^n(z_2-t)s_2, \ \ \th_1(s_1+s_2)=\th_1s_1+ts_0, \\ 
&z_1^{n+1}\th_1(s_1+s_2)=z_1^n(z_1+t)\th_1s_1, \ \
\th_2(s_1+s_2)=\th_2s_2+ts_0, \ \ z_2^n(z_2-t)\th_2s_2,
\end{align*}
where $n\ge 0$.
Hence, the quotient by this submodule has the images of $s_0, s_1$ as a basis over $\OO_S$.

\noindent
(b) Near the node, $\om_{X/S}^2$ is a free $\OO_X$-module with one even generator $e$. Furthermore, it
is easy to check that with respect to the map $\om_{X/S}^{\ot 2}\to \om_{X/S}^2$ one has
$$s_1^2=z_1e, \ \ s_2^2=-z_2e, \ \ s_1s_2=te, \ \ s_0s_i=\th_ie,$$
for $i=1,2$.
Thus, the image of the multiplication by $s_1+s_2$ is the $\OO_X$-submodule generated
by 
$$s_1(s_1+s_2)=(z_1+t)e, \ \ s_2(s_1+s_2)=(t-z_2)e, \ \ s_0(s_1+s_2)=(\th_1+\th_2)e.$$
In other words, we have an identification
$$\hat{\GG}\simeq \OO_X/(z_1+t, z_2-t, \th_1+\th_2)=\OO_Z\simeq \OO_S[\th_1],$$
as claimed.

\noindent
(c) The operator of multiplication by $s_1$ acts on the bases of $\hat{\FF}$ and $\hat{\GG}$ as
$$s_1\mapsto s_1^2=-te, \ \ s_0\mapsto s_1s_0=\th_1e,$$
Hence, the Berezinian is equal to $-t$ (recall that $s_1$ is an even generator of $\hat{\FF}$, while $s_0$ is an odd generator).
\end{proof}

%Note that we can assume that over the reduced base, $s$ restricts to a regular section

%Lemma \ref{F-G-completions-lem} implies that $\FF$ and $\GG$ are flat over $S$.

Let us consider the following resolutions for the derived push-forwards under $\pi$:
\begin{align*}
&R\pi_*(\OO_X): & [\pi_*(\Pi \om_{X/S}(D)\to \pi_*\FF], \\
&R\pi_*(\Pi \om_{X/S}): & [\pi_*(\om_{X/S}^2(D))\to \pi_*\GG], \\
&R\pi_*(\Pi \om_{X/S}): & [\pi_*(\Pi \om_{X/S}(D))\to \pi_*(\Pi \om_{X/S}(D)|_D)], \\
&R\pi_*(\om_{X/S}^2): & [\pi_*(\om_{X/S}^2(D))\to \pi_*(\om_{X/S}^2(D)|_D)].
\end{align*}
These resolutions give a canonical isomorphism over $S$,
\begin{equation}\label{Ber-om2-NS-node-eq}
\BB(\om_{X/S}^2)\ot \BB(\OO_X)\simeq\Ber(R\pi_*(\Pi\om_{X/S}))^{\ot 2}\ot \Ber(\pi_*\GG)\ot \Ber(\pi_*\FF)^{-1}.
\end{equation}
Now the super Mumford isomorphism for the induced smooth family 
$X_{S'}\to S'$ is obtained by choosing an isomorphism of 
$$\FF|_{X_{S'}}\simeq \GG|_{X_{S'}}$$ 
(where both sheaves are supported on the zero divisor of $s$ in $X_{S'}$)
and using the Berezinian of the induced map on $R\pi_*$ to get an isomorphism of $\Ber(\pi_*\FF)|_{S'}$ with $\Ber(\pi_*\GG)|_{S'}$ (see \eqref{Ber-L-D-eq}).
For this we are going to construct a morphism
$$\a:\FF\to \GG,$$
which is an isomorphism away from the node, and use its restriction to $X_{S'}$.

Note that both $\FF$ and $\GG$ are supported on $Z\sqcup E$, 
where $E$ is a relative divisor supported in the smooth locus of $\pi$, so
we have decompositions into subsheaves supported on $Z$ and on $E$,
$$\FF\simeq \FF_Z\oplus \FF_{E}, \ \ \GG\simeq \GG_Z\oplus \GG_{E}.$$
We can choose separately morphisms $\FF_{E}\to \GG_{E}$ and $\FF_Z\to \GG_Z$. 

We have
$$\FF_{E}\simeq \pi_*(\Pi \om_{X/S}(D)|_{E}), \ \ \GG_{E}\simeq \pi_*(\om_{X/S}^2(D)|_{E}),$$
so both are isomorphic to $\OO_{E}$, and we choose any isomorphism between them as a morphism $\a_{E}:\FF_{E}\to \GG_{E}$.
On the other hand, we have
$$\FF_Z\simeq \hat{F}, \ \ \GG_Z\simeq \hat{\GG},$$
%$$\pi_*\FF=\pi_*\hat{\FF}\oplus \pi_*(\Pi \om_{X/S}(D)|_{E'}),$$
%$$\pi_*\GG=\pi_*\hat{\GG}\oplus \pi_*(\om_{X/S}^2(D)|_{E'}).$$
%Using  we get an isomorphism
%$$\Ber(\pi_*\GG)\ot \Ber(\pi_*\FF)^{-1}\simeq \Ber(\pi_*\hat{\GG})\ot \Ber(\pi_*\hat{\FF})^{-1}.$$
Hence, we can use the morphism of $\OO_X$-modules,
$$\mu_{s_1}:\hat{\FF}\to \hat{\GG},$$ 
given by the multiplication with $s_1$, as our morphism $\FF_Z\to \GG_Z$.
By Lemma \ref{F-G-completions-lem}, it restricts to an isomorphism over $t\neq 0$.
Thus, we get the desired morphism
$$\a=(\mu_{s'},\a_{E}):\FF\to \GG.$$

By definition, the super Mumford isomorphism for $X_{S'}\to S'$ is obtained by restricting the isomorphism of ilne bundles
\eqref{Ber-om2-NS-node-eq} to $S'$ and multiplying it with $\ber R\pi_*(\a|_{X_{S'}})^{-1}\in \BB(\FF)\ot \BB(\GG)^{-1}$.
Since $\ber R\pi_*(\a_{E'})$ is invertible on $S$, we see that
$$\ber R\pi_*(\a|_{X_{S'}})=u\cdot \ber \pi_*(\mu_{s_1}),$$
where $u$ is an invertible function on $S$.
%Hence, the induced morphism $\ber(\mu_{s_1})^{-1}$ composed with the isomorphism
%\eqref{Ber-om2-NS-node-eq} gives a morphism
%\begin{equation}\label{partial-Mumford-isom}
%\BB(\om_{X/S}^2)\to \BB(\OO_X)^{-1}\ot \Ber(R\pi_*(\Pi\om_{X/S}))^{\ot 2}\simeq \BB(\OO_X)^{-3}
%\end{equation}
%that restricts to the super Mumford isomorphism over $t\neq 0$.
Using Lemma \ref{F-G-completions-lem}(iii), we deduce that
$$\ber R\pi_*(\a|_{X_{S'}})^{-1}=u'\cdot t^{-1},$$ 
for some invertible function $u'$ on $S$, with respect to some
bases of $\BB(\FF)$ and $\BB(\GG)$ on $S$.

Thus, we obtained an isomorphism
\begin{equation}\label{partial-Mumford-isom}
\BB(\om_{X/S}^2)\rTo{\sim} \BB(\OO_X)^{-1}\ot \Ber(R\pi_*(\Pi\om_{X/S}))^{\ot 2}(\De),
\end{equation}
where $\De$ is the divisor given by $t=0$. Combining this with isomorphism
\eqref{BB-om3-X-Mum-eq}, we get the following local result.

\begin{prop}\label{super-Mum-prop}
In the above situation the super Mumford isomorphism
$$\BB(\om_{X'/S'}^{-2})\simeq \BB(\om_{X'/S'}^3)\simeq \BB(\OO_{X'})^5$$
extends to an isomorphism
$$\BB(\om_{X/S}^{-2})\simeq \BB(\om_{X/S}^3)\simeq \BB(\OO_X)^5(-\De).$$
\end{prop}

Combining this with Corollary \ref{KS-super-NS-node-cor} and with Proposition \ref{KS-super-R-prop} we obtain the following result.

\begin{theorem}\label{canonical-class-no-punct-thm} 
Let $\ov{\SS}_g$ denote the moduli stack of stable supercurves of genus $g$, and 
let $\De_{NS}\sub \ov{\SS}_g$ (resp., $\De_R\sub \ov{\SS}_g$) denote the boundary divisor supported on the locus where a supercurve
acquires an NS node (resp., a Ramond node), defined in Sec.\ \ref{boundary-equation-sec}.
Then one has an isomorphism of line bundles 
\begin{equation}\label{can-class-smooth-moduli-eq}
K_{\ov{\SS}_g}\simeq \Ber_1^5(-2\De_{NS}-\De_R).
\end{equation}
%has a pole of order $2$ near the generic point of the boundary component where the curve acquires an NS node
%and a pole of order $1$ near the generic point of the boundary component corresponding to a Ramond node.
\end{theorem}

\begin{proof}
Let $\pi:X\to\ov{\SS}_g$ denote the universal supercurve. 
First, we have an isomorphism
$$K_{\ov{\SS}_g}\simeq \Ber^{-1} R^1\pi_*\om_{X/\ov{\SS}_g}^{-2}(-\De_{NS}-\De_R)\simeq \BB(\om_{X/\ov{\SS}_g}^{-2})(-\De_{NS}-\De_R)$$
(see Corollary \ref{KS-super-NS-node-cor} and Proposition \ref{KS-super-R-prop}).
Next, by Proposition \ref{super-Mum-prop}, we have
$$\BB(\om_{X/\ov{\SS}_g}^{-2})\simeq \BB(\om_{X/\ov{\SS}_g}^3)\simeq \Ber_1^5(-\De_{NS})$$
(recall that the super-Mumford isomorphism \eqref{super-Mumford-eq} extends over the Ramond type boundary divisor $\De_R$).
Combining this with the previous isomorphism we get the result.
\end{proof}

%Note that the boundary divisor on the compactified moduli space $\ov{\SS}_g^+$ has components whose generic points correspond
%to nodal curves $C_1\cup C_2$ with the spin-structures $L_1\oplus L_2$, where $L_i$ is an odd spin-structure on $C_i$.
%The above theorem does not tell us about the behavior of the isomorphism \eqref{can-class-smooth-moduli-eq} near such components.

\subsection{Canonical line bundle on the moduli space of stable supercurves with punctures}\label{can-line-bun-punctures-sec}

Let $\pi:X\to S$ be a stable supercurve with NS punctures $P_1,\ldots,P_m$ and Ramond divisors $R_1,\ldots,R_n$.
We denote by 
%us consider $\MM=\MM_{g,n,0}$, the moduli stack of smooth supercurves with $n$ NS-punctures.
%Let $\pi:X\to\MM$ denote the universal curve, $P_1,\ldots,P_n:\MM\to X$ the punctures,
$D_i\sub X$ the divisors associated with $P_i$ as in Sec.\ \ref{NS-correspondence-sec}.

For every $i=1,\ldots,m$, let us define a line bundle of rank $0|1$ on $S$,
$$\Psi_i:=P_i^*\om_{X/S}.$$
%(D_i)\simeq \om_{D_i/S}|_{P_i}.$$
%The analogous line bundles on the classical moduli spaces lead to the so called $\psi$-classes.

\begin{lemma}\label{psi-lemma} 
One has natural isomorphisms of line bundles of rank $0|1$ on $S$,
$$\Psi_i\simeq \pi_*\OO_{D_i}/\OO_S,$$
$$\Psi_i^{-1}\simeq P_i^*\om_{X/S}(D_i)\simeq P_i^*\om_{D_i/S}\simeq \Ber(\pi_*\OO_{D_i}).$$
Hence, for any line bundle $L$ of rank $1|0$ on $X$, one has
$$\Ber(\pi_*L|_{D_i})\simeq \Psi_i^{-1},$$
while for a line bundle $M$ of rank $0|1$, one has
$$\Ber(\pi_*M|_{D_i})\simeq \Psi_i.$$
\end{lemma}

\begin{proof} 
Let us set $\Psi=\Psi_i$ and
$$\wt{\Psi}:=\pi_*\OO_D/\OO_S,$$
where $D=D_i$ and $P=P_i$.
Then $\wt{\Psi}$ is a line bundle of rank $0|1$, and we have
$$\Ber(\pi_*\OO_D)\simeq \wt{\Psi}^{-1}.$$
Hence, from \eqref{Ber-L-D-eq} we get the last two isomorphisms with $\Psi_i$ replaced by $\wt{\Psi}$.

Applying this to the line bundle $\om_{X/S}^{-2}$ of rank $1|0$, we get an isomorphism
$$\Ber(\pi_*\om_{X/S}^{-2}|_D)\simeq \wt{\Psi}^{-1}.$$
But from Lemma \ref{NS-div-corr-lem}, we have an isomorphism over $S$,
$$\pi_*(\om_{X/S}^{-2}|_D)\simeq \TT_{X/S}|_P.$$
Hence, passing to the Berezinians we get
$$\wt{\Psi}^{-1}\simeq \om_{X/S}^{-1}|_P,$$
i.e., $\wt{\Psi}\simeq \Psi$.

Next, let $I_D\sub I_P\sub \OO_X$ be the ideal sheaves of $D=D_i$ and $P=P_i$. Since the projection to $S$ induces an isomorphism $P\simeq S$,
we have a decomposition
$$\pi_*\OO_D=\OO_S\oplus \pi_*(I_P/I_D),$$
so 
$$\Psi\simeq \wt{\Psi}\simeq \pi_*(I_P/I_D).$$ 
Note also that $I_P^2\sub I_D$, so $I_P/I_D$ can be identified with the conormal sheaf to $P$ in $D$.
Now, since $\om_{P/S}=\OO_S$, the 
exact sequence
$$0\to \Psi \to \Om_{D/S}|_P\to \OO_P\to 0$$
gives an isomorphism
$$P^*\om_{D/S}\simeq \Ber(\Psi)=\Psi^{-1}.$$
\end{proof}

\begin{cor}\label{square-NS-triv-cor} 
The line bundle $P_i^*\om_{X/S}^2(D_i)$ on $S$ is canonically trivialized.
\end{cor}

\begin{example}
Suppose $X\to S$, $(P_i)$, is a family of supercurves with NS-punctures over an {\it even} base $S$. 
Then $\OO_X=\OO_C\oplus L$, where $(C,L)$ is the underlying family of curves with spin-structures, and
$\II_{P_i}=\II_{p_i}\oplus L$, where $p_i\sub C$ are marked points on $C$. We also have $\II_D=\II_{p_i}\ot_{\OO_C} \OO_X$.
Thus, 
$$\Psi_i^{-1}=P_i^*(\om_{X/S}(D_i))=P_i^*(L(p_i)\oplus \om_{C/S}(p_i))\simeq p_i^*L(p_i)\simeq p_i^*L^{-1},$$
where the last isomorphisms is induced by the trivialization of $L^2(p_i)|_{p_i}\simeq \om_{C/S}(p_i)|_{p_i}$.
\end{example}

\begin{remark}\label{Res-rem}
Note that for an NS-puncture $P_i$ and the corresponding divisor $D_i$, the decomposition
\eqref{NS-div-splitting-eq} can be rewritten as
$$\pi_*\OO_{D_i}=\OO_S\oplus \Psi_i,$$
where $\Psi_i$ is a square-zero ideal.
For any line bundle $L$ over $D_i$ we have an exact sequence
$$0\to P_{i_*}(\Psi_i\ot P_i^*L)\to L\to P_{i*}P_i^*L\to 0$$
%Since $\pi_*\om_{D_i/S}$ is a locally free $\pi_*\OO_{D_i}$-module, we also have the induced splitting
%$$\pi_*(\om_{X/S}(D_i)|_{D_i})=\pi_*\om_{D_i/S}\simeq \Psi_i^{-1}\oplus \OO_S.$$
%Indeed, to identify the summands, we use the isomorphism $P_i^*\om_{D_i/S}\simeq \Psi_i^{-1}$.
%In fact, the corresponding projection
In the case $L=\om_{X/S}(D_i)|_{D_i}=\om_{D_i/S}$, the induced exact sequence of push-forwards to $S$ has a splitting
$$\pi_*(\om_{X/S}(D_i)/\om_{X/S})\to \OO_S$$
given by the residue map for the $0|1$-dimensional superscheme $D_i/S$, so we have a decomposition
$$\pi_*(\om_{D_i/S})=\Psi_i^{-1}\oplus \OO_S.$$
\end{remark}

For every Ramond divisor $R_j\sub X$ let us consider the line bundle on the base,
$$\Phi_j:=\Ber(\pi_*\OO_{R_j})^{-1}\simeq \Ber(\pi_*\om_{R_j}).$$
Note that $\Phi_j=\Phi_{R_j}$ (see \eqref{PhiR-def}), so the line bundles $\Phi_j^2$ are canonically trivialized.

%Extra assumptions on the family near the nodes???
%Recall that for
Let us consider the line bundle of rank $1|0$ over $X$,
$$\LL(X,P_\bullet,R_\bullet):=\om_{X/S}^{-2}(-\sum_i D_i-2\sum_j R_j).$$
%the coherent sheaf $R^1\pi_*\LL(X,P_\bullet,D_\bullet)$ is locally free.

\begin{lemma}\label{Psi-Phi-contribution-lem}
One has a natural isomorphism
$$\BB(\LL(X,P_\bullet,R_\bullet))\simeq \BB(\om_{X/S}^3)\ot\bigotimes_i\Psi_i$$
%\ot \bigotimes_j\Phi_j^2$$
\end{lemma}

\begin{proof}
Let us set $D=\sum_i D_i$, $R=\sum_j R_j$.
First, the Grothendieck duality gives an isomorphism
$$\BB(\LL(X,P_\bullet,R_\bullet))\simeq \BB(\om_{X/S}^3(D+2R)).$$
Now using the exact sequence
\begin{equation}\label{om-3-sh-ex-seq}
0\to \om_{X/S}^3(2R)\to \om_{X/S}^3(2R)(D)\to \om_{X/S}^3(2R)(D)|_D\to 0,
\end{equation}
we get using \eqref{Ber-L-D-eq} and Lemma \ref{psi-lemma},
$$\BB(\om_{X/S}^3(D+2R)) \simeq \BB(\om_{X/S}^3(2R))\ot \Ber(\pi_*\om_{X/\MM}^3(D)|_D)\simeq
\BB(\om_{X/M}^3(2R))\ot\bigotimes_i\Psi_i.$$
Similarly, using the fitration of $\om_{X/S}^3(2R)/\om_{X/S}^3$ with subquotients
$\om_{X/S}^3(2R)|_R$ and $\om_{X/S}^3(R)|_R$, we get an isomorphism
$$\BB(\om_{X/S}^3(2R))\simeq \BB(\om_{X/S}^3)\ot\bigotimes_j\Phi_j^2.$$
It remains to we use the trivializations of $\Phi_j^2$.
\end{proof}

Now the same argument as in Theorem \ref{canonical-class-no-punct-thm} gives the formula for the canonical line bundle on the moduli stack of stable supercurves with punctures.

%\begin{theorem}\label{canonical-class-thm} 
%Let $\ov{\SS}=\ov{\SS}_{g,m,n}$ denote the moduli stack of stable supercurves with $m$ NS punctures and $n$ Ramond divisors.
%One has a natural isomorphism
%\begin{equation}\label{canonical-class-eq}
%K_{\ov{\SS}}\simeq \Ber_1^5\otimes \bigotimes_{i=1}^m\Psi_i(-2\De_{NS}-\De_R),
%\end{equation}
%where $\De_{NS}$ and $\De_R$ have the same meaning as in Theorem \ref{canonical-class-no-punct-thm}.
%\end{theorem}

\begin{proof}[Proof of Theorem B]
Let $\ov{\SS}=\ov{\SS}_{g,n_{NS},n_R}$. We combine the isomorphism
$$K_{\ov{\SS}}\simeq \Ber^{-1} R^1\pi_*\LL(X,P_\bullet,R_\bullet)(-\De_{NS}-\De_R)\simeq \BB(\LL(X,P_\bullet,R_\bullet))(-\De_{NS}-\De_R)$$
(see Sec.\ \ref{Kodaira-punct-sec}) with Lemma \ref{Psi-Phi-contribution-lem} and the isomophism
$$\BB(\om_{X/\ov{\SS}}^3)\simeq \Ber_1^5(-\De_{NS}).$$
This gives the required isomorphism
\begin{equation}\label{canonical-class-eq}
K_{\ov{\SS}}\simeq \Ber_1^5\otimes \bigotimes_{i=1}^m\Psi_i(-2\De_{NS}-\De_R),
\end{equation}
\end{proof}

\section{Splitting at the boundary divisor}
\label{Splitting-sec}

Now we are going to study the restriction of the isomorphism of Theorem B to the boundary divisor. Using our presentation of the line bundle corresponding
to the boundary divisor as a Berezinian (see Sec.\ \ref{boundary-equation-sec}) we find a natural identification of the normal line bundle to the boundary divisor.
Then we give a proof of Theorem C concerning the NS boundary component. We also give a conjectural statement for the Ramond boundary component.

\subsection{NS boundary components}\label{NS-splitting-sec}

Let $\iota:B\to \ov{\SS}$ be the standard gluing map covering an NS type boundary component, i.e., one of the maps \eqref{NS-sep-node-morphism} or 
\eqref{NS-nonsep-node-morphism}, restricted to the locus of smooth supercurves.
Let $X_B\to B$ denote the universal stable supercurve, which is obtained by identifying two NS punctures $P_1,P_2$ on a smooth supercurve
$\wt{X}\to B$ into a node $Q\sub X_B$. Let $D_1,D_2\sub \wt{X}$ be the corresponding divisors. Note that we have a finite morphism $\rho:\wt{X}\to X_B$, and an exact sequence
on $X_B$,
$$0\to \OO_{X_B}\to \rho_*\OO_{\wt{X}}\to \OO_Q\to 0.$$
In particular, the Berezinian line bundle $\BB(\OO_{X_B})$ for the family $X_B\to B$ is naturally identified with $\BB(\OO_{\wt{X}})$ defined for the family $\wt{X}\to B$.

Note that in the case of a separating node, where $B=\SS_1\times \SS_2$, the line bundle $\BB(\OO_{\wt{X}})$ is the exterior product of two similar line bundles on the factors
$\SS_1$ and $\SS_2$.
%Let $\Ber_1^B$ be either $\Ber_1$ line bundle on it
%in the case of a non-separating node, or the product of two $\Ber_1$ line bundles from two factors in
%the case of a separating node. We have an exact sequence on the universal curve over $B$,

%Ignoring the punctures, 
We can rewrite \eqref{canonical-class-eq} near $B$ as an isomorphism
$$K_{\ov{\SS}}(2\De_{NS})\simeq \BB(\OO_X)^5.$$
Thus, pulling it back to $B$ leads to an isomorphism
$$K_B\ot N_B\simeq \BB(\OO_{X_B})^5,$$
where $N_B$ is the normal bundle defined as the pull-back of $\OO(\De_{NS})$ to $B$.
Note that the universal supercurve over $B$ is equipped with two NS punctures $P_1,P_2$ and the isomorphism \eqref{canonical-class-eq}
in smaller genus gives
$$K_B\simeq \BB(\OO_{\wt{X}})^5\ot\Psi_1\ot\Psi_2.$$
Comparing with the previous isomorphism we get an isomorphism
%\begin{equation}\label{NS-normal-bundle-eq}
$$N_B\simeq \Psi_1^{-1}\ot \Psi_2^{-1}.$$
%\end{equation}
Below we will define such a canonical isomorphism independently and then will check its compatibility with two above isomorphisms.
%(for this one has to study the structure of $\de:\Om_{X/S}\to \om_{X/S}$ for the universal curve over $B$).

%Here is one way to define an isomorphism \eqref{NS-normal-bundle-eq}.

First, recall that we have a line bundle $\om_{X/S}^2$ defined on the universal curve of $\ov{\SS}$ by extending from the smooth locus (see Theorem \ref{can-square-thm}).

\begin{lemma}\label{splitting-ex-seq-lem}
(i) For any integer $m$ let us set 
$$\om_{X_B/B}^{2m}:=\om_{X/\ov{\SS}}^{2m}|_B.$$
Then one has an exact sequence on $X_B$,
\begin{equation}\label{om-2m-res-ex-seq}
0\to \om_{X_B/B}^{2m}\to \rho_*\om_{\wt{X}/B}^{2m}(mD_1+mD_2)\to \OO_Q\to 0
\end{equation}
where $Q\sub X_B$ is the relative node obtained by gluing $P_1\simeq B\simeq P_2$.
Here we use the canonical trivializations of the restrictions of $\om_{\wt{X}/B}^2(D_1+D_2)$ to $P_1$ and $P_2$
(see Corollary \ref{square-NS-triv-cor}).
Hence, we have a natural isomorphism 
\begin{equation}\label{om-2m-res-NS-eq}
\BB(\om_{X/\ov{\SS}}^{2m})|_B\simeq \BB(\om_{\wt{X}/B}^{2m}(mD_1+mD_2)).
\end{equation}

\noindent
(ii) Let us set $\om_{X_B/B}^{2m+1}:=\om_{X_B/B}^{2m}\ot \om_{X_B/B}$. Then one has an exact sequence on $X_B$,
\begin{equation}\label{om-res-simple-seq}
0\to \rho_*\om^{2m+1}_{\wt{X}/B}(mD_1+mD_2)\to \om_{X_B/B}^{2m+1}\to Q_*\OO_S\to 0.
\end{equation}
Hence, we have a natural isomorphism
\begin{equation}\label{om-2m+1-res-NS-eq}
\BB(\om_{X/\ov{\SS}}^{2m+1})|_B
%\simeq \BB(\om_{\wt{X}/B}^{2m+1}((m+1)D_1+(m+1)D_2))\ot \Psi_1^{-1}\Psi_2^{-1}
\simeq\BB(\om_{\wt{X}/B}^{2m+1}(mD_1+mD_2)).
\end{equation}

In addition, we have a natural exact sequence
$$0\to \om_{X_B/B}^{2m+1}\to \rho_*\om_{\wt{X}/B}^{2m+1}((m+1)D_1+(m+1)D_2)\to \CC\to 0,$$
where $\CC$ is a sheaf supported on the node fitting into an exact sequence
\begin{equation}\label{Q-C-Psi-seq}
0\to Q_*\OO_S\to \CC\to Q_*(\Psi_1^{-1}\oplus \Psi_2^{-1})\to 0.
\end{equation}

\noindent
(iii) One has an exact sequence on $X_B$,
$$0\to \KK\to \Om_{X_B/B}\to \rho_*\Om_{\wt{X}/B}\to 0,$$
where the sheaf $\KK$ is supported on the node, and has a filtration with the subfactors
$$Q_*(\Psi_1^2\Psi_2^2), \ Q_*(\Psi_1^2\Psi_2), \ Q_*(\Psi_1\Psi_2^2), \ Q_*(\Psi_1\Psi_2).$$
In particular, $\Ber \pi_*\KK$ is canonically trivial,
so
$$\Ber R\pi_*(\Om_{X_B/B})\simeq\Ber R\pi_*\Om_{\wt{X}/B}.$$

\noindent
(iv) One has an exact sequence over the smooth locus of $B$,
$$0\to j_*\Om_{U/\ov{\SS}}|_{X_B}\to \Om_{\wt{X}/B}(D_1+D_2)\to \wt{\CC}\to 0,$$
where $U\sub X$ is the smooth locus of $X\to \ov{\SS}$,
and the sheaf $\wt{\CC}$ is supported on the node and has a filtration with the subfactors
$$\OO_Q,  \ \OO_Q, \  Q_*(\Psi_1^{-1}\oplus \Psi_2^{-1}).$$
\end{lemma}

\begin{proof}
(i) We just have to identify the pull-back of the line bundle $\om_{X_B/B}^2$ to $\wt{X}$, $\rho^*\om_{X_B/B}^2$ with  
$\om_{\wt{X}/B}^2(D_1+D_2)$. Note that we have a natural identification of these line bundles over the smooth locus of $\wt{X}$.
Thus, we need to check that it extends to an isomorphism over the node. For this, it is enough to study these line bundles in an \'etale
neighborhood of the node. Thus, we can place ourselves in the framework of Sec.\ \ref{super-KS-NS-sec} and use a generator
$e$ of $\om_{X/S}^2$ (see \eqref{e-sec-eq}). Since $\rho^*e$ is a generator of $\om_{\wt{X}/B}^2(D_1+D_2)$, the assertion follows.

\noindent
(ii) First, let us consider the case $m=0$.
Using Lemma \ref{NS-om-lem} we easily see that there is an injective map $\om_{X_B/B}\to \rho_*\om_{\wt{X}/B}(D_1+D_2)$, 
and that $\om_{X_B/B}$ contains $\rho_*\om_{\wt{X}/B}$. Thus, we have 
$$\CC:=\coker(\om_{X_B/B}\to \rho_*\om_{\wt{X}/B}(D_1+D_2))=\coker(\om_{X_B/B}/\om_{\wt{X}/B}\to \rho_*(\om_{\wt{X}/B}(D_1+D_2)/\om_{\wt{X}/B})).$$

Note that since $\Psi_i^{-1}\simeq P_i^*\om_{\wt{X}/B}(D_i)$, 
by Remark \ref{Res-rem}, we have an exact sequence
$$0\to \OO_{P_1}\oplus \OO_{P_2}\to \om_{\wt{X}/B}(D_1+D_2)/\om_{\wt{X}/B}\to P_{1*}\Psi_1^{-1}\oplus P_{2*}\Psi_2^{-1}\to 0$$ 
In local coordinates, the projection to $P_{i*}\Psi_i^{-1}$ sends $\frac{1}{z_i}[dz_i|d\th_i]$ to a generator, and sends $\frac{\th_i}{z_i}[dz_i|d\th_i]$ to zero.
In particular, this projection vanishes on the image of $\om_{X_B/B}/\om_{\wt{X}/B}$.
Thus, we have an embedding
%exact sequence
%$$0\to \CC'\to \CC\to Q_*(\Psi_1^{-1}\oplus \Psi_2^{-1})\to 0$$
%and 
%$$\CC'=\coker(
$$\om_{X_B/B}/\om_{\wt{X}/B}\hra \OO_Q\oplus \OO_Q.$$
We claim that $\om_{X_B/B}/\om_{\wt{X}/B}$ coincides with the kernel of
the addition map $\OO_Q\oplus \OO_Q\to \OO_Q$.
Indeed, this is a local statement. In local coordinates the generators of two summands $\OO_Q$ are $\frac{\th_i}{z_i}[dz_i|d\th_i]$,
and the generator $s_0$ of $\om_{X_B/B}/\om_{\wt{X}/B}$ is mapped to their difference.

From this we also see that $\CC$ fits into the exact sequence \eqref{Q-C-Psi-seq}.

%Note the above argument also shows the image of $\om_{X_B/B}/\om_{\wt{X}/B}$ in $\OO_Q\oplus \OO_Q$ is isomorphic to $\OO_Q$,
%which leads to the exact sequence \eqref{om-res-simple-seq}.

%we have a natural morphism
%$$\CC\to Q_*(\Psi_1^{-1}\oplus \Psi_2^{-1}).$$
%A local computation shows that it is surjective: $\frac{1}{z_i}[dz_i|d\th_i]$ maps to a generator of $\Psi_i^{-1}$. 
%We claim that kernel is naturally identified with $\OO_Q$.

To derive the case of arbitrary $m$ from that of $m=0$, we tensor the sequence \eqref{Q-C-Psi-seq} with the line bundle
$\om^{2m}_{X_B/B}$. By the triviality of $Q^*(\om_{X_B/B}^2$, we get the sequence of the required form.
%In addition, we have a natural morphism
%$$\pi_*(\om_{\wt{X}/B}(D_1+D_2)/\om_{\wt{X}/B})\rTo{\Res_{P_1}+\Res_{P_2}} \OO_S$$
%(see Remark \ref{Res-rem}).
%Combining these we get a canonical map
%$$\CC\to \OO_S\oplus \Psi_1^{-1}\oplus \Psi_2^{-1}.$$
%The fact that it is an isomorphism is checked by a local computation.

\noindent
(iii) First, we note that $\KK$ is supported on the node, so it is enough to prove the assertion after replacing $X_B$ with the formal neighborhood of the node.
Then $\wt{X}$ becomes the union of two branches $X_1\sqcup X_2$, so that $P_1\in X_1$, $P_2\in X_2$.
Now we have inclusions of ideals
$$\rho_*(I_{P_1}\oplus 0), \rho_*(0\oplus I_{P_2})\sub \OO_{X_B}\sub \rho_*\OO_{\wt{X}}.$$
Furthermore, the product of these ideals is zero. This implies that for $a_1\in I_{P_1}$, $a_2\in I_{P_2}$, one has
$$(0,a_2)\cdot d(a_1,0)=\pm (a_1,0)\cdot d(0,a_2).$$
Hence, we have well defined map
$$\kappa:I_{P_1}/I_{P_1}^2\ot I_{P_2}/I_{P_2}^2\to \KK\sub\Om_{X_B/B}: a_1\ot a_2)\mapsto (a_1,0)\cdot d(0,a_2).$$

Note that $I_{P_1}/I_{P_1}^2$ fits into an exact sequence
$$0\to I_{D_1}/I_{D_1}I_{P_1}\to I_{P_1}/I_{P_1}^2\to I_{P_1}/I_{D_1}\to 0$$
with $I_{D_1}/I_{D_1}I_{P_1}\simeq \Psi_1^2$ and $I_{P_1}/I_{D_1}\simeq \Psi_1$.
Using Lemma \ref{differential-generators-lem}(ii), one checks that 
$\kappa$ is an isomorphism, and the assertion follows.

\noindent
(iv) This follows from parts (i), (ii) and from the exact sequence \eqref{j-Om-ex-seq} (and a similar exact sequence for $\Om_{\wt{X}/B}$ which
holds over the smooth locus).
\end{proof}

%Passing to the Berezinians of pushforwards in the exact sequence of the above Lemma, we get
%$$\BB(\om_{X/S}^{-2})|_B\simeq \BB(\om_{\wt{X}/S}^{-2}(-D_1-D_2))\simeq \BB(\om_{\wt{X}/S}^{-2})\ot \Psi_1\ot\Psi_2.$$

%Now by Proposition \ref{super-Mum-prop}, we have an isomorphism of line bundles near $B$,
%\begin{equation}\label{Mumford-near-B-eq}
%\BB(\om_{X/S}^{-2})\simeq \BB(\OO_X)^5(-B).
%\end{equation}
%On the other hand, Mumford isomorphism for $\wt{X}$ gives
%$$\BB(\om_{\wt{X}/S}^{-2})\simeq \BB(\OO_{\wt{X}})^5.$$
%Since $\BB(\OO_X)|_B\simeq \BB(\OO_{\wt{X}})$, comparing it with the restriction of
%\eqref{Mumford-near-B-eq}, we get an isomorphism \eqref{NS-normal-bundle-eq}.
%$$\Psi_1\ot\Psi_2\simeq \OO(-B)|_B.$$

We can use the definition of the NS boundary line bundle \eqref{boundary-line-bundle} to compute the normal bundle
over the smooth part of $B$.
From the above Lemma we get
$$\Ber R\pi_*(j_*\Om_{U/S})|_B\simeq \Ber R\pi_*(\Om_{\wt{X}/B}(D_1+D_2))\ot \Psi_1^{-1}\Psi_2^{-1},$$
$$\Ber R\pi_*(\Om_{X/S}|_B)\simeq \Ber R\pi_* \Om_{\wt{X}/B}.$$
Finally the exact sequence
$$0\to \om_{\wt{X}/B}^2\to \Om_{\wt{X}/B}\to \om_{\wt{X}/B}\to 0$$
near $D_1,D_2$, together with Lemma \ref{psi-lemma}, show that
$$\Ber R\pi_* \Om_{\wt{X}/B}(D_1+D_2)\simeq \Ber R\pi_*\Om_{\wt{X}/B}.$$
Hence, we deduce an isomorphism
\begin{equation}\label{NS-normal-bundle-formula}
\OO(\De)|_B\simeq \Ber R\pi_*(j_*\Om_{U/S})|_B\ot \Ber^{-1} R\pi_*(\Om_{X/S}|_B)\simeq \Psi_1^{-1}\Psi_2^{-1}.
\end{equation}

Now we can state a compatibility result between the super Mumford isomorphisms over $\ov{\SS}$ and over the NS boundary component $B$.
Recall that we have the super Mumford isomorphism
$$\mu_{\ov{\SS}}:\BB(\om_{X/\ov{\SS}}^{-2})\rTo \BB(\OO_X)^5(-\De)$$
near $B$ (see Prop.\ \ref{super-Mum-prop}). We also have the super Mumford isomorphism for the family $\wt{X}/B$, over the smooth locus:
$$\mu_B:\BB(\om_{\wt{X}/B}^{-2}(-D_1-D_2))\rTo{\sim} \BB(\OO_{\wt{X}})^5\ot \Psi_1\ot\Psi_2$$
(see Lemma \ref{Psi-Phi-contribution-lem}).

\begin{theorem}\label{splitting-thm} 
The following diagram of isomorphisms of line bundles on $B$ is commutative up to a sign
\begin{diagram}
\BB(\om_{X/\ov{\SS}}^{-2})|_B&\rTo{\mu_{\ov{\SS}}|_B}&\BB(\OO_X)^5(-\De)|_B\\
\dTo{\eqref{om-2m-res-NS-eq}}&&\dTo{\eqref{NS-normal-bundle-formula}}\\
\BB(\om_{\wt{X}/B}^{-2}(-D_1-D_2))&\rTo{\mu_B}&\BB(\OO_{\wt{X}})^5\ot \Psi_1\ot\Psi_2
\end{diagram}
%Using \eqref{om-2m-res-NS-eq} and \eqref{NS-normal-bundle-formula},
%the restriction of $\mu_{\ov{\SS}}$ to the open locus in $B$, corresponding to smooth supercurves,
%$$\BB(\om_{\wt{X}/B}^{-2}(-D_1-D_2))\simeq\BB(\om_{X/\ov{\SS}}^{-2})\rTo{\mu_{\ov{\SS}}|_B} \BB(\OO_X)^5(-B)|_B\simeq \BB(\OO_{\wt{X}})^5\ot \Psi_1\ot\Psi_2,$$
%coincides with $\mu_B$.
\end{theorem}

Recall that to get $\mu_{\ov{\SS}}$ we used three isomorphisms: the Grothendieck-Serre duality isomoprhism
$$SD_{X/\ov{\SS}}:\BB(\om_{X/\ov{\SS}}^{-2})\rTo{\sim}\BB(\om_{X/\ov{\SS}}^3),$$
the isomorphism 
$$\BB(\om_{X/\ov{\SS}}^3)\simeq \BB(\om_{X/\ov{\SS}}^2)^{-1}\ot \BB(\OO_X)^2,$$
and the isomorphism
$$\BB(\om_{X/\ov{\SS}}^2)\simeq \BB(\om_{X/\ov{\SS}})^{-2}\ot \BB(\OO_X)^{-1}(\De)$$
(see Sec.\ \ref{super-Mum-NS-sec}). Similarly, $\mu_B$ is a composition of similar three isomorphisms. So we can reduce the proof of Theorem \ref{splitting-thm}
to separate compatibilities involving each of these three isomorphisms. We deal with this compatibilities in the next three lemmas.

\begin{lemma}\label{split-SD-lem} 
For any $m\in\Z$, the following diagram of isomorphisms between line bundles on $B$ is commutative:
\begin{diagram}
\BB(\om_{X/\ov{\SS}}^{-2m})|_B&\rTo{SD_{X/\ov{\SS}}|_B}&\BB(\om_{X/\ov{\SS}}^{2m+1})|_B\\
\dTo{\eqref{om-2m-res-NS-eq}}&&\dTo{\eqref{om-2m+1-res-NS-eq}}\\
\BB(\om_{\wt{X}/B}^{-2m}(-mD_1-mD_2))&\rTo{SD_{\wt{X}/B}}&\BB(\om_{\wt{X}/B}^{2m+1}(mD_1+mD_2))
\end{diagram} 
where the horizontal arrows are given by Grothendieck-Serre duality. 
%and the vertical arrows are induced by isomorphisms \eqref{om-2m-res-NS-eq} and \eqref{om-2m+1-res-NS-eq}.
\end{lemma}

\begin{proof}
To begin with we can replace the upper horizontal arrow with the one induced by the Grothendieck-Serre duality for $X_B$ over $B$,
$$\BB(\om^{-2m}_{X_B/B})\rTo{SD_{X_B/B}}\BB(\om^{2m+1}_{X_B/B}).$$
Let us set for brevity 
$$L:=\om^{-2m}_{X_B/B}, \ \ \wt{L}:=\om_{\wt{X}/B}^{-2m}(-mD_1-mD_2).$$

Recall that for a sufficiently nice morphism $f:X\to Y$ and a perfect complex $F$ on $X$, the Grothendieck Serre duality gives an isomorphism
$$SD_f:Rf_*(F)\rTo{\sim} Rf_*(R\und{\Hom}(F,\om_f[\dim f]))^\vee.$$

Then we claim that there is an isomorphism of the exact triangles induced by the exact sequences \eqref{om-2m-res-ex-seq} and \eqref{om-res-simple-seq}
\begin{equation}\label{morphism-SD-triangles-eq}
\begin{diagram}
R\pi_*(L)&\rTo{}&R\pi_*(\rho_*\wt{L})&\rTo{}&\OO_S&\rTo{}\ldots\\
\dTo{SD_{X_B/B}}&&\dTo{SD_{\wt{X}/B}}&&\dTo{\id}\\
R\pi_*(L^{-1}\ot\om_{X_B/B}[1])^\vee&\rTo{}&R\pi_*(\rho_*(\wt{L}^{-1}\ot\om_{\wt{X}/B}[1]))^\vee&\rTo{}&\OO_S&\rTo{}\ldots
\end{diagram}
\end{equation}
Clearly this would imply the claimed commutativity. 

Now we claim that commutativity of both squares in \eqref{morphism-SD-triangles-eq} follows from the
general property of Grothendieck-Serre duality for a pair of morphisms $X\rTo{g}Y\rTo{f}X$,
\begin{diagram}
Rf_*(Rg_*(F))&\rTo{\sim}&R(f\circ g)_*(F)\\
\dTo{SD_f}&&\dTo{SD_g}\\
Rf_*(R\und{\Hom}(Rg_*(F),\om_f[\dim f]))^\vee&\rTo{\sim}&R(f\circ g)_*(R\und{\Hom}(F,\om_{f\circ g}[\dim (f\circ g)]))^\vee
\end{diagram}
where the lower horizontal arrow is induced by the Grothendieck-Serre duality isomorphism
$$R\und{\Hom}(Rg_*(F),\om_f[\dim f])\simeq Rg_*R\und{\Hom}(F,\om_{f\circ g}[\dim (f\circ g)]).$$
%te that the trace morphism for $\wt{X}/B$ is equal to the composition
%$$R(\pi\circ\rho)_*\om_{\wt{X}/B}\simeq R\pi_*\rho_*\om_{\wt{X}/B}\to R\pi_*\om_{X_B/B}\to \OO_B[-1],$$
%where the last map is the trace morphism for $X_B/B$.
%Furthermore, we have a natural isomorphism
%$$R\und{\Hom}(\rho_*\wt{L},\om_{X_B/B})\simeq \rho_*(\wt{L}^{-1}\ot\om_{\wt{X}/B}),$$
%so that the middle vertical arrow in \eqref{morphism-SD-triangles-eq} gets identified with the map
%$$R\pi_*(\rho_*\wt{L})\to R\pi_*R\und{\Hom}(\rho_*\wt{L},\om_{X_B/B})$$
%This implies the commutativity of the left square in \eqref{morphism-SD-triangles-eq}.

Indeed, applying this to $f=\pi$, $g=\rho$ and $F=\wt{L}$ allows us to identify the middle vertical arrow in \eqref{morphism-SD-triangles-eq} with
the map
$$R\pi_*(\rho_*\wt{L})\to R\pi_*R\und{\Hom}(\rho_*\wt{L},\om_{X_B/B}[1])^\vee$$
given by the Grothendieck-Serre duality for $X_B/B$. Taking this into account, commutativity of the left square in \eqref{morphism-SD-triangles-eq}
becomes a basic functoriality of $SD_{\pi}$. On the other hand, commutativity of the right square in \eqref{morphism-SD-triangles-eq} similarly follows 
from the functoriality of $SD_{\pi\circ \rho}$ applied to the natural morphism $\wt{L}\to Q_*\OO_S$ and 
from the above compatibility for $f=\pi\circ \rho$, $g=Q$ and $F=\OO_S$.
\end{proof}
%\BB(\om_{X/\ov{\SS}}^3)|_B&\rTo{\sim}& \BB(\om_{X/\ov{\SS}}^2)^{-1}|_B\ot \BB(\OO_X)^2|_B\\
%\BB(\om_{\wt{X}/B}^3(D_1+D_2))&\rTo{\sim}& \BB(\om_{\wt{X}/B}^2(D_1+D_2))\ot \BB(\OO_{\wt{X}})^2

\begin{lemma}\label{resolution-Ber-compatibility-lem}
Let $\rho:Y\to X$ be a morphism of families of stable supercurves over $B$, which is a fiberwise resolution of the node $Q:B\to X$ (so it is an isomorphism
away from $X\setminus Q(B)$). For any coherent sheaf $\FF$ on $Y$ which is locally free of rank $1|0$ over the smooth locus, and any
line bundle $L$ on $X$ with a trivialization of $Q^*L$, we have a commutative diagram of isomorphisms of line bundles on $B$,
\begin{diagram}
\BB(\rho_*\FF\ot L)&\rTo{\eqref{super-Deligne-sheaf-eq}}&\BB(\rho_*\FF)\ot \BB(L)\ot \BB(\OO_X)^{-1}\\
\dTo{}&&\dTo{}\\
\BB(\FF\ot \rho^*L)&\rTo{\eqref{super-Deligne-sheaf-eq}}&\BB(\FF)\ot \BB(\rho^*L)\ot \BB(\OO_Y)^{-1}
\end{diagram}
where in the right vertical arrow we use isomorphisms $\BB(\OO_X)\simeq \BB(\OO_Y)$ and $\BB(L)\simeq \BB(\rho^*L)$
coming from the exact sequences
$$0\to \OO_X\to \rho_*\OO_Y\to \OO_Q\to 0$$
$$0\to L\to L\ot \rho_*\OO_Y\to L\ot \OO_Q\to 0$$
and the trivialization of $Q^*L$.
\end{lemma}

\begin{proof}
The question is local in the base, so we can assume that we can 
choose a relative divisor $D\sub X$ supported in the smooth locus of $\pi:X\to B$, and a global section $s$ of $L(D)$ such that the 
$E=\div(s)$ is also supported in the smooth locus. Now we compute both horizontal arrows using the section $s$ on $X$ and its pull-back $\rho^*s$ on $Y$.
Thus, for the top horizontal arrow we use resolutions
\begin{align*}
&R\pi_*(\OO_X): & [\pi_*L(D)\to \pi_*(L(D)|_E)], \\
&R\pi_*(\rho_*\FF): & [\pi_*(\rho_*\FF\ot L(D))\to \pi_*(\rho_*\FF\ot L(D)|_E)], \\
&R\pi_*(L): & [\pi_*(L(D))\to \pi_*(L(D)|_D)], \\
&R\pi_*(\rho_*\FF\ot L): & [\pi_*(\rho_*\FF\ot L(D))\to \pi_*(\rho_*\FF \ot L(D)|_D)],
\end{align*}
while for the bottom horizontal arrow we use similar resolutions on $Y$ that use the section $\rho^*s$ of $\rho^*L(D')$, e.g.,
$$R\pi_*(\OO_Y): [\pi'_*(\rho^*L(D'))\to \pi'_*\rho^*L(D)|_{E'}],$$
where $\pi':Y\to B$ is the projection, $D'=\rho^{-1}(D)$, $E'=\rho^{-1}(E)$.
Note that $D'\to D$ and $E'\to E$ are isomorphisms, and the assertion follows from the commutativity of the squares
\begin{diagram}
\pi_*(L(D)|_E)&\rTo{\sim}&\pi_*(\rho_*\FF\ot L(D)|_E) \\
\dTo{}&&\dTo{}\\
\pi'_*(\rho^*L(D')|_{E'})&\rTo{\sim}&\pi'_*(\FF\ot\rho^*L(D')|_{E'})
\end{diagram}
\begin{diagram}
\pi_*(L(D)|_D)&\rTo{\sim}&\pi_*(\rho_*\FF\ot L(D)|_D) \\
\dTo{}&&\dTo{}\\
\pi'_*(\rho^*L(D')|_{D'})&\rTo{\sim}&\pi'_*(\FF\ot\rho^*L(D')|_{D'})
\end{diagram}
\end{proof}

\begin{lemma}\label{main-splitting-lem} The following diagram is commutative up to a sign
\begin{equation}\label{main-splitting-diagram}
\begin{diagram}
\BB(\om_{X/\ov{\SS}}^2)|_B&\rTo{\sim}& \BB(\om_{X/\ov{\SS}})^{-2}|_B\ot \BB(\OO_X)^{-1}(\De)|_B\\
\dTo{\sim}&&\dTo{\sim}\\
\BB(\om_{\wt{X}/B}^2(D_1+D_2))&\rTo{\sim}&\BB(\om_{\wt{X}/B})^{-2}\ot \BB(\OO_{\wt{X}})^{-1}\ot \Psi_1^{-1}\Psi_2^{-1}
\end{diagram} 
\end{equation}
\end{lemma}
%\end{document}

\begin{proof} Note that it is enough to prove the commutativity of this diagram working in
an \'etale neighborhood $S$ of a stable supercurve $X_s$ with one NS node. We also choose standard presentation in an \'etale neighborhood of the node
on $X_s$, and use constructions of Sec.\ \ref{super-Mum-NS-sec}.

\medskip

\noindent
{\bf Step 1}. Under isomorphism \eqref{NS-normal-bundle-formula}, the trivialization of $N_B$ induced by the equation $t=0$ of the boundary divisor, corresponds to the trivialization of 
$\Psi_1^{-1}\ot \Psi_2^{-1}$ given by $\th_1^{-1}\ot \th_2^{-1}|_B$ (where we identify $\Psi_i$ with $\pi_*\OO_{D_i}/\OO_S$).
%This is a local statement on the base, so we can consider a formal neighborhood $S$ of a point on $B$. 
Below we will use the notation from Sec.\ \ref{boundary-str-sec}.

First, let us consider the canonical section $c$ of $\Ber\pi_*[\Om_{X/S}\to j_*\Om_{U/S}]$ (where $j_*\Om_{U/S}$ is placed in degree $0$).
The proof of Proposition \ref{Ber-boundary-prop}(i) shows that the complex $\pi_*[\Om_{X/S}\rTo{\iota} j_*\Om_{U/S}]$ can be represented by a morphism of trivial bundles with
bases $(b_i)$, $(c_i)$ such that $b_1,c_1,b_4,c_4$ are even; $b_2,c_2,b_3,c_3$ are odd; and the differential $\iota$ is given by 
$$\iota(b_1)=t^2c_1,\ \ \iota(b_i)=tc_i, \text{ for } i=2,3,4; \ \ \iota(b_i)=c_i, \ \text{ for } i>4.$$
In addition, over $B=S_0$, the elements $b_1,\ldots,b_4$ (resp., $c_1,\ldots,c_4$) induce the standard bases of the sheaves
$\pi_*\KK$ and $\pi_*\CC_0$ from Lemma \ref{differential-generators-lem}.
This shows that the restriction of $c/t$ to $B=S_0$ corresponds to the trivialization of
$$\Ber\pi_*[\Om_{X/S}\to j_*\Om_{U/S}]|_B\simeq \Ber\pi_*\CC_0\ot (\Ber\pi_*\KK)^{-1}$$
induced by the standard bases of $\pi_*\KK$ and $\pi_*\CC_0$.

Note that in the commutative square 
\begin{diagram}
\Om_{X_B/B}&\rTo{}&j_*\Om_{U/S}|_{X_B}\\
\dTo{}&&\dTo{}\\
\rho_*\Om_{\wt{X}/B}&\rTo{}&\rho_*\Om_{\wt{X}/B}(D_1+D_2)
\end{diagram}
the bottom horizontal arrow and the right vertical arrows are injective.
Hence, the subsheaves $\KK\sub \Om_{X_B/B}$ in Lemma \ref{differential-generators-lem} and in Lemma \ref{splitting-ex-seq-lem}(iii) are the same.
This also means that we have an exact sequence
$$0\to \rho_*\Om_{\wt{X}/B}\to j_*\Om_{U/S}|_{X_B}\to \CC_0\to 0$$
and hence, an exact sequence
\begin{equation}\label{CC0-wtCC-seq}
0\to \CC_0\to \Om_{\wt{X}/B}(D_1+D_2)|_{D_1+D_2}\to \wt{\CC}\to 0
\end{equation}
where $\wt{\CC}$ is the cokernel of the right vertical arrow in the above diagram (see Lemma \ref{splitting-ex-seq-lem}(iv)).

Recall that the trivialization of $\pi_*\KK$ used in Lemma \ref{splitting-ex-seq-lem}(iii) comes from the filtration of $\KK$ with the subfactors
$Q_*(\Psi_1^2\Psi_2^2)$, $Q_*(\Psi_1^2\Psi_2)$, $Q_*(\Psi_1\Psi_2^2)$ and $Q_*(\Psi_1\Psi_2)$. It is easy to check that this filtration coincides with the
filtration coming from the basis \eqref{KK-basis-eq} of $\KK$.
Let us consider the following basis of $\pi_*\Om_{\wt{X}/B}(D_1+D_2)|_{D_1+D_2}$:
$$e_i:=\frac{dz_i-\th_id\th_i}{z_i}, \ \ \frac{\th_idz_i}{z_i}, \ \ \frac{d\th_i}{z_i}, \ \ \frac{\th_id\th_i}{z_i}, \ \ i=1,2.$$
Note that it is compatible (up to a sign) with the canonical trivialization of $\Ber(\pi_*\Om_{\wt{X}/B}(D_1+D_2)|_{D_1+D_2}$.
The filtration of $\wt{\CC}$ considered in Lemma \ref{splitting-ex-seq-lem}(iv) is compatible with this basis:
the subsheaf $\OO_Q$ corresponds to the image of $e_1$ (or equivalently of $e_2$); the next subfactor $\OO_Q$ is given by the image
of $\th_1d\th_1/z_1$ (or of $\th_2d\th_2/z_2$); and the quotient $Q_*(\Psi_1^{-1}\oplus\Psi_2^{-1}$ is given by the image of $(d\th_1/z_1,d\th_2/z_2)$.
Hence, the trivialization of $\Ber\pi_*\wt{\CC}$ coming from this basis is compatible with the isomorphism
$$\Ber\pi_*\wt{\CC}\simeq (\Psi_1^{-1}\Psi_2^{-1})^{-1}$$
and the trivializations of $\Psi_i^{-1}$ given by the image of $d\th_i/z_i$, i.e., by the generator $\th_i$ under the identification $\Psi_i^{-1}\simeq \pi_*\OO_{D_i}/\OO_S$.

On the other hand, the images in $\CC_0$ of the basis vectors 
$$e_1-e_2, \frac{\th_1d\th_1}{z_1}-\frac{\th_2d\th_2}{z_2}, \frac{\th_1dz_1}{z_1}, \frac{\th_2dz_2}{z_2}$$
are given by $e-f, f, \th_1e, \th_2e$ in terms of the basis \eqref{CC-basis-eq} of $\CC_0$.
Hence, the trivialization of $\Ber\pi_*\CC_0$ coming from the latter basis coincides up to a sign with its trivialization induced by 
the isomorphism
$$\Ber\pi_*\CC_0\simeq \Ber^{-1}\pi_*\wt{\CC}\simeq \Psi_1^{-1}\Psi_2^{-1},$$
coming from the exact sequence \eqref{CC0-wtCC-seq}, and by the trivialization of $\Psi_1^{-1}\Psi_2^{-1}$ given by $\th_1\ot\th_2$.

Combining all the above steps we see that the trivialization $c/t|_B$ of $\Ber\pi_*[\Om_{X/S}\to j_*\Om_{U/S}]|_B$ coincides up to a sign with the trivialization
coming from the isomorphism
$$\Ber\pi_*[\Om_{X/S}\to j_*\Om_{U/S}]|_B\simeq \Psi_1^{-1}\Psi_2^{-1}$$
and the trivialization $\th_1\ot\th_2$ of $\Psi_1^{-1}\Psi_2^{-1}$.

\medskip

\noindent
{\bf Step 2}. Let  
$$\phi:\BB(\om_{X/\ov{\SS}}^2)\to\BB(\om_{X/\ov{\SS}})^{-2}\ot \BB(\OO_X)^{-1}(\De)$$
be isomorphism \eqref{partial-Mumford-isom}. Using the equation $(t=0)$ of $\De$, we get
an isomorphism 
\begin{equation}\label{t-phi-eq}
t\phi:\BB(\om_{X/\ov{\SS}}^2)\to \BB(\om_{X/\ov{\SS}})^{-2}\ot \BB(\OO_X)^{-1}.
\end{equation}
Our goal in this step is to compute it.

We start by recalling the exact sequence (which depends on a choice of coordinates)
\begin{equation}\label{om-reg-om-OQ-ex-seq}
0\to \om^{\reg}_{X/S}\to \om_{X/S}\to \OO_Q\to 0
\end{equation}
(see Lemma \ref{choice-of-section-lem}). On the other hand,
by Lemma \ref{splitting-ex-seq-lem}(ii), we have a natural exact sequence
\begin{equation}\label{om-wtX-om-XB-OQ-ex-seq}
0\to \om_{\wt{X}/B}\to \om_{X_B/B}\to \OO_Q\to 0.
\end{equation}
It is easy to see that the restriction of the embedding $\om^{\reg}_{X/S}\to \om_{X/S}$ to $B$
gives a morphism
$$\om^{\reg}_{X/S}|_{X_B}\to \om_{X_B/B}$$ 
with the image $\om_{\wt{X}/B}$, so that the restriction of \eqref{om-reg-om-OQ-ex-seq}
is compatible with \eqref{om-wtX-om-XB-OQ-ex-seq}.

Recall that we choose a sufficiently positive effective divisor $D\sub X$ with support in the smooth locus and a global section $s$ of $\om^{\reg}_{X/S}(D)$
as in Lemma \ref{choice-of-section-lem}.
Let us set
$$\FF^{\reg}:=\om^{\reg}_{X/S}(D)/(s).$$
We can modify the derivation of the isomorphism \eqref{Ber-om2-NS-node-eq} by replacing the resolutions for $R\pi_*(\OO_X)$ with
$$[\pi_*(\Pi \om^{\reg}_{X/S}(D))\to \pi_*\FF^{\reg}].$$
This leads to an isomorphism
\begin{equation}\label{Ber-om2-NS-node-modified-eq} 
\BB(\om_{X/S}^2)\simeq \BB(\OO_X)^{-1}\ot \BB(\Pi\om^{\reg}_{X/S})\ot \BB(\Pi\om_{X/S})\ot \BB(\GG)\ot \BB(\FF^{\reg})^{-1}.
\end{equation}
In addition, exact sequence \eqref{om-reg-om-OQ-ex-seq} gives an isomorphism 
\begin{equation}\label{Ber-om-reg-om-isom}
\BB(\om^{\reg}_{X/S})\simeq \BB(\om_{X/S}).
\end{equation}
On the other hand, we have an isomorphism
$$\a^{\reg}=(\frac{\mu_{s_1}}{t},\a_E):\FF^{\reg}\rTo{\sim}\GG$$
induced by some isomorphism $\a_E$ of the parts supported on $E$ and by $\mu_{s_1}/t$ on the parts supported on $Z$. 
Namely, near $Z$, $\FF^{\reg}$ has an $\OO_S$-basis $s_1,\th_1s_1$ and $\mu_{s_1}/t$ sends this basis to the basis $e,\th_1e$ of $\GG$.

Hence, we get the induced isomorphism
\begin{equation}\label{Ber-F-reg-G-isom}
\ber(\pi_*\a^{\reg}):\BB(\FF^{\reg})\rTo{\sim} \BB(\GG)
\end{equation}
Now we obtain that $t\phi$ is the isomorphism induced by \eqref{Ber-om2-NS-node-modified-eq}, together with \eqref{Ber-om-reg-om-isom} and \eqref{Ber-F-reg-G-isom}.

\medskip

\noindent
{\bf Step 3}. Let us set 
$$\HH:=\om^2_{X_B/B}/\om^2_{\wt{X}/B}.$$
Let also $\wt{s}$ be the global section of $\om_{\wt{X}}(D)$ induced by $s|_{X_B}$, with the zero divisor $\wt{E}\sub \wt{X}$ (which is disjoint from the preimage of the node). 
Let us set $Z_B=Z\cap X_B\sub X_B$. Note that $Z_B$ is supported on the node and the completion of its ideal is generated by $z_1$, $z_2$ and $\th_1+\th_2$,
so $Z_B$ is smooth of dimension $0|1$ over $B$.
We will construct exact sequences
\begin{equation}\label{Freg-om-divres-eq}
%0\to \Pi\OO_Q\to \FF^{\reg}|_B\to \OO_Q\oplus \Pi\om_{\wt{X}/B}(D)|_{\wt{E}}\to 0,
0\to \Pi\OO_{Z_B}\to \FF^{\reg}|_{X_B}\to \Pi\om_{\wt{X}/B}(D)|_{\wt{E}}\to 0,
\end{equation}
\begin{equation}\label{G-om2-divres-eq}
0\to \om^2_{\wt{X}/B}(D)|_{\wt{E}}\to \GG|_{X_B}\to \HH/\Pi\OO_Q\to 0,
\end{equation}
\begin{equation}\label{HH-ex-seq}
0\to \HH\to \om^2_{\wt{X}/B}(D_1)|_{D_1}\oplus \om^2_{\wt{X}/B}(D_1)|_{D_1}\to \OO_Q\to 0.
\end{equation}
such that the following diagram is commutative up to a sign:
\begin{diagram}
\BB(\FF^{\reg}|_{X_B})&\rTo{\ber(\pi_*\a^{\reg})}&\BB(\GG|_{X_B})\\
\dTo{}&&\dTo{}\\
\BB(\Pi\om_{\wt{X}/B}(D)|_{\wt{E}})&\rTo{(\ber(\pi_*(\a_E)|_B,\tau)}&\BB(\om^2_{\wt{X}/B}(D)|_{\wt{E}})\ot \BB(\HH/\Pi\OO_Q)
\end{diagram}
where the vertical arrows come from the exact sequences 
\eqref{Freg-om-divres-eq} and \eqref{G-om2-divres-eq}, and $\tau$ is the trivialization of $\BB(\HH/\Pi\OO_Q)$ coming from \eqref{HH-ex-seq}
and the standard bases $(\bb_i,\th_i\bb_i)$ of $\om^2_{\wt{X}/B}(D_i)|_{D_i}$.

First, we have a decomposition of $\FF^{\reg}$ into the parts supported on $E$ and on $Z$. The latter part is isomorphic to the completion
$\hat{\FF}^{\reg}$ of $\FF^{\reg}$. It is easy to see that the section $s_1$ induces an isomorphism
$$\Pi\OO_Z\rTo{\sim} \hat{\FF}^{\reg},$$ 
%restricting \eqref{} to $X_B$, we obtain an exact sequence
%$$0\to \OO_Q\to \om^{\reg}_{X/S}|_B\to \om_{\wt{X}/B}\to 0,$$
%where the embedding of $\OO_Q$ corresponds to the section $ts_0|_B=\th_1s_2=\th_2s_1$. Taking the quotient by $(s)$ we get
so we obtain a split exact sequence \eqref{Freg-om-divres-eq}.

Next, we have an injective morphism of exact sequences
\begin{diagram}
0&\rTo{}&\Pi\om_{\wt{X}/B}&\rTo{}&\Pi\om_{X_B/B}&\rTo{}&\Pi\OO_Q&\rTo{}& 0\\
&&\dTo{\wt{s}}&&\dTo{s|_{X_B}}&&\dTo{s_0(s_1+s_2)}\\
0&\rTo{}&\om^2_{\wt{X}/B}(D)&\rTo{}&\om^2_{X_B/B}(D)&\rTo{}&\HH&\rTo{}&0
\end{diagram}
Passing to the quotients we get \eqref{G-om2-divres-eq}.

Exact sequence \eqref{HH-ex-seq} is immediately obtained from Lemma \ref{splitting-ex-seq-lem}. This sequence shows that
$\HH$ has an $\OO_B$-basis $(\ov{e}, \th_1\ov{e}, \th_2\ov{e})$, where $\ov{e}$ is the image of the local generator $e$ of $\om^2_{X_B/B}$.
The embedding $\Pi\OO_Q\to \HH$ is given by $s_0(s_1+s_2)=(\th_1-\th_2)\ov{e}$ 
It follows that $\ov{e}, \th_1\ov{e}$ is a basis of $\HH/\Pi\OO_Q$. Now we observe that
all three isomorphisms
$$\Pi\OO_{Z_B}\rTo{\sim} \FF^{\reg}_Z|_{X_B}\rTo{\mu_{s_1}/t}\GG|_{X_B}\rTo{\sim} \HH/\Pi\OO_Q$$
send standard $\OO_S$-bases to each other.
This gives the desired commutative diagram.

\medskip

\noindent
{\bf Step 4}. From sequences \eqref{Freg-om-divres-eq}, \eqref{G-om2-divres-eq} and \eqref{HH-ex-seq},
together with the identifications 
$$\BB(\om^2_{\wt{X}/B}(D_i)|_{D_i})\simeq\Psi_i^{-1}$$
(see Lemma \ref{psi-lemma}),
we get isomorphisms
\begin{equation}\label{Ber-F-reg-G-isom-coord-free}
\BB(\GG|_{X_B})\ot\BB(\FF^{\reg}|_{X_B})^{-1}\simeq \BB(\HH)\simeq \Psi_1^{-1}\Psi_2^{-1}.
\end{equation}
We claim that the top horizontal arrow in diagram \eqref{main-splitting-diagram} composed with the isomorphism $\OO(\De)|_B\simeq \Psi_1^{-1}\Psi_2^{-1}$
gets identified with the composition of \eqref{Ber-om2-NS-node-modified-eq} with \eqref{Ber-F-reg-G-isom-coord-free} (we also take into account \eqref{Ber-om-reg-om-isom}).

Indeed, if we use the standard trivialization of $\OO_S(\De)|_B$, then the top horizontal arrow in \eqref{main-splitting-diagram} is precisely the isomorphism $(t\phi)|_B$.
By Step 2, it is obtained as the composition of isomorphisms \eqref{Ber-om2-NS-node-modified-eq} and \eqref{Ber-F-reg-G-isom}, restricted to $B$.
By Step 3, we can replace \eqref{Ber-F-reg-G-isom} by \eqref{Ber-F-reg-G-isom-coord-free} together with the standard trivialization of $\Psi_1^{-1}\Psi_2^{-1}$.
By Step 1, the latter trivialization corresponds to the standard trivialization of $\OO_S(\De)|_B\simeq \Psi_1^{-1}\Psi_2^{-1}$, so our claim follows.

\medskip

\noindent
{\bf Step 5}. We see that the restriction of isomorphism \eqref{Ber-om2-NS-node-modified-eq} to $B$ coincides with the isomorphism
\begin{equation}\label{B-mixed-resolutions-Mumford-isom}
\BB(\om_{X_B/B}^2)\ot \BB(\Pi\om_{\wt{X}/B})^{-1}\ot\BB(\Pi\om_{X_B/B})^{-1}\ot\BB(\OO_{X_B})\to\BB(\GG|_{X_B})\ot \BB(\FF^{\reg}|_{X_B})^{-1}
\end{equation}
obtained by using resolutions 
\begin{eqnarray}
R\pi_*(\OO_{X_B}): & [\pi_*\Pi\om^{\reg}_{X/S}(D)|_{X_B}\to \pi_*\FF^{\reg}|_{X_B}], \label{O-XB-sq-Mum-isom-1} \\
R\pi_*(\Pi \om_{X_B/B}): & [\pi_*\om_{X_B/B}^2(D))\to \pi_*\GG|_{X_B}], \\ \label{O-XB-sq-Mum-isom-2}
R\pi_*(\Pi \om_{\wt{X}/B}): & [\pi_*\Pi \om_{\wt{X}/B}(D)\to \pi_*\Pi \om_{\wt{X}/B}(D)|_D], \label{O-XB-sq-Mum-isom-3}\\
R\pi_*(\om_{X_B/B}^2): & [\pi_*\om_{X_B/B}^2(D)\to \pi_*\om_{X_B/B}^2(D)|_D], \label{O-XB-sq-Mum-isom-4}
\end{eqnarray}
and using the isomorphism $\Ber(\pi_*\Pi \om_{\wt{X}/B}(D)|_D)\simeq \Ber(\pi_*\om_{X_B/B}^2(D)|_D)$.

To prove the commutativity of \eqref{main-splitting-diagram} we need to compare \eqref{B-mixed-resolutions-Mumford-isom} with the similar isomorphism where 
$\OO_{X_B}$ (resp., $\om_{X_B/B}$ and $\om_{X_B/B}^2$) is replaced by $\OO_{\wt{X}}$ 
(resp., $\om_{\wt{X}}$ and $\om_{\wt{X}/B}^2$).
For this we will modify resolutions \eqref{O-XB-sq-Mum-isom-1}, \eqref{O-XB-sq-Mum-isom-2} and \eqref{O-XB-sq-Mum-isom-4} in a controlled manner.

%Let us denote by $\wt{s}$ the global section of $\Pi\om_{\wt{X}/B}(D)$ induced by $s$, and let $\wt{E}$ denote its zero divisor.
We start by noticing that the exact sequence
$$0\to \OO_{X_B}\to \OO_{\wt{X}}\to \OO_Q\to 0$$
is represented by the exact triangle 
$$[\Pi\OO_Q\to \Pi\OO_{Z_B}]\to [\Pi \om^{\reg}_{X/S}(D)|_{X_B}\to \FF^{\reg}|_{X_B}]\to [\Pi \om_{\wt{X}/B}(D)\to \Pi \om_{\wt{X}/B}(D)|_{\wt{E}}]\to \ldots,$$
which measures the difference between resolution \eqref{O-XB-sq-Mum-isom-1} and the corresponding resolution of $R\pi_*(\OO_{\wt{X}})$.

Next, the exact sequence
$$0\to \Pi\om_{\wt{X}/B}\to \Pi\om_{X_B/B} \to \Pi\OO_Q\to 0$$
is realized by the exact triangle of resolutions
$$[\om^2_{\wt{X}/B}(D)\to\om^2_{\wt{X}/B}(D)|_{\wt{E}}]\to [\om^2_{X_B/B}(D)\to \GG|_{X_B}]\to [\HH\to \HH/\Pi\OO_Q].$$
%where 
This gives the modification of resolution \eqref{O-XB-sq-Mum-isom-2}.

Finally, the exact sequence
$$0\to \om^2_{\wt{X}/B}\to \om^2_{X_B/B}\to \HH\to 0$$
is realized by the exact triangle of resolutions
$$[\om^2_{\wt{X}/B}(D)\to \om^2_{\wt{X}/B}(D)|_D]\to [\om^2_{X_B/B}(D)\to \om^2_{X_B/B}(D)|_D]\to [\HH\to 0],$$
which gives the modification of resolution \eqref{O-XB-sq-Mum-isom-4}.

The modified resolutions,
\begin{align*}
&R\pi_*(\OO_{\wt{X}_B}): & [\pi_*\Pi \om_{\wt{X}/B}(D)\to \pi_*\Pi \om_{\wt{X}/B}(D)|_{\wt{E}}], \\
&R\pi_*(\Pi \om_{\wt{X}/B}): & [\pi_*\om_{\wt{X}/B}^2(D)\to \pi_*\om_{\wt{X}/B}^2(D)|_{\wt{E}}], \\
&R\pi_*(\Pi \om_{\wt{X}/B}): & [\pi_*(\Pi \om_{\wt{X}/B}(D))\to \pi_*\Pi \om_{\wt{X}/B}(D)|_D], \\
&R\pi_*(\om_{\wt{X}/B}^2): & [\pi_*\om_{\wt{X}/B}^2(D))\to \pi_*\om_{\wt{X}/B}^2(D)|_D],
\end{align*}
give an isomorphism
\begin{equation}\label{super-Mum-wtX-B-aux-eq}
\BB(\om_{\wt{X}/B}^2)\ot\BB(\Pi\om_{\wt{X}/B})^{-1}\ot\BB(\Pi\om_{\wt{X}/B})^{-1}\ot \BB(\OO_{\wt{X}})\rTo{\sim}
\BB(\om^2_{\wt{X}/B}(D)|_{\wt{E}})\ot\BB(\Pi\om_{\wt{X}/B}(D)|_{\wt{E}})^{-1}
\end{equation}
Note that the composition of this isomorphism with the natural trivialization of
$\BB(\om^2_{\wt{X}/B}(D)|_{\wt{E}})\ot\BB(\Pi\om_{\wt{X}/B}(D)|_{\wt{E}})^{-1}$ is precisely the super Mumford isomorphism for $\wt{X}/B$,
\begin{equation}\label{super-Mum-wt-X-B-eq}
\BB(\om_{\wt{X}/B}^2)\ot\BB(\Pi\om_{\wt{X}/B})^{-1}\ot\BB(\Pi\om_{\wt{X}/B})^{-1}\ot \BB(\OO_{\wt{X}})\rTo{\sim}\OO_B.
\end{equation}

%Thus, isomorphism \eqref{B-mixed-resolutions-Mumford-isom} fits into 
From the exact triangles connecting the resolutions above we get the following commutative square of isomorphisms:
\begin{diagram}
\BB(\om_{X_B/B}^2)\ot \BB(\Pi\om_{\wt{X}/B})^{-1}\ot\BB(\Pi\om_{X_B/B})^{-1}\ot\BB(\OO_{X_B})&\rTo{\eqref{B-mixed-resolutions-Mumford-isom}}
&\BB(\GG|_{X_B})\ot \BB(\FF^{\reg}|_{X_B})^{-1}\\
\dTo{}&&\dTo{\eqref{Ber-F-reg-G-isom-coord-free}}\\
\BB(\HH)\ot\BB(\om_{\wt{X}/B}^2)\ot\BB(\Pi\om_{\wt{X}/B})^{-1}\ot\BB(\Pi\om_{\wt{X}/B})^{-1}\ot \BB(\OO_{\wt{X}})&\rTo{\eqref{super-Mum-wtX-B-aux-eq}}
&\BB(\HH)\ot\BB(\om^2_{\wt{X}/B}(D)|_{\wt{E}})\ot\BB(\Pi\om_{\wt{X}/B}(D)|_{\wt{E}})^{-1}
\end{diagram}
where the left vertical arrow comes from the exact sequences mentioned above. 
%and the right vertical arrow is given by \eqref{Ber-F-reg-G-isom-coord-free}.
By Step 4, the composition of the top horizontal and right vertical arrows corresponds to the same composition in diagram
\eqref{main-splitting-diagram}.

Finally, it is easy to check that the following diagram of isomorphisms
\begin{diagram}
\BB(\om^2_{X_B/B})&\rTo{}&\BB(\om^2_{\wt{X}/B})\ot \BB(\HH)\\
\dTo{}&&\dTo{}\\
\BB(\om^2_{\wt{X}/B}(D_1+D_2))&\rTo{}&\BB(\om^2_{\wt{X}/B})\ot\Psi_1^{-1}\ot\Psi_2^{-1}
\end{diagram} 
is commutative. This allows us to replace the composition of the left vertical and bottom horizontal arrows in the previous diagram with the same composition
in diagram \eqref{main-splitting-diagram}, thus, finishing the proof.
\end{proof}

\begin{proof}[Proof of Theorem \ref{splitting-thm}]
Let us consider the diagram
\begin{diagram}
\BB(\om_X^{-2})|_B&\rTo{}&\BB(\om_X^3)|_B&\rTo{}&\BB(\om_X^2)^{-1}\BB(\OO_X)^2|_B&\rTo{}&\BB(\om_X)^2\BB(\OO_X)^3(-B)|_B\\
\dTo{}&&\dTo{}&&\dTo{}&&\dTo{}\\
\BB(\om_{\wt{X}}^{-2}(-D_1-D_2))&\rTo{}&\BB(\om_{\wt{X}}^3(D_1+D_2))&\rTo{}&\BB(\om_{\wt{X}}^2(D_1+D_2))^{-1}\BB(\OO_{\wt{X}})^2&\rTo{}&
\BB(\om_{\wt{X}})^2\BB(\OO_{\wt{X}})^3\Psi_1\Psi_2
\end{diagram}
We claim that each square in this diagram is commutative.
Indeed, the left square is commutative by Lemma \ref{split-SD-lem}, while the right square is commutative by Lemma \ref{main-splitting-lem}. 
The commutativity of the middle square follows from the commutative diagram
\begin{diagram}
\BB(\om_{X_B/B}\ot \om_{X_B/B}^2)&\rTo{\eqref{super-Deligne-sheaf-eq}}&\BB(\om_{X_B/B})\BB(\om_{X_B/B}^2)^{-1}\BB(\OO_{X_B})\\
\dTo{}&&\dTo{}\\
\BB(\rho_*(\om_{\wt{X}/B})\ot \om_{X_B/B}^2)&\rTo{\eqref{super-Deligne-sheaf-eq}}&\BB(\rho_*(\om_{\wt{X}/B}))\BB(\om_{X_B/B}^2)^{-1}\BB(\OO_{X_B})\\
\dTo{}&&\dTo{}\\
\BB(\om_{\wt{X}/B}^3(D_1+D_2))&\rTo{\eqref{super-Deligne-sheaf-eq}}&\BB(\om_{\wt{X}/B})\BB(\om_{\wt{X}/B}^2(D_1+D_2))^{-1}\BB(\OO_{\wt{X}})
\end{diagram}
Here the lower square is commutative by Lemma \ref{resolution-Ber-compatibility-lem} applied to the morphism $\rho:\wt{X}\to X_B$, the line bundle
$\om_{X_B/B}^2$ and the coherent sheaf $\om_{\wt{X}/B}$.
The commutativity of the upper square can be checked using the compatibility of the isomorphism \eqref{super-Deligne-sheaf-eq}
with the resolution $[\om_{X_B/B}\to \OO_Q]$ for $\rho_*\om_{\wt{X}/B}$.
\end{proof}

\subsection{Splitting of the Kodaira-Spencer map at an NS boundary component}\label{KS-splitting-sec}

\begin{lemma}\label{KS-splitting-lem}
There is an isomorphism of exact sequences on $B$,
\begin{equation}\label{KS-splitting-diagram}
\begin{diagram}
0&\rTo{}&\OO_B&\rTo{}&T_{\ov{\SS},B}|_B&\rTo{}& T_B&\rTo{}&0\\
&&\dTo{\id}&&\dTo{KS_{\ov{\SS}}|_B}&&\dTo{KS_B}\\
0&\rTo{}&\OO_B&\rTo{}&R^1\pi_*(\om_{X_B/B}^{-2})&\rTo{}& R^1\pi_*(\om_{\wt{X}/B}^{-2}(-D_1-D_2))&\rTo{}&0
\end{diagram}
\end{equation}
where the lower exact sequence is induced by \eqref{om-2m-res-ex-seq}.
\end{lemma}

\begin{proof}
We have a natural morphism of exact sequences
\begin{equation}\label{KS-B-diagram}
\begin{diagram}
0&\rTo{}&\AA_{X/\ov{\SS}}|_B&\rTo{}&\AA_{X,X_B}|_B&\rTo{}&\pi^{-1}\TT_{\ov{\SS},B}|_B&\rTo{}&0\\
&&\dTo{}&&\dTo{}&&\dTo{}\\
0&\rTo{}&\AA_{X_B/B}&\rTo{}&\AA_{X_B}&\rTo{}&\pi^{-1}\TT_B&\rTo{}&0\\
\end{diagram}
\end{equation}
Furthermore, the lower sequence can be identified with
$$0\to \rho_*\om^{-2}_{\wt{X}}(-D_1-D_2)\to \rho_*\AA_{\wt{X},P_1,P_2}\to \pi^{-1}\TT_B\to 0$$
Thus, the left vertical arrow in \eqref{KS-B-diagram} can be identified with the natural map
$$\om_{X/\ov{\SS}}^{-2}\to \rho_*\om^{-2}_{\wt{X}}(-D_1-D_2)$$
which has cokernel $\OO_Q$ (by \eqref{om-2m-res-ex-seq}).
On the other hand, the morphism $\TT_{\ov{\SS},B}|_B\to \TT_B$ has the kernel $\OO_B$.

Applying the functor $R\pi_*$ to \eqref{KS-B-diagram} we immediately derive commutativity of the right square in \eqref{KS-splitting-diagram}.

Next, we claim that the coboundary morphism
$$\ker(\pi^{-1}\TT_{\ov{\SS},B}|_B\to \pi^{-1}\TT_B)\to \coker(\om_{X/\ov{\SS}}^{-2}\to \rho_*\om^{-2}_{\wt{X}}(-D_1-D_2))$$
associated with \eqref{KS-B-diagram} gets identified with the natural map $\pi^{-1}\OO_B\to \OO_Q$.
This is a local statement, so we can use coordinates as in Sec.\ \ref{super-KS-NS-sec}. Note that the section $1$ of $\pi^{-1}\OO_B$ is represented by the vector field
$t\partial_t$ in $\pi^{-1}\TT_{\ov{\SS},B}|_B$. It lifts to a vector field $v$ in $\AA_{X,X_B}$ given by 
$$v(z_i)=z_i, \ \ v(\th_i)=\frac{1}{2}\th_i, \ \ v(t)=t$$
(see Lemma \ref{KS-super-NS-lem}(i)).
The restriction of $v$ to $X_B$ corresponds to the vector field 
$$z_1\pa_{z_1}+\frac{1}{2}\th_1\pa_{\th_1}+z_2\pa_{z_2}+\frac{1}{2}\th_2\pa_{\th_2}$$
which lives in $\AA_{\wt{X},P_1,P_2}$.
The isomorphism 
$$\om_{\wt{X}/B}^{-2}(-D_1-D_2)\rTo{\sim} \AA_{\wt{X},P_1,P_2}$$
is given by
$$z_i[dz_i|d\th]^{-2}\mapsto z_i\pa_{z_i}+\frac{1}{2}(\pa_{\th_i}+\th_i\pa_{z_i})(z_i)\cdot (\pa_{\th_i}+\th_i\pa_{z_i})=z_i\pa_{z_i}+\frac{1}{2}\pa_{\th_i}$$
(see the proof of Lemma \ref{NS-div-corr-lem}),
which immediately implies our claim.

From this we deduce that the following square commutes
\begin{diagram}
\pi^{-1}\OO_B&\rTo{}& \pi^{-1}\TT_{\ov{\SS},B}|_B\\
\dTo{}&&\dTo{}\\
\OO_Q&\rTo{}& \om_{X/\ov{\SS}}^{-2}[1]
\end{diagram}
Applying $R\pi_*$ we get the commutativity of the left square in \eqref{KS-splitting-diagram}.
\end{proof}

\begin{cor}
We have a commutative diagram
\begin{diagram}
K_{\ov{\SS}}(\De)|_B&\rTo{\ber(KS_{\ov{\SS}})^{-1}|_B}&\BB(\om_{X/\ov{\SS}}^{-2})|_B\\
\dTo{}&&\dTo{}\\
K_B&\rTo{\ber(KS_B)^{-1}}&\BB(\om_{\wt{X}/B}^{-2}(-D_1-D_2))
\end{diagram}
where the horizontal arrows are induced by the Kodaira-Spencer isomorphisms
(see Proposition \ref{stable-can-KS-formula-prop}),
and the right vertical arrow is given by \eqref{om-2m-res-NS-eq}.
\end{cor}

Combining the above Corollary with Theorem \ref{splitting-thm}, we get the statement of Theorem C.

\subsection{Ramond boundary components}\label{R-splitting-sec}

Now let $\iota:B\to\ov{\SS}$ be a gluing map for the Ramond boundary component, i.e., one of the maps \eqref{NS-nonsep-node-morphism},
restricted to the locus of smooth supercurves. Recall that we have
a smooth map of relative dimension $0|1$, 
$$p:B\to \ov{B}$$
to a moduli space $\ov{B}$ of smooth supercurves of smaller genus, where the universal smooth supercurve over $\ov{B}$ is equipped with a pair of Ramond punctures $R_1,R_2$.
Furthermore, $p$ has a structure of the principal bundle over the group scheme $\Aut(\PP_{R_1})$ which is an extension of $\Z/2$ by
$\Phi_1^{-1}$. Thus, we have a natural isomorphism
$$K_B\simeq p^*(K_{\ov{B}}\ot \Phi_1).$$
On the other hand, we have an exact sequence on the universal stable supercurve $X_B$ over $B$,
$$0\to \OO_{X_B}\to \OO_{\wt{X}}\to \OO_{R_1}\to 0,$$
where $\wt{X}$ is induced by the universal curve $\ov{X}$ over $\ov{B}$. Taking push-forwards to $B$ and considering the Berezinians,
we get
$$\iota^*\Ber_1\simeq p^*(\Ber^{\ov{B}}_1\ot \Phi_1).$$
Note that since $\Phi_1^2$ is canonically trivial, this leads to an isomorphism
$$\iota^*\Ber_1^5\simeq p^*((\Ber^{\ov{B}}_1)^5\ot \Phi_1).$$
%Ignoring the punctures, 
We can rewrite \eqref{canonical-class-eq} near $B$ as an isomorphism
$$K_{\ov{\SS}}(\De)\simeq \Ber_1^5.$$
Thus, the restriction to $B$ gives an isomorphism
$$K_B\simeq \iota^*\Ber_1^5.$$
We conjecture that the following diagram is commutative (up to a sign)
\begin{diagram}
\iota^* (K_{\ov{\SS}}(\De))&\rTo{}& \iota^*\Ber_1^5\\
\dTo{}&&\dTo{}\\
K_B&\rTo{}&p^*((\Ber^{\ov{B}}_1)^5\ot \Phi_1)
\end{diagram}

%his should be compatible with the above identification of $K_B$ and with the
%lower genus isomorphism 
%$$K_{\ov{B}}\simeq (\Ber^{\ov{B}}_1)^5.$$

%tangent bundle to $\MM$ is given by
%$$\TT_\MM\simeq R^1\pi_*(\om_{X/\MM}^{-2}(-D_1-\ldots-D_n)).$$
%By Serre duality, the cotangent bundle is
%$$\Om^1_\MM\simeq \pi_*(\om_{X/\MM}^3(D_1+\ldots+D_n)).$$
%Hence, the canonical line bundle of $\MM$ is given by
%$$K_\MM\simeq \BB(\om_{X/\MM}^3(D_1+\ldots+D_n)).$$
%It is easy to see that the right-hand side of \eqref{canonical-class-eq} has a natural extension to
%the moduli $\ov{\MM}$ of stable curves, so it makes sense to ask about the behavior of the isomorphism 
%\eqref{canonical-class-eq} at the boundary (extending $K_\MM$ to $K_{\ov{\MM}}$).
%\end{proof}

%\subsection{Mumford isomorphism with punctures near the boundary}

%Now we consider the analog of Theorem \ref{Mumford-boundary-thm} in the presence of punctures.

\appendix

\section{Relative ampleness criterion}\label{ample-sec}

%\begin{remark}
% superschemes to cover arbitrary coherent sheaves by locally free ones, due 
%\end{remark}

Recall that for every superscheme $X$ let $\NN_X\sub \OO_X$ denote the (nilpotent) ideal locally generated by odd functions,
and denote by $X_{\bos}$ the usual scheme such that
$$\OO_{X_{\bos}}=\OO_X/\NN_X.$$

\begin{lemma}\label{proj-embedding-lem} 
Let $f:X\to S$ be a morphism of superschemes, where $S$ is purely even. Assume that $f$ factors through
a morphism $\phi:X\to \P(\EE^\vee)$, where $\EE=\EE^+\oplus \EE^-$ is a supervector bundle on $S$. 
Then $\phi$ is a closed embedding if and only if $\phi|_{X_{\bos}}$ is a closed embedding (that necessarily factors through $\P((\EE^+)^\vee)$) and the morphism
\begin{equation}\label{E-odd-to-N-map}
f^*\EE^-|_{X_{\bos}}\to \NN_X/\NN_X^2\ot \phi^*\OO(1)|_{X_{\bos}},
\end{equation}
induced by $\phi$, is surjective.
\end{lemma}

\begin{proof} The ``only if" part is clear, so let us prove the ``if" part.
Let us set $\NN=\NN_X$ for brevity.
The underlying map of topological spaces for $\phi$ is the same as $\phi_{X_{\bos}}:X_{\bos}\to \P((\EE^+)^\vee)$, so
we just need to show the surjectivity of the homomorphism of sheaves of rings
$$\phi^\sharp:{\bigwedge}^\bullet_{\OO_{\P((\EE^+)^\vee)}}(p^*\EE^-(-1))=\OO_{\P(\EE^\vee)}\to \phi_*\OO_X,$$
where $p:\P(\EE^\vee)\to S$ is the projection.
Both sides have a natural filtration by powers of an ideal: on $\phi_*\OO_X$ we take the filtration $(\phi_*\NN^i)$, while 
on the exterior algebra the filtration ${\bigwedge}^{\ge i}(\cdot)$.
Since the above map is compatible with the surjective map
$$\OO_{\P((\EE^+)^\vee)}\to \phi_*(\OO_X/\NN)=\phi_*\OO_{X_{\bos}},$$
we see that $\phi^\sharp$ is compatible with the filtrations. Hence, it is enough to check surjectivity on the
consecutive quotients, which follows from the surjectivity of the induced map 
$$p^*\EE^-\to \phi_*(\NN/\NN^2)\ot \OO(1)$$
or equivalently, of \eqref{E-odd-to-N-map}.
\end{proof}

The following criterion generalizes a similar result for supermanifolds in \cite{LPW}.

Let $f_{\bos}:X_{\bos}\to S_{\bos}$ be a morphism of usual schemes induced by $f$.

\begin{prop}\label{ample-crit-prop}
Let $f:X\to S$ be a flat morphism of superschemes. 
If a line bundle $L$ on $X$ is such that $L_{\bos}:=L|_{X_{\bos}}$ is strongly relatively ample over $S_{\bos}$, then $L$ is strongly relatively ample over $S$. 
\end{prop}

\begin{proof}
First, let us show that for every coherent sheaf $\FF$ on $X$, for $n\gg 0$ one has $R^{>0}f_*(\FF\ot L^n)=0$ and
the natural map
\begin{equation}\label{F-Ln-global-generation-map}
f^*f_*(\FF\ot L^n)\to \FF\ot L^n
\end{equation}
is surjective. Let us consider the commutative diagram
\begin{diagram} 
X_{\bos} &\rTo{i}& X\\
\dTo{f_{\bos}}&&\dTo{f}\\
S_{\bos} &\rTo{j}& S
\end{diagram}
We know that the above assertion is true with $(X,f)$ replaced by $(X_{\bos},f_{\bos})$.
Since each $\FF$ has a finite filtration with subsequent quotients that are scheme-theoretically supported on $X_{\bos}$,
we can assume that $\FF=i_*\FF'$ with $\FF'$ a coherent sheaf on $X_{\bos}$. We have
$$R^pf_*(i_*(\FF')\ot L^n)=j_*R^pf_{\bos *}(\FF'\ot L_{\bos}^n),$$
which immediately implies the vanishing of higher direct images. Furthermore,
the restriction of the map \eqref{F-Ln-global-generation-map} to $X_{\bos}$ is the map
$$i^*f^*f_*i_*(\FF'\ot L_{\bos}^n)=f_{\bos}^*f_{\bos *}(\FF'\ot L_{\bos}^n)\to (\FF'\ot L_{\bos}^n)$$
which is surjective. Hence, \eqref{F-Ln-global-generation-map} is also surjective.

%Let $\NN$ be the nilradical in ???

Applying the above statement for $\FF=\OO_X$ we see that for $n\gg 0$,
$R^{>0}f_*(L^n)=0$ and $f^*f_*(L^n)\to L^n$ is surjective. In particular, in this case
$f_*(L^n)$ is locally free (here we use flatness of $f$) and the map
\begin{equation}\label{X-to-P-f-Ln-map}
X\to \P(f_*(L^n)^\vee)
\end{equation}
of superschemes over $S$ is well defined. We want to check that \eqref{X-to-P-f-Ln-map} is a closed embedding. 
Let us consider the cartesian square
\begin{diagram} 
X_0 &\rTo{}& X\\
\dTo{f_0}&&\dTo{f}\\
S_0 &\rTo{j}& S
\end{diagram}
with $S_0:=S_{\bos}$ (however, $X_0\neq X_{\bos}$ in general). For $n\gg 0$ the above assertions also hold for $(X_0,f_0,L^n_0)$, where $L_0=L|_{X_0}$, and
the base change map
$$f_*(L^n)|_{S_0}\to f_{0 *}(L_0^n)$$
is an isomorphism. Thus, the map $X_0\to \P(f_{0*}(L^n)^\vee)$ is obtained by the base change $S_0\to S$ from the map
\eqref{X-to-P-f-Ln-map}. Since $S_0$ is defined by a nilpotent ideal in $S$, it is enough to check that the map
$X_0\to \P(f_{0*}(L^n)^\vee)$ is a closed embedding for $n\gg 0$.

Thus, we can assume that $S=S_0$. Then we have a decomposition of the supervector bundle $\EE:=f_*(L^n)$ into
even and odd components,
$$\EE=\EE^+\oplus \EE^-.$$
First, we claim that for $n\gg 0$, the map $X_{\bos}\to \P((\EE^+)^\vee)$, corresponding to the surjection
$$f^*f_*(L^n)|_{X_{\bos}}\to L^n|_{X_{\bos}},$$
is an embedding. Indeed, for $n\gg 0$, we have a surjection
$$\EE=f_*(L^n)\to f_*(i_*\OO_{X_{\bos}}\ot L^n)=f_{\bos *}(L_{\bos}^n),$$
which factors through a surjection $\EE^+\to f_{\bos *}(L_{\bos}^n)$.
Since $L_{\bos}$ is relatively ample, we get a composition of two closed embeddings
$$X_{\bos}\hra \Proj(S^\bullet(f_{\bos *}(L_{\bos}^n)))\hra \P((\EE^+)^\vee),$$
and our claim follows.

By Lemma \ref{proj-embedding-lem}, it remains to show that the map
$$f^*\EE^-|_{X_{\bos}}\to \NN_X\ot L^n|_{X_{\bos}}$$
is surjective.
%$$f^*f_*(\NN\ot L^n)|_{X_{\bos}}\to (\NN\ot L^n)|_{X_{\bos}}$$
We know that for $n\gg 0$, the map 
$$f^*f_*(\NN_X\ot L^n)\to \NN_X\ot L^n$$ 
is surjective. Hence, its restriction to $X_{\bos}$ is still surjective.
Since $\NN_X$ is generated by odd functions, we have
$$\NN_X|_{X_{\bos}}=(\NN_X|_{X_{\bos}})^-.$$
Hence, we get the surjectivity of the map
$$f^*(f_*(\NN_X\ot L^n)^-)|_{X_{\bos}}\to \NN_X\ot L^n|_{X_{\bos}}.$$
But $f_*(\NN_X\ot L^n)^-$ is a subsheaf in $f_*(L^n)^-=\EE^-$, so we are done.
%Let $f_0:X_0\to S_0=S_{\red}$ be the induced family (so $X_0=X\times_{S_0} S$). 
%First, we claim that it is enough to check that 
\end{proof}

\end{document}